\documentclass[12pt,twoside]{amsart}
\usepackage[english]{babel}
\usepackage[T1]{fontenc}
\usepackage[utf8]{inputenc}
\usepackage{amsmath, amssymb, amsthm}            
\usepackage{amstext, amsfonts, a4}

\usepackage{graphics,graphicx,subfigure}

\usepackage{caption} 
\usepackage{xcolor}

\usepackage{hyperref}

\usepackage{comment}
\usepackage{float} 
\usepackage{lmodern}
\usepackage{xfrac}

\theoremstyle{plain}
	\newtheorem{Thm}{Theorem}[section]
	\newtheorem*{Thm2}{Theorem}
	\newtheorem{Prop}[Thm]{Proposition}        
	\newtheorem{Lem}[Thm]{Lemma}
	\newtheorem{Coro}[Thm]{Corollary}

\theoremstyle{definition}
	\newtheorem{Def}[Thm]{Definition}
        \newtheorem{Exem}[Thm]{Example}
	\newtheorem{Nota}[Thm]{Notation}

\theoremstyle{remark}
	\newtheorem{Rem}[Thm]{Remark}

\def\emptyset{\varnothing}

\def\NN{{\mathbb N}}    
\def\ZZ{{\mathbb Z}}    
\def\RR{{\mathbb R}}	
\def\QQ{{\mathbb Q}}  

\newcommand{\mfA}{\mathfrak{A}}
\newcommand{\mfS}{\mathfrak{S}}
\newcommand{\cE}{\mathcal{E}}
\newcommand{\cL}{\mathcal{L}}
\newcommand{\cP}{\mathcal{P}}
\newcommand{\cJ}{\mathcal{J}}
\newcommand{\cQ}{\mathcal{Q}}
\newcommand{\cM}{\mathcal{M}}

\newcommand{\eps}{\varepsilon}
\newcommand{\IET}{\operatorname{IET}}
\newcommand{\SAF}{\operatorname{SAF}}
\newcommand{\Itv}{\operatorname{Itv}}
\newcommand{\Homeo}{\operatorname{Homeo}}
\newcommand{\support}{\operatorname{Supp}}

\newcommand{\Vect}{\operatorname{Vect}}

\newcommand{\setm}{\smallsetminus}
\newcommand{\Id}{\operatorname{Id}}
\newcommand{\Ker}{\operatorname{Ker}}
\newcommand{\interfo}[2]{\mathopen{[} #1,#2 \mathclose{[}}
\newcommand{\Card}{\operatorname{Card}}
\newcommand{\modulo}[1]{~[\textup{mod}~#1]}
\newcommand{\SGIET}{\Gamma}
\newcommand{\Wedgealt}{{}^\circleddash\!\!\bigwedge^2_{\ZZ}}

\begin{document}

\title{Abelianization of  some groups of interval exchanges}
\author{Octave Lacourte}
\date{\today}

\begin{abstract}
Let IET be the group of bijections from $\mathopen{[}0,1 \mathclose{[}$ to itself that are continuous outside a finite set, right-continuous and piecewise translations. The abelianization homomorphism $f: \textup{IET} \to A$, called SAF-homomorphism, was described by Arnoux-Fathi and Sah. The abelian group $A$ is the second exterior power of the reals over the rationals.

For every subgroup $\Gamma$ of $\mathbb{R/Z}$ we define $\textup{IET}(\Gamma)$ as the subgroup of IET consisting of all elements $f$ such that $f$ is continuous outside $\Gamma$. Let $\tilde{\Gamma}$ be the preimage of $\Gamma$ in $\mathbb{R}$. We establish an isomorphism between the abelianization of $\textup{IET}(\Gamma)$ and the second skew-symmetric power of $\tilde{\Gamma}$ over $\mathbb{Z}$ denoted by ${}^\circleddash\!\!\bigwedge^2_{\mathbb{Z}} \tilde{\Gamma}$. This group often has non-trivial $2$-torsion, which is not detected by the $\textup{SAF}$-homomorphism.

We then define $\textup{IET}^{\bowtie}$ the group of all interval exchange transformations with flips. Arnoux proved that this group is simple thus perfect. However for every subgroup $\textup{IET}^{\bowtie}(\Gamma)$ we establish an isomorphism between its abelianization and $\langle \lbrace a \otimes a ~ [\textup{mod}~2] \mid a \in \tilde{\Gamma} \rbrace \rangle \times \langle \lbrace \ell \wedge \ell \modulo{2} \mid \ell \in \tilde{\Gamma} \rbrace \rangle$ which is a $2$-elementary abelian subgroup of $\bigotimes^2_{\mathbb{Z}} \tilde{\Gamma} / (2\bigotimes^2_{\mathbb{Z}} \tilde{\Gamma}) \times {}^\circleddash\!\!\bigwedge^2_{\mathbb{Z}} \tilde{\Gamma} / (2 {}^\circleddash\!\!\bigwedge^2_{\mathbb{Z}} \tilde{\Gamma})$.

\end{abstract}

\maketitle

\section{Introduction}

Let $X$ be the right-open and left-closed interval $\interfo{0}{1}$. We denote by $\mathfrak{S}(X)$ the group of all permutations of $X$ and let $\mfS_{\mathrm{fin}}$ be its subgroup consisting of all finitely supported elements. For every group $G$ we denote by $D(G)$ its derived subgroup and by $G_{\mathrm{ab}} = G/D(G)$ its abelianization.

We define $\IET$ (Interval Exchange Transformations) as the subgroup of $\mathfrak{S}(X)$ consisting of all elements that are continuous outside a finite set, right-continuous and that are piecewise a translation. For every subgroup $\SGIET$ of $\mathbb{R/Z}$ we define $\IET(\SGIET)$ as the subgroup of $\IET$ consisting of all elements $f$ that are continuous outside $\SGIET$.

We also consider an overgroup of $\IET$ and $\IET(\SGIET)$, where we allow flips, that is, we replace ``translation'' with ``isometry''. There is a difficulty in defining this group, namely that, properly speaking, it cannot be defined as a group acting on $\interfo{0}{1}$. For this we define $\widehat{\IET^{\bowtie}}$ as the subgroup of $\mfS(X)$ consisting of all elements that are continuous outside a finite set and piecewise isometric. Let $\widehat{\IET^{\bowtie}(\SGIET)}$ be its subgroup consisting of all elements $f$ continuous outside $\SGIET$. We define the group of interval exchanges with flips, denoted by $\IET^{\bowtie}$, as the image of $\widehat{\IET^{\bowtie}}$ in $\mfS(X) / \mfS_{\mathrm{fin}}(X)$. Similarly let $\IET^{\bowtie}(\SGIET)$ be the image of $\widehat{\IET^{\bowtie}(\SGIET)}$ in $\mfS(X) / \mfS_{\mathrm{fin}}(X)$.

The group $\IET$ is a subgroup of $\widehat{\IET^{\bowtie}}$. Also it is naturally isomorphic to its image in $\IET^{\bowtie}$ thus we freely see $\IET$ as a subgroup of $\IET^{\bowtie}$.

\medskip

Most of the existing research on interval exchanges since the 70s pertains to classical dynamics, mostly studying the dynamical and ergodic properties of a single element of $\IET$ (see the survey of Viana \cite{Vianasurvey}). The study of IET as a group started with the identification of its abelianization by Sah \cite{sah1981scissors} and Arnoux-Fathi \cite{Arnoux} (see below), and was much later revived by work of Dahmani-Fujiwara-Guirardel \cite{DFG2013,DFG2017} and C. Novak \cite{Novak09} in the late 2000s.

In comparison the group $\IET^{\bowtie}$ is considerably less known. P. Arnoux proved in \cite{ArnouxThese} that it is a simple group. Recently I. Liousse and N. Guelman \cite{LiousseGuelman2019} proved that $\IET^{\bowtie}$ is uniformly perfect.

\medskip

For every $\ZZ$-module $A$, we denote by $\Wedgealt A$ the second skew-symmetric power. It is the quotient of the second tensor power $\bigotimes^2_{\ZZ} A$ by the $\ZZ$-submodule generated by the set $\lbrace x \otimes y + y \otimes x \mid x,y \in A \rbrace$. We denote by $a \wedge b$ the projection of $a \otimes b$ in $\Wedgealt A$; the set of all these elements generates $\Wedgealt A$. There exists a canonical surjective group homomorphism from $\Wedgealt A$ into the second exterior power $\bigwedge^2_{\ZZ} A$, which is injective if $A=2A$. In general, its kernel is a 2-elementary abelian group, which is not trivial in general. For instance, if $A \simeq \ZZ^d$, then $\bigwedge^2_{\ZZ} A$ is isomorphic to $\ZZ^{d(d-1)/2}$, while this kernel is isomorphic to $(\ZZ/2\ZZ)^d$.
In the case where $A=\RR$ we recall that every map is $\ZZ$-bilinear if and only if it is $\QQ$-bilinear, thus $\bigwedge^2_{\QQ} \RR = \bigwedge^2_{\ZZ} \RR$.

We recall the result of Arnoux-Fathi and Sah \cite{Arnoux,sah1981scissors} about the identification of the abelianization of $\IET$:

\begin{Thm2}[Arnoux-Fathi-Sah \cite{Arnoux,sah1981scissors}]

The map:
\[
\begin{array}{cccc}
\varphi: & \IET & \longrightarrow & \bigwedge^2_{\ZZ} \RR \\
 & f & \longmapsto & [\sum\limits_{a \in \RR} a \otimes \lambda((f-\Id)^{-1}(\lbrace a \rbrace)) ]
\end{array}
\]
is a surjective group homomorphism whose kernel is the derived subgroup $D(\IET)$. It is called the $\SAF$-homomorphism. It induces an isomorphism between $\IET_{\mathrm{ab}}$ and the second exterior power of the reals over the rationals.
\end{Thm2}

In this article we identify $\IET(\SGIET)_{\mathrm{ab}}$ and $\IET^{\bowtie}(\SGIET)_{\mathrm{ab}}$ for every subgroup $\SGIET$ of $\mathbb{R/Z}$. We denote by $\tilde{\SGIET}$ the preimage of $\SGIET$ in $\RR$.

When $\SGIET$ if finite, both $\IET(\SGIET)$ and $\IET^{\bowtie}(\SGIET)$ are finite Coxeter groups and their abelianization can be described easily (see Section \ref{Subsection When SGIET is finite}). So from now on, we always assume that $\SGIET$ is a dense subgroup of $\mathbb{R/Z}$.

In this case, we show that the derived subgroups $D(\IET(\SGIET))$ and $D(\IET^{\bowtie}(\SGIET))$ are simple (see Section \ref{Subsection About the derived subgroup}). They are respectively the smallest normal subgroup of $\IET(\SGIET)$ and $\IET^{\bowtie}(\SGIET)$. Then we see $\IET(\SGIET)$ and $\IET^{\bowtie}(\SGIET)$ as topological full group of the groupoids of germs of their action on a Stone space. Then by the work of Nekrashevych \cite{nekrashevych_2019}, we obtain the simplicity of their smallest normal subgroup.

\medskip

We denote by $\varphi_{\SGIET}$ the restriction of the $\SAF$-homomorphism $\varphi$ to $\IET(\SGIET)$ for every subgroup $\SGIET$ of $\mathbb{R/Z}$. It turns out that $\varphi_{\SGIET}$ may fail to give the abelianization of $\IET(\SGIET)$ because not every element of order $2$ of $\IET(\SGIET)$ is necessarily in $D(\IET(\SGIET))$. In fact the restriction of $\varphi_{\SGIET}$ is a surjective group homomorphism onto $2\Wedgealt \tilde{\SGIET}$ and here the main theorem is:

\begin{Thm}
There exists a surjective group homomorphism $\eps_{\SGIET}: \IET(\SGIET) \rightarrow \Wedgealt \tilde{\SGIET}$ whose kernel is the derived subgroup $D(\IET(\SGIET))$. It induces an isomorphism between $\IET(\SGIET)_{\mathrm{ab}}$ and the second skew-symmetric power $\Wedgealt \tilde{\SGIET}$.
\end{Thm}

\begin{Rem}
One can notice that if $2 \tilde{\SGIET}= \tilde{\SGIET}$ then $\eps_{\SGIET}=\varphi_{\SGIET}$. Actually, in this case, the proof of Arnoux-Fathi-Sah works with immediate changes.
\end{Rem}

The novelty occurs when $\tilde{\SGIET} \neq 2 \tilde{\SGIET}$. In this case $2$-torsion appears in $\Wedgealt \tilde{\SGIET}$ with in particular the subset $\lbrace a \wedge a \mid a \in \tilde{\SGIET} \rbrace$.

Before constructing $\eps_{\SGIET}$ it is important to understand the kernel of $\varphi_{\SGIET}$. For this we introduce some elements of $\IET$ and $\IET(\SGIET)$.

A \textit{restricted rotation of type $(a,b)$} is an element $r$ of $\IET$ such that there exist consecutive intervals $I$ and $J$ of length $a$ and $b$ respectively with $\sup(I) = \inf(J)$, where $r$ is the translation by $b$ on $I$ and the translation by $-a$ on $J$ (and the identity outside $I \cup J$). The two intervals $I$ and $J$ are called the intervals associated with $r$. A \textit{$\SGIET$-restricted rotation} is a restricted rotation in $\IET(\SGIET)$. A tuple of restricted rotations is \textit{balanced} if the number of factors of type $(a,b)$ is equal to the one of type $(b,a)$. A product of restricted rotations is balanced if it can be written as the product of a balanced tuple of restricted rotations.

\begin{Rem}
Let $r$ be a restricted rotation of type $(a,b)$ and $I$ and $J$ be the two consecutive intervals associated with $r$. Then if we only look at the interval $I \cup J$ and identify its endpoints, we obtain that $r$ is the actual rotation of angle the length of $J$. This is why we called these elements ``restricted rotation''.
\end{Rem}

Y. Vorobets proves in \cite{Vorobets2011} the following:

\begin{Thm2}[Y. Vorobets \cite{Vorobets2011}]
The kernel $\Ker(\varphi)$ of the $\SAF$-homomorphism is generated by the set of all balanced products of restricted rotations and it is also generated by the set of all elements of order $2$. 
\end{Thm2}

Here we prove that this result also holds in restriction to $\IET(\SGIET)$:

\begin{Lem}\label{Lemma Ker restriction SAF is generated by balanced product of IET Gamma restricted rotations}
The set of all balanced products of $\SGIET$-restricted rotations is a generating subset of $\Ker(\varphi_{\SGIET})$.
\end{Lem}

A corollary can be proved by following the ideas of Y.Vorobets. We introduce some elements of order $2$. An \textit{$\IET$-transposition} of type $a$ is an element of $\IET$ that swaps two nonoverlapping intervals of length $a$ while fixing the rest of $\interfo{0}{1}$ and an \textit{$\IET(\SGIET)$-transposition} is an $\IET$-transposition in $\IET(\Gamma)$.

\begin{Coro}
The set of all $\IET(\SGIET)$-transpositions is a generating subset of the kernel $\Ker(\varphi_{\SGIET})$.
\end{Coro}

After this, in Section \ref{Section Description of the abelianization of IET(Gamma)}, we define the group homomorphism $\eps_{\SGIET}$ in the same spirit as the signature on finite permutation groups. For every $n \in \NN$ and $\sigma$ in the finite permutation group $\mfS_n$, a pair $(x,y) \in \lbrace 1,2,\ldots , n \rbrace$ is an inversion of $\sigma$ if $x<y$ and $\sigma(x)>\sigma(y)$. The signature of $\sigma$ is equal to the counting measure of the set of all inversions of $\sigma$.

We keep the same definition for an inversion of an element of $\IET(\SGIET)$ and explicit how to measure this set. We define $A_{\SGIET}$ as the Boolean algebra of subsets of $\interfo{0}{1}$ generated by the set of intervals $\lbrace \interfo{a}{b} \mid a,b \in \tilde{\SGIET} \rbrace$. Thanks to the Lebesgue measure $\lambda$ on $\RR$ we obtain a Boolean algebra measure $w_{\SGIET}$ for $A_{\SGIET} \otimes A_{\SGIET}$ in the second tensor power $\bigotimes^2_{\ZZ} \tilde{\SGIET}$. We show that the set of inversions of an element of $\IET(\SGIET)$ is an union of rectangles inside $A_{\SGIET} \otimes A_{\SGIET}$. We define $\eps_{\SGIET}$ as the projection on $\Wedgealt \tilde{\SGIET}$ of the measure of the set of all inversions. We show that $\eps_{\SGIET}$ is a surjective group homomorphism. We establish the equality $-2\eps_{\SGIET}=\varphi_{\SGIET}$ and manage to prove that $\Ker(\eps_{\SGIET})$ is equal to the derived subgroup $D(\IET(\SGIET))$.

\medskip

In Section \ref{Section Abelianization of IET ^ bowtie (Gamma)} we treat the case of $\IET^{\bowtie}(\SGIET)$. \textit{A reflection map of type $a$} is an element of $\widehat{\IET^{\bowtie}}$ that reverses a right-open and left-closed subinterval of $\interfo{0}{1}$ of length $a$. A \textit{reflection of type $a$} is the image of a reflection map of type $a$ in $\IET^{\bowtie}$. A $\SGIET$-reflection is a reflection in $\IET^{\bowtie}(\SGIET)$. We know that the group $\IET^{\bowtie}(\SGIET)$ is generated by the set of reflections (see Proposition \ref{Proposition Generating set for IET bowtie SGIET}) so its abelianization is a $2$-group. We are no longer able to measure the set of inversions but we can measure the union of this set with its symmetric. By projecting on $\bigotimes^2_{\ZZ} \tilde{\SGIET} / 2(\bigotimes^2_{\ZZ} \tilde{\SGIET})$ we obtain a group homomorphism $\eps_{\SGIET}^{\bowtie}$ whose image is $\langle \lbrace a \otimes a ~ [\textup{mod}~2] \mid a \in \tilde{\SGIET} \rbrace \rangle$. We also prove that its kernel is generated by all reflections of type inside $\tilde{\SGIET} \setm (2 \tilde{\SGIET})$. 
Thus for $a \in \tilde{\SGIET} \smallsetminus 2 \tilde{\SGIET}$ every restricted rotation of type $(a,a)$ is in the kernel and by the work on $\IET(\SGIET)$ we do not expect it to be in $D(\IET^{\bowtie}(\SGIET))$. We notice that such an element is conjugate to a reflections of type $2a$. In fact we prove :

\begin{Prop}
Let $\Omega_{\SGIET}$ be the conjugate closure of the group generated by the set of all $\SGIET$-reflections of type $2\ell$ with $\ell \in \tilde{\SGIET} \setm 2 \tilde{\SGIET}$. Then :
$$
\Ker(\eps_{\SGIET}^{\bowtie}) = D(\IET^{\bowtie}(\SGIET)) \Omega_{\SGIET}
$$
\end{Prop}

Then we construct a group homomorphism $\psi_{\SGIET} : \IET^{\bowtie}(\SGIET) \to \Wedgealt \tilde{\SGIET} / 2(\Wedgealt \tilde{\SGIET})$ which uses $\eps_{\SGIET}$, called the positive contribution of $\IET^{\bowtie}(\SGIET)$. Morally we approximate every element $f \in \IET^{\bowtie}(\SGIET)$ by elements in $\IET(\SGIET)$ and we use the group homomorphism $\eps_{\SGIET}$ on them. This can also be seen as approximate triangles by rectangles in the set of inversions. The construction is strongly dependent on a totally ordered set given by the fact that $\tilde{\SGIET}$ is ultrasimplicially ordered (see Section \ref{Section a fact about ultrasimplicial group}). For every $a \in \tilde{\SGIET} \smallsetminus 2 \tilde{\SGIET}$, this group homomorphism sends restricted rotation of type $(a,a)$ on a nontrivial element. We also prove that the image of its restriction to $\Ker(\eps_{\SGIET}^{\bowtie})$ is  $\langle \lbrace \ell \wedge \ell \modulo{2} \mid \ell \in \tilde{\SGIET} \rbrace \rangle$.

In order to obtain the following theorem we will need to notice that groups of exponent $2$ are also $\mathbb{F}_2$-vector spaces.

\begin{Thm}
For every subgroup $\SGIET$ of $\mathbb{R/Z}$ we have the following isomorphism:
\[
 \IET^{\bowtie}(\SGIET)_{\mathrm{ab}} \simeq \langle \lbrace a \otimes a ~ [\textup{mod}~2] \mid a \in \tilde{\SGIET} \rbrace \rangle \times \langle \lbrace \ell \wedge \ell \modulo{2} \mid \ell \in \tilde{\SGIET} \rbrace \rangle,
\]
where the left term of the product is in $\bigotimes^2_{\ZZ} \tilde{\SGIET} / (2\bigotimes^2_{\ZZ} \tilde{\SGIET})$ and the right one is in 
$\Wedgealt \tilde{\SGIET} / (2\Wedgealt \tilde{\SGIET})$.
\end{Thm}

The inclusion of $\IET(\SGIET)$ in $\IET^{\bowtie}(\SGIET)$ induces a natural group homomorphism $\iota$ from $\IET(\SGIET)_{\mathrm{ab}}/ (2\IET(\SGIET)_{\mathrm{ab}})$ to $\IET^{\bowtie}(\SGIET)_{\mathrm{ab}}$. We show that $\iota$ is injective for every dense subgroup $\SGIET$ of $\RR / \ZZ$. The image of $\iota$ is $\langle \lbrace p \otimes q + q \otimes p \modulo{2} \mid p,q \in \tilde{\SGIET} \rbrace \rangle \times \langle \lbrace l \wedge l \modulo{2} \mid l \in \tilde{\SGIET} \rbrace \rangle$; it is also isomorphic to $\Wedgealt \tilde{\SGIET} /(2\Wedgealt \tilde{\SGIET})$. We deduce that $\iota$ is surjective if and only if $\SGIET=2\SGIET$. In the case where $\SGIET$ if finitely generated we can precise the dimension of $\IET^{\bowtie}(\SGIET)_{\mathrm{ab}}$ and $\textup{Im}(\iota)$ as $\mathbb{F}_2$-vector spaces. The group $\tilde{\SGIET}$ is also finitely generated and we denote by $d$ its rank. Then $\IET^{\bowtie}(\SGIET)_{\mathrm{ab}}$ has dimension $\frac{d(d+3)}{2}$ and $\textup{Im}(\iota)$ has dimension $\frac{d(d+1)}{2}$.

\medskip

\noindent \textbf{Acknowledgments.} I would like to thank Friedrich Wehrung for communicating me references about ultrasimplicial groups and in particular the work of G.A. Elliott in \cite{Elliott}. I thank also Laurent Bartholdi, Mikael de la Salle and Nicolás Matte Bon for pertinent comments. I wish to express my gratitude to my advisor Yves Cornulier for his careful rereading and support.

\section{Preliminaries}

The aim of this section is to introduce some notation and gather known results used in the next sections. Let $\SGIET$ be a subgroup of $\mathbb{R/Z}$ and let $\tilde{\SGIET}$ be its preimage in $\RR$.

\bigskip

We recall that $\mfS_n$ is the permutation group of $n$ elements. For every group $G$ we denote by $G_{\mathrm{ab}}$ the abelianization of $G$, which is the quotient of $G$ by its derived subgroup.

For every function $f$ we define the \textit{support of $f$} by $\support(f):=\lbrace x \in \interfo{0}{1} : f(x) \neq x \rbrace$.

For any real interval $I$ let $I^{\circ}$ be its interior in $\RR$. If $I$ is equal to $\mathopen{[}0,t \mathclose{[}$ we agree that its interior $I^{\circ}$ is equal to $\mathopen{]}0,t \mathclose{[}$. We denote by $\Itv(\SGIET)$ the set of all intervals $\interfo{a}{b}$ with $a,b \in \tilde{\SGIET}$ and $0 \leq a <b \leq 1$. We denote by $\lambda$ the Lebesgue measure on $\RR$.

\subsection{\texorpdfstring{When $\Gamma$ is finite}{When Gamma is finite}}\label{Subsection When SGIET is finite}~

\smallskip

We assume that $\SGIET$ is a finite subgroup of $\mathbb{R/Z}$. Then there exists $n \in \NN_{\geq 1}$ such that $\tilde{\SGIET}$ is equal to $\frac{1}{n} \ZZ$. We deduce that for every $1 \leq i \leq n$, every element of $\widehat{\IET^{\bowtie}(\SGIET)}$ is continuous on the interval $\interfo{\frac{i-1}{n}}{ \frac{i}{n}}$ up to a finitely supported permutation.

Then the group $\IET(\SGIET)$ is naturally isomorphic to the finite permutation group $\mfS_n$. It is a Coxeter group of type $A_{n-1}$ so its abelianization is $\lbrace 1 \rbrace$ if $n=1$ and it is $\ZZ / 2 \ZZ$ if $n>1$. 

The group $\IET^{\bowtie}(\SGIET)$ is isomorphic to the signed symmetric group $ \mathbb{Z}/2\mathbb{Z} \wr \mfS_n$. It is a Coxeter group of type $B_n$ so its abelianization is $\ZZ / 2 \ZZ$ for $n=1$ and it is $(\ZZ / 2 \ZZ)^2$ if $n>1$.

\subsection{Descriptions of an element of IET}~

\smallskip

In order to describe an element $f$ of $\IET$ we use partitions into right-open and left-closed intervals of $\interfo{0}{1}$. Let $\cP$ be such a partition.

\begin{Def}\label{Definition partition associated}
If $f$ is continuous on the interior of every interval of $\cP$ then $\cP$ is called \textit{a partition into intervals associated with $f$}. The set of all partitions into intervals associated with $f$ is denoted by $\Pi_f$.
If every interval $I$ of $\cP$ is in $\Itv(\SGIET)$ then $\cP$ is said to be a \textit{$\SGIET$-partition}.
We denote by $f(\cP)$ the partition into intervals of $\interfo{0}{1}$ composed of all right-open and left-closed intervals whose interior is the image by $f$ of the interior of an interval in $\cP$. It is called \textit{the arrival partition of $f$ associated with $\cP$}.
\end{Def}

\begin{Rem}
For every $f \in \IET$, the set $\Pi_f$ has a minimal element for the refinement. This element is also the unique partition that has a minimal number of interval.
\end{Rem}

One can notice that for every $f \in \IET$ there is an equivalence between $f$ belongs to $\IET(\SGIET)$ and the existence of a $\SGIET$-partition inside $\Pi_f$. From now on every partition is assumed to be finite.

There are descriptions of an element of $\IET$ that are more combinatorial. They are useful to understand the $\SAF$-homomorphism and further $\eps_{\SGIET}$:

\begin{Def}
Let $f$ be an element of $\IET$ and let $\cP \in \Pi_f$; let $n$ be the number of intervals of $\cP$. For $1 \leq i \leq n$ let $I_i$ (resp.\ $J_i$) be the consecutive intervals composing $\cP$ (resp.\ $f(\cP)$). Let $\alpha$ be the $n$-tuple such that $\alpha_i$ is the length of $I_i$ for every $1 \leq i \leq n$. There exists $\sigma \in \mfS_n$ such that $f(I_i^{\circ})=J_{\sigma(i)}^{\circ}$. Then we say that $(\alpha,\tau)$ is a \textit{combinatorial description of $f$}. If each component of $\alpha$ is in $\tilde{\SGIET}$ we say that $(\alpha,\tau)$ is a \textit{$\SGIET$-combinatorial description of $f$}. The partition $\cP$ can also be called the \textit{partition associated with $(\alpha,\tau)$}.
\end{Def}

\begin{Prop}\label{Proposition SAF invariant with combinatorial description}
Let $f \in \IET$ and let $(\alpha,\tau)$ be a combinatorial description of $f$; let $n$ be the length of $\alpha$. Then we have:
$$
\varphi(f)=\sum\limits_{j=1}^{n} \big( \sum\limits_{\substack{i: \\ \tau(i) < \tau(j)}} \alpha_i - \sum\limits_{i<j} \alpha_i\big) \wedge \alpha_j
$$
\end{Prop}

\begin{proof}
Let $f \in \IET$ and $(\alpha,\tau)$ be a combinatorial description of $f$. Let $n$ be the length of $\alpha$ and let $\lbrace I_1,I_2, \ldots ,I_n \rbrace$ be the partition associated with $(\alpha,\tau)$. For each $j$ we denote by $v(j)$ the value of $f-\Id$ on $I_j$. Thus we deduce that $\varphi(f)=\sum\limits_{j=1}^{n} v(j) \wedge \lambda(I_j)=\sum\limits_{j=1}^{n} v(j) \wedge \alpha_j$. Also we know that $v(j)=\sum\limits_{\substack{i \\ \tau(i) < \tau(j)}} \alpha_i - \sum\limits_{i<j} \alpha_i$ (see Figure \ref{figure valeur de translation} below), and this gives the conclusion.
\end{proof}

\begin{figure}[ht]
\includegraphics[width=0.4\linewidth]{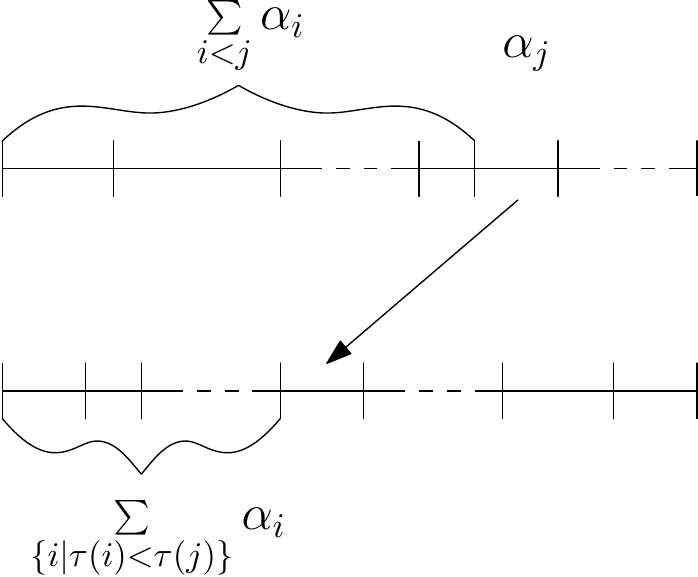}
\caption{\footnotesize{Illustration for the value of the $\SAF$-homomorphism in Proposition \ref{Proposition SAF invariant with combinatorial description}}}
\label{figure valeur de translation}
\end{figure}

\subsection{Three important families}~

\smallskip

There are two families of elements of $\IET$ appearing in this work and one more for $\IET^{\bowtie}$.

\begin{Def}\label{Definition IET transposition}
Let $a \in \interfo{0}{\frac{1}{2}}$ and $f$ be in $\IET$. We call $f$ an \textit{$\IET$-transposition of type $a$} if it swaps two subintervals of $\interfo{0}{1}$ of length $a$ with non-overlapping support while fixing the rest of $\interfo{0}{1}$. If in addition $f$ belongs to $\IET(\SGIET)$ then $f$ is called an \textit{$\IET(\SGIET)$-transposition of type $a$}.
\end{Def}

Also it is natural to consider the identity as an $\IET(\SGIET)$-transposition. We precise it whenever it is needed.

One can notice that, for $f$ an $\IET$-transposition of type $a$, asking $f$ to be in $\IET(\SGIET)$ implies that $a$ is in $\tilde{\SGIET} \cap \interfo{0}{1}$ but the converse is false.

These elements are called \textit{interval swap maps} by Y. Vorobets in \cite{Vorobets2011}. They are elements of order $2$ and generate $D(\IET)$. More precisely, we describe below (Figure \ref{Figure IET transposition is inside D(IET)}) how such an element can be seen in $D(\IET)$. For $f$ an $\IET(\SGIET)$-transposition of type $a$, this suggests to ask if $a$ belongs to $2 \tilde{\SGIET}$ or not, before concluding that $f$ is in $D(\IET(\SGIET))$ or not.

\begin{figure}[ht]
\includegraphics[width=\linewidth]{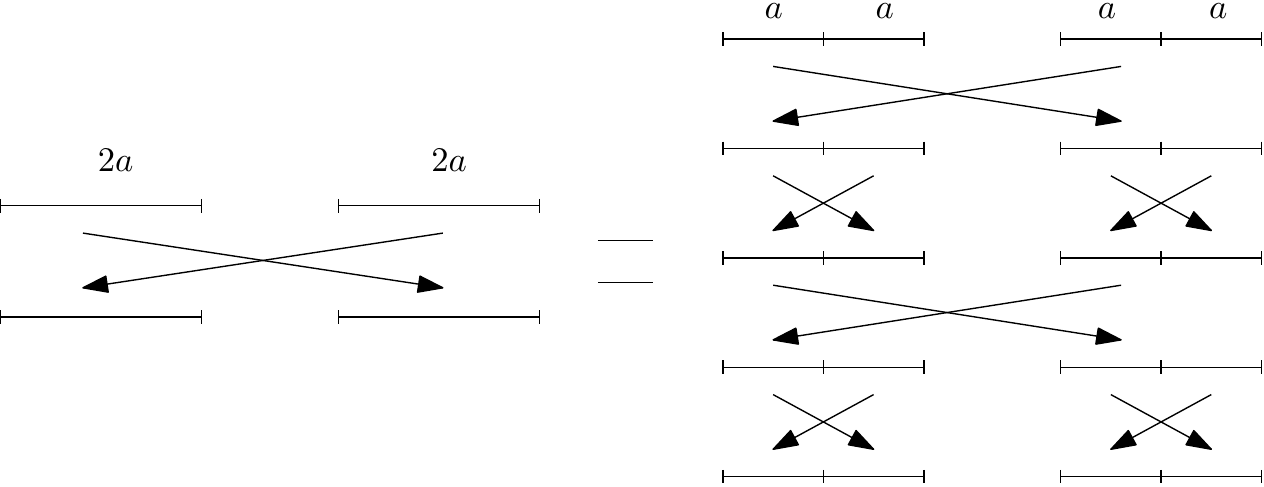}
\caption{\footnotesize{Illustration of the property that an $\IET$-transposition of type $2a$ is a commutator of two $\IET$-transpositions of type $a$.}}
\label{Figure IET transposition is inside D(IET)}
\end{figure}

A \textit{restricted rotation of type $(a,b)$} is an element $r$ of $\IET$ such that there exist consecutive intervals $I$ and $J$ of length $a$ and $b$ respectively with $\sup(I) = \inf(J)$, where $r$ is the translation by $b$ on $I$ and the translation by $-a$ on $J$ (and the identity outside $I \cup J$). The two intervals $I$ and $J$ are called the intervals associated with $r$.

\begin{Def}
Let $a,b \in \interfo{0}{1}$ with $0 \leq a+b \leq 1$ and $f$ be in $\IET$. We call $f$ an \textit{$\IET$-restricted rotation of type $(a,b)$} if there exist consecutive intervals $I$ and $J$ of length $a$ and $b$ respectively with $\sup(I) = \inf(J)$, where $f$ is the translation by $b$ on $I$ and the translation by $-a$ on $J$ (and the identity outside $I \cup J$).The two intervals $I$ and $J$ are called the intervals associated with $f$. If in addition $f$ belongs to $\IET(\SGIET)$ then $f$ is called an \textit{$\SGIET$-restricted rotation of type $(a,b)$}.
\end{Def}

Again one can notice that for $f$ an $\IET$-restricted rotation of type $(a,b)$, if $f$ is in $\IET(\SGIET)$ this implies that $a$ and $b$ are in $\tilde{\SGIET} \cap \interfo{0}{1}$, but the converse is false.

The case $a=b$ coincide with an $\IET$-transposition of type $a$.

\medskip

The family of restricted rotations generates $\IET$ and we have more precision on the type of restricted rotations in the decomposition, by the following fact due to Vorobets in \cite{Vorobets2011}:

\begin{Lem}[Vorobets]\label{Lemma Decomposition in IET(Gamma)-restricted rotations}
Any transformation $f$ in $\IET$ is a product of $\IET$-restricted rotations. Furthermore for every $\cP \in \Pi_f$ there exists a decomposition of $f$ into $\IET$-restricted rotations with type inside the set of length of the intervals in $\cP$.
\end{Lem}

In particular if we start with an element of $\IET(\SGIET)$ we can find a decomposition into $\SGIET$-restricted rotations that satisfies the condition about the type of restricted rotations.

The group $\IET$ can be seen as a subgroup of $\IET^{\bowtie}$; thus the notion of restricted rotation still makes sense in $\IET^{\bowtie}$. We now add a family that occurs only in $\IET^{\bowtie}$.

\begin{Def}\label{Definition reflection map and reflection}
For every right-open and left-closed subinterval $I$ of $\interfo{0}{1}$ we define the \textit{$I$-reflection map} as the element of $\widehat{\IET^{\bowtie}}$ that reverses $I^{\circ}$ while fixing the rest of $\interfo{0}{1}$. The type of an $I$-reflection map is the length of $I$. We define the \textit{$I$-reflection} as the image of the $I$-reflection map in $\IET^{\bowtie}$. The type of an $I$-reflection is the length of $I$. A \textit{$\SGIET$-reflection} is an $I$-reflection for some $I \in \Itv(\SGIET)$.
\end{Def}

The set of all reflections is a generating subset of $\IET^{\bowtie}$. This follows from Proposition \ref{Proposition decomposition in widehat IET bowtie} and the fact that the restricted rotation, whose intervals associated are $I$ and $J$, is the product of the $I$-reflection, the $J$-reflection and the $(I \sqcup J)$-reflection. More precisely, with this proposition and Lemma \ref{Lemma Decomposition in IET(Gamma)-restricted rotations} we obtain the following:

\begin{Prop}\label{Proposition Generating set for IET bowtie SGIET}
The set of all $\SGIET$-reflections is a generating subset of $\IET^{\bowtie}(\SGIET)$.
\end{Prop}

We make explicit some elements of the conjugacy class of a reflection.

\begin{Prop}\label{Proposition reflections of same type are conjugated}
Two reflections $r,s$ which have the same type are conjugate in $\IET^{\bowtie}(\SGIET)$.
\end{Prop}

\begin{Prop}\label{Proposition reflections of type 2a are conjugated to restricted rotation of type (a,a)}
Let $a \in \interfo{0}{1}$. A reflection of type $2a$ is conjugate to a restricted rotation of type $(a,a)$.
\end{Prop}

\begin{proof}
Let $r$ be such a reflection and let $I$ be the interval reversed by $r$. By cutting $I$ into two intervals of same size $I_1$ and $I_2$ we obtain that $r$ is conjugate to the restricted rotation that permutes $I_1$ and $I_2$, by the $I_1$-reflection.	
\end{proof}

\subsection{Ultrasimplicial groups}\label{Section a fact about ultrasimplicial group}~

\smallskip

Let us introduce some classical terminology from the theory of ordered abelian groups. An ordered abelian group is an abelian group endowed with an invariant partial ordering.

For any subgroup $H$ of $\RR$ we denote by $H_{+}:= \lbrace x \in H \mid x \geq 0 \rbrace$ the positive cone of $H$. A difficulty is that $H_+$ is not, in general, finitely generated as a subsemigroup. For every subset $B$ of $\RR$ we denote by $\Vect_{\NN}(B)$ the subsemigroup generated by $B$.

A subsemigroup of an abelian group is \textit{simplicial} if it is generated, as a subsemigroup, by a finite $\ZZ$-independent subset, and it is \textit{ultrasimplicial} if it is the filtered union of simplicial subsemigroups. An ordered abelian group is \textit{simplicially ordered} if its positive cone is simplicial and is \textit{ultrasimplicially ordered} if its positive cone is ultrasimplicial.
The next theorem is proved by G.A. Elliott in \cite{Elliott}:

\begin{Thm}\label{Theorem ultrasimplicially ordered group}
Every totally ordered abelian group is ultrasimplicially ordered.
\end{Thm}

We deduce the following corollary for the totally ordered abelian group $\tilde{\SGIET}$:

\begin{Coro}\label{Corollary the sequence of finite Z indepndent subset for tilde SGIET}
There exists a sequence $(S_n)_n$ of finite $\ZZ$-independent subset of $\tilde{\SGIET}$ such that for each $n$ we have $\Vect_{\NN}(S_n) \subset \Vect_{\NN}(S_{n+1})$ and $\tilde{\SGIET}_+$ is equal to the increasing union $\bigcup_n \Vect_{\NN}(S_n)$.
\end{Coro}

We will use the Theorem \ref{Theorem ultrasimplicially ordered group} in the form of the following corollary which specifies the finite rank case: 

\begin{Coro}\label{Corollary ultrasimplicially ordered group}
Let $H$ be a subgroup of $\RR$ which is free abelian of finite rank $d$. Then there exists a sequence $(B_n)_{n \in \NN}$ of $\ZZ$-basis of $H$ such that for each $n$ we have $\Vect_{\NN}(B_n) \subset \Vect_{\NN}(B_{n+1})$ and $H_+$ is equal to the increasing union $\bigcup_n \Vect_{\NN}(B_n)$.
Furthermore for every $k \in \NN$ and $L_1,L_2, \ldots ,L_k \in H_{+}$ there exists a basis $\lbrace \ell_1,\ell_2,\ldots, \ell_d \rbrace \subset H_{+}$ of $H$ such that for every $1 \leq i \leq k$ the element $L_i$ is a linear combination of $\ell_1,\ell_2,\ldots ,\ell_d$ with coefficients in $\NN$.
\end{Coro}

\subsection{About the derived subgroup}\label{Subsection About the derived subgroup}~

We prove here that the derived subgroups $D(\IET(\SGIET))$ and $D(\IET^{\bowtie}(\SGIET))$ are simple. There are two steps in the proof. The first one is to give a sufficient condition for a group $G$ to obtain that its derived subgroup $D(G)$ is contained in every nontrivial normal subgroup of $G$; this is used by Arnoux in \cite{ArnouxThese} to show that $\IET^{\bowtie}$ is simple. 

The second step is to remark that $\IET(\SGIET)$ and $\IET^{\bowtie}(\SGIET)$ are subgroups of the group of homeomorphism of a Stone space $Y$; that is a compact totally disconnected space. We show that they satsify some conditions that allows us to use the work of Nekrashevych \cite{nekrashevych_2019} which gives us a normal subgroup of $\IET(\SGIET)$ and one of $\IET^{\bowtie}(\SGIET)$ which are simple and contained in all normal subgroups.

\begin{Def}
For every $\eps>0$ we denote by $\mathcal{F}_{\eps}$ the union of the set of all $\SGIET$-restricted rotations of type $(a,b)$ with $a+b \leq \eps$, with the set of all $\IET(\SGIET)$-transpositions of type $a$ with $2a \leq \eps$.
We denote by $\mathcal{F}_{\eps}^{\bowtie}$ the union of the set $\mathcal{F}_{\eps}$ with the set of all $\SGIET$-reflections of type $a$ with $a \leq \eps$.
\end{Def}

\begin{Lem}\label{Lemma for every epsilon any element is a product of elements of support of measure smaller than epsilon}
For every $\eps>0$, any element $f$ in $\IET(\SGIET)$ (resp.\ $\IET^{\bowtie}(\SGIET)$) can be written as a finite product of element in $\mathcal{F}_{\eps}$ (resp.\ $\mathcal{F}^{\bowtie}_{\eps}$).
\end{Lem}

\begin{proof}
Let $\eps>0$ and by density of $\tilde{\SGIET}$ in $\RR$, let $w \in \tilde{\SGIET}_+$ such that $w < \eps$.

If $f$ is an $\IET(\SGIET)$-transposition of type $a$ let $I$ and $J$ be the two non-overlapping intervals swapped by $f$. There exist $k \in \NN$ and $u \in \tilde{\SGIET}_+$ such that $u<w$ and $a=kw+u$. Then $I$ and $J$ can be cut into $k+1$ consecutive intervals $I_1,I_2, \ldots ,I_{k+1}$ and $J_1,J_2, \ldots ,J_{k+1}$ respectively; the first $k$ intervals with length $w$ and the last one with length $u$. Then $f$ is equal to the product $s_1s_2 \ldots s_{k+1}$ where $s_i$ is the $\IET(\SGIET)$-transposition that swaps $I_i$ with $J_i$.

If $f$ is an $I$-reflection of type $a$ with $a=kw+u$ as in the previous case. Then we consider $\lbrace I_1,I_2,\ldots,I_{k+1} \rbrace $ and $\lbrace J_1,J_2,\ldots,J_{k+1} \rbrace$ two partitions into consecutive intervals of $I$ such that the length of $I_i$ and $J_j$ is $w$ for every $1 \leq i \leq k$ and $2 \leq j \leq k+1$ and the length of $I_{k+1}$ and $J_1$ is $u$. Then $f$ is equal to the product $r_1r_2\ldots r_{k+1} s_1s_2 \ldots s_{k+1}$ where $r_i$ is the $I_i$-reflection and $s_i$ is the $\IET(\SGIET)$-transposition that swaps $I_i$ with $J_{k+1-i}$.  

If $f$ is a $\SGIET$-restricted rotation of type $(a,b)$. We can assume $a\geq b$ as the other case is similar. Then $f$ is the product of an $\IET(\SGIET)$-transposition of type $b$ with a restricted rotation of type $(a-b,b)$ (see Figure \ref{figure restricted rotation decomposition}). Then by iterating this operation and thanks to the first case we deduce the result.

\begin{figure}[!h]
\includegraphics[width=\linewidth]{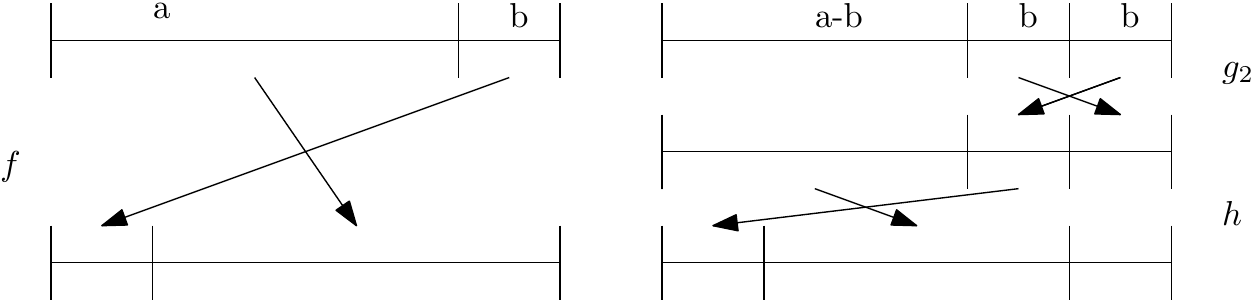}
\caption{Illustration of the decomposition of a restricted rotation in Lemma \ref{Lemma for every epsilon any element is a product of elements of support of measure smaller than epsilon}.}\label{figure restricted rotation decomposition}
\end{figure}

By Lemma \ref{Lemma Decomposition in IET(Gamma)-restricted rotations} and Proposition \ref{Proposition Generating set for IET bowtie SGIET} we obtain the result for any $f$.
\end{proof}

\begin{Prop}\label{Proposition the derived subgroup is included in any normal subgroups}
The derived subgroup $D(\IET(\SGIET))$ (resp.\ $D(\IET^{\bowtie}(\SGIET))$) is contained in every nontrivial normal subgroup of $\IET(\SGIET)$ (resp.\ $\IET^{\bowtie}(\SGIET)$).
\end{Prop}

\begin{proof}
We only do the case of $\IET(\SGIET)$. In $\IET^{\bowtie}(\SGIET)$ one should be careful about having equalities up to finitely supported permutations, but the proof is similar. For every $\eps >0$, we recall that $\mathcal{F}_{\eps}$ is a generating subset of $\IET(\SGIET)$. Then the set of all commutators of two elements in $\mathcal{F}_{\eps}$ is a generating subset of $D(\IET(\SGIET))$.

Let $N$ be a normal subgroup of $\IET(\SGIET)$ and $f \in N \smallsetminus \lbrace \Id \rbrace$. Then there exists a subinterval $I$ of $\interfo{0}{1}$ such that $f$ is continuous on $I$ and $f(I) \cap I= \emptyset$. Let $\delta$ be the length of $I$ and let $\eps = \frac{\delta}{3}$. Let $g_1,g_2 \in \mathcal{F}_{\eps}$ and let $A:= \support(g_1) \cup \support(g_2)$. There exists $h \in \IET(\SGIET)$ such that $h(A) \subset I$ because the measure of $A$ is at most $2\eps$ which is less than $\delta$. Then the element $f':=h^{-1} \circ f \circ h$ is an element of $N$ and $f'(A) \cap A= \emptyset$. This implies that the support of $g_1$ and the one of $f'g_2f'^{-1}$ do not overlap; thus these elements commute and we obtain that $[g_1,g_2]=[g_1,[g_2,f']]$. Also $N$ is normal in $\IET(\SGIET)$, hence as $f' \in N$ we deduce that $[g_2,f'] \in N$, thus $[g_1,g_2] \in N$. We deduce that every commutator of elements in $\mathcal{F}_{\eps}$ is in $N$ so $D(\IET(\SGIET))$ is a subgroup of $N$.
\end{proof}

\begin{Def}
Let $X$ be a Stone space and $G$ be a subgroup of $\Homeo(X)$. The group $G$ is a topological full group if for every $n \in \NN$ and for all $(Y_1,Y_2, \ldots ,Y_n)$ and $(Z_1,Z_2,\ldots , Z_n)$ finite partitions into clopen subsets of $X$ and every $g_i \in G$ such that $g_i(Y_i)=Z_i$, the element $g$ that satisfies $g=g_i$ on $X_i$ is an element of $G$.
\end{Def}

One can notice that $\IET(\SGIET)$ and $\IET^{\bowtie}(\SGIET)$ act minimally on the space $Y$ obtained from $\mathbb{R/Z}$ by replacing every point $x \in \SGIET$ by two points $x_-$ and $x_+$. The space $Y$ is a Stone space without isolated points.
The fact that $\IET(\SGIET)$ and $\IET^{\bowtie}(\SGIET)$ are both topological full group is immediate. We can use the following theorem due to V. Nekrashevych in \cite{nekrashevych_2019}:

\begin{Thm}[Nekrashevych]
Let $X$ be a infinite Stone space and $G$ be a subgroup of $\Homeo(X)$ such that $G$ is a topological full group and acts minimally on $X$. Then the subgroup $\mathcal{A}(G)$ of $G$ generated  by the set of all elements of order $3$, is simple and contained in every nontrivial normal subgroup of $G$.
\end{Thm}

Thanks to this theorem and Proposition \ref{Proposition the derived subgroup is included in any normal subgroups} we deduce that $\mathcal{A}(\IET(\SGIET))=D(\IET(\SGIET))$ and $\mathcal{A}(\IET^{\bowtie}(\SGIET))=D(\IET^{\bowtie}(\SGIET))$ and both of them are simple.

\bigskip

\section{\texorpdfstring{Behaviour of the composition of two $\IET$-transpositions}{Behaviour of the composition of two IET-transpositions}}\label{Section Behaviour of the composition of two IET-transpositions}

Understanding compositions of $\IET(\SGIET)$-transpositions is important to describe the abelianization of $\IET(\SGIET)$. The next lemma relates finite order elements with $\IET$-transpositions (see Definition \ref{Definition IET transposition}). In \cite{Vorobets2011}, Y. Vorobets proves it in the case $\SGIET=\mathbb{R/Z}$. The general case follows from this case:

\begin{Lem}\label{Lemma finite order implies being a product of IET transpositions}
Let $f$ be a finite order element of $\IET(\SGIET)$. Let $n \in \NN$ and $\cP:=\lbrace I_1, \ldots , I_n \rbrace \in \Pi_f$. Then there exists $\sigma \in \mfS_n$ such that $f(I_i)=I_{\sigma(i)}$ thus $f$ is a product of $\IET(\SGIET)$-transpositions.
\end{Lem}

The aim is to show that the product of two $\IET$-transpositions has finite order. It is not a trivial question because we can build such examples of product of order $n$ for every $n \in \NN$ (see Proposition \ref{Proposition example of product of two IET transposition can has arbitrarily large order.}). The result will be used in Section \ref{Section Description of the abelianization of IET(Gamma)}.

\begin{Lem}\label{Lemma product of two transpositions has orbit that leave the intersection of their support}
Let $f$ and $g$ be two $\IET$-transpositions. Let $I$ and $J$ be the non-overlapping intervals swapped by $f$ with $\sup(I) = \inf(J)$. Let $A$ and $B$ be the non-overlapping intervals swapped by $g$ with $\sup(A) = \inf(B)$. Let $\alpha$ and $\beta$ be the elements of $\mathopen{[} 0,\dfrac{1}{2} \mathclose{[}$ such that $J= I+\alpha$ and $B=A + \beta$. If $\alpha > \beta$ then for all $x$ in $\mathopen{[} 0,1 \mathclose{[}$ there exists $n$ in $\NN_{\geq 1}$ such that $(gf)^n (x) \notin \support(f) \cap \support(g)$.
\end{Lem}

\begin{proof}
We have the conclusion for all $x$ outside $\support(f) \cup \support(g)$. By contradiction we assume that there exists $x$ in $\support(f) \cup \support(g)$ such that for every $n$ in $\NN_{\geq 1}$, the element $(gf)^n(x)$ belongs to $\support(f) \cap \support(g)$. Up to consider $(gf)(x)$ instead of $x$ we can assume that $x$ is in $\support(f) \cap \support(g)$. We distinguish three cases:
\begin{enumerate}
\item If $I \cap B = \emptyset$ and $ J \cap A = \emptyset$. Then $\support(f) \cap \support(g)= (I \cap A) \sqcup (J \cap B)$. Moreover we check that for every $n \in \NN_{\geq 1}$ if $y=(gf)^n(x)$ is in $I \cap A$ then $f(y)$ is in $ J \cap B$. Indeed $f(y)$ is clearly in $J$ thus $f(y) \notin A$. If it is not in $B$ then $f(y)$ is fixed by $g$. Thus $gf(y)=f(y)$ is neither in $A$ nor in $B$ which is in contradiction with the assumption on $x$. Similarly we get $gf(y)$ in $I \cap A$. Thus in one iteration we oscillate between the two intervals. The same is true if $y$ is in $J \cap B$ at first. Hence for $n > \frac{1}{\alpha-\beta}$ if $x$ is in $I \cap A$ then $(gf)^n(x)=x + n \times (\alpha -\beta) > x+1 \geq 1$ which is a contradiction. Also if $x$ is in $J \cap B$ then $(gf)^n(x)=x - n \times (\alpha -\beta) < x-1 < 0$ which is a contradiction.

\item Now let us assume that $I \cap B \neq \emptyset$ (this gives the strict inequality $\sup(I) > \inf(B)$), then $J \cap A = \emptyset$ because $\inf(J)= \sup(I) > \inf(B) = \sup(A)$. We have two cases:
\begin{enumerate}
\item If the positive $(gf)$-orbit of $x$ does not intersect $I \cap B$. Let $f'$ be the $\IET(\SGIET)$-transposition which swaps the interval $I\setm B$ with $f(I \setm B)$. Then for every $n$ in $\NN_{\geq 1}$ we have $(gf')^n(x)$ in $\support(f') \cap \support(g)$. Moreover $(I \setm B) \cap B = \emptyset$ and $f(I \setm B) \cap J= \emptyset$. Hence we return in the first case which lead to a contradiction.
\item If the positive $(gf)$-orbit of $x$ intersect $I \cap B$. Let $k \in \ZZ$ such that $y:= (gf)^k(x)$ belongs to $I \cap B$. Then for every $n$ in $\NN_{\geq 1}$ we always have $(gf)^n(y)$ in $\support(f) \cap \support(g)$. Hence as $y$ is in $I \cap B$ we get $f(y)$ in $J \cap B$. Indeed it is clearly in $J$ and it is in $B$ otherwise $f(y)$ will not be in $A \cup B$ which is equal to $\support(g)$ and lead to a contradiction. So $gf(y)$ is in $A$ and $ff(y)=y$ is in $B$. We deduce that $gf(y)= y + \alpha-\beta \leq \sup(A)  = \inf(B) \leq ff(y) =y$ which is a contradiction because $\alpha > \beta$.
\end{enumerate}

\item The last case is where $J \cap A \neq \emptyset$ but it is a similar argument as the previous one and it also lead to a contradiction.
\end{enumerate}
In conclusion the lemma holds for every $x \in \support(f) \cup \support(g)$.
\end{proof}

\begin{Rem}\label{calcul de puissance}
Let $x$ be an element of $\mathopen{[}0,1 \mathclose{[}$. If there exists $n$ in $\NN_{\geq 1}$ such that $(gf)^n(x)$ is not in $\support(f)$ then $gf(gf)^n(x)=g(gf)^n(x)=f(gf)^{n-1}(x)$ so $(gf)^n(gf)^n(x)=f(x)$. If there exists $n$ in $\NN_{\geq 1}$ such that $(gf)^n(x)$ is not in $\support(g)$ then $(gf)^n(x)=g(gf)^n(x)=f(gf)^{n-1}(x)$ so $(gf)^{n-1}(gf)^n(x)=f(x)$.\\
We deduce that if there exist $n$ and $k$ in $\NN_{\geq 1}$  such that $(gf)^n(x)$ and $(gf)^k (f(x))$ are not in $\support(f) \cup \support(g)$ then $x$ has finite order at most $2n+2k$.
\end{Rem}

\begin{Lem}\label{Lemma product of two transpositions has finite order}
Let $f$ and $g$ be two $\IET$-transpositions. Then $gf$ has finite order.
\end{Lem}

\begin{proof}
As $gf$ is in $\IET$ we know that if $gf$ has periodic points, the set of periods is finite. Hence it is sufficient to prove that each point of $\interfo{0}{1}$ has a periodic $(gf)$-orbit to conclude. For every $x$ which is neither in $\support(f)$ nor $\support(g)$ then $gf(x)=x$. Hence we only need to check for $x$ is in $\support(f) \cup \support(g)$.\\
Let $I$ and $J$ be the non-overlapping intervals swapped by $f$ with $\sup(I) = \inf(J)$ and let $A$ and $B$ be the non-overlapping intervals swapped by $g$ with $\sup(A) = \inf(B)$. Let $\alpha$ and $\beta$ be the elements of $\mathopen{[} 0,\dfrac{1}{2} \mathclose{[}$ such that $J= I+\alpha$ and $B=A + \beta$.\\
If $\alpha=\beta$ then for every $x$ in $\support(f) \cap \support(g) $ then $gf(x)=x$. Moreover for $x$ in $\support(f) \setm \support(g)$ we have $f(x)$ not in $\support(g)$ thus $gf(x)=f(x)$ and $gfgf(x)=gf (fx)=g(x)=x$. The same stands for $x$ in $\support(g) \setm \support(f)$. In conclusion $gf$ has finite order at most $2$.\\
If $\alpha \neq \beta$, we can assume that $\alpha > \beta$ up to change the role of $f$ and $g$ and because $fg$ is conjugate to $gf$. Let $x$ in $\support(f) \cup \support(g)$. Then by Lemma \ref{Lemma product of two transpositions has orbit that leave the intersection of their support} there exist $n$ and $k$ in $\NN_{\geq 1}$ such that $(gf)^n(x)$ is not in $\support(f) \cap \support(g)$ and $(gf)^k (f(x))$ is not in $\support(f) \cap \support(g)$. Thanks to Remark \ref{calcul de puissance} we know that $x$ has finite order at most $2n+2k$.
\end{proof}

\begin{Prop}\label{Proposition example of product of two IET transposition can has arbitrarily large order.}
For every $n$ in $\NN_{\geq 1}$ there exist two $\IET$-transpositions $f$ and $g$ such that the product $gf$ has order $n$.
\end{Prop}

\begin{proof}
We distinguish the case where $n$ is even or odd. In both cases we illustrate the proof with Figure \ref{figure_produit_transposition}. The case $n=1$ is given by the equality $f^2=\Id$ for any $\IET$-transposition $f$. Let $n \in \NN_{\geq 1}$. 

Let $I$ and $J$ be two consecutive intervals of the same length $\ell \in \mathopen{[}0,\frac{1}{2} \mathclose{]}$ and let $g$ be the $\IET$-transposition that swaps $I$ and $J$. Let $A_1,A_2, \ldots , A_{n-1}$ and $C$ be consecutive intervals of length $\frac{\ell}{n}$ such that the left endpoint of $A_1$ is the left endpoint of $I$ (hence the right endpoint of $C$ is the right endpoint of $I$). Let $D$ and $B_1,B_2, \ldots , B_{n-1}$ be consecutive intervals of length $\frac{\ell}{n}$ such that the right endpoint of $B_{n-1}$ is the right endpoint of $J$ (hence the left endpoint of $D$ is the left endpoint of $J$). Let $f$ be the $\IET$-transposition that swaps $A_i$ and $B_i$ for every $1 \leq i \leq n-1$. Hence by definition we get $g(A_1)=D$, $g(A_i)=B_{i-1}$ for every $2 \leq i \leq n-1$ and $g(C)=B_{n-1}$. So the composition $gf$ is equal to the permutation $(A_1 ~ A_2 ~ \ldots ~ A_{n-1} ~ C ~ B_{n-1} ~ B_{n-2} ~ \ldots ~ B_1 ~ D)$. Thus $gf$ has order $2n$.

It remains the case of order $2n-1$. Let $I,J$ and $K$ be three consecutive intervals with $I$ and $J$ of length $\ell \in \interfo{0}{\frac{1}{3}}$ and $K$ of length $\ell' \in \mathopen{]} \frac{\ell}{n},\ell \mathclose{[}$. Let $g$ be the $\IET$-transposition that swaps $I$ and $J$. We define $A_1,A_2, \ldots, A_{n-1}$ consecutive intervals of length $\frac{\ell}{n}$ such that the right endpoint of $A_{n-1}$ is the right endpoint of $I$. We define also $D$ and $B_1,B_2, \ldots , B_{n-1}$ consecutive intervals of length $\frac{\ell}{n}$ such that the left endpoint of $B_{n-1}$ is the left endpoint of $K$. Let $f$ be the $\IET$-transposition that swaps $A_i$ and $B_i$ for every $1 \leq i \leq n-1$. One can check that the product $gf$ in this case is the permutation $(A_1 ~ A_2 ~ \ldots ~ A_{n-1} ~ B_{n-1} ~ B_{n-2} ~ \ldots ~ B_1 ~ D)$ so $gf$ has order $2n-1$.
\end{proof}

\begin{figure}[!h]
\includegraphics[width=0.4\linewidth]{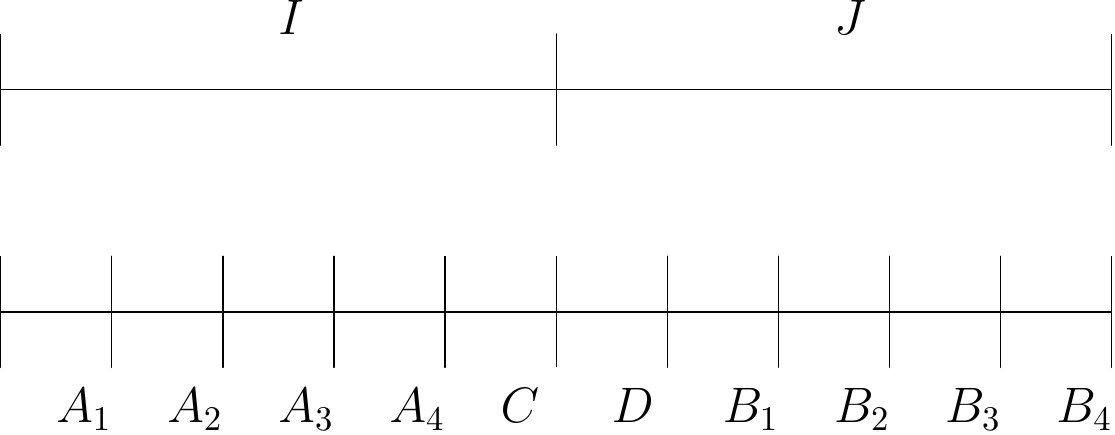}
\hspace{20pt}
\includegraphics[width=0.48\linewidth]{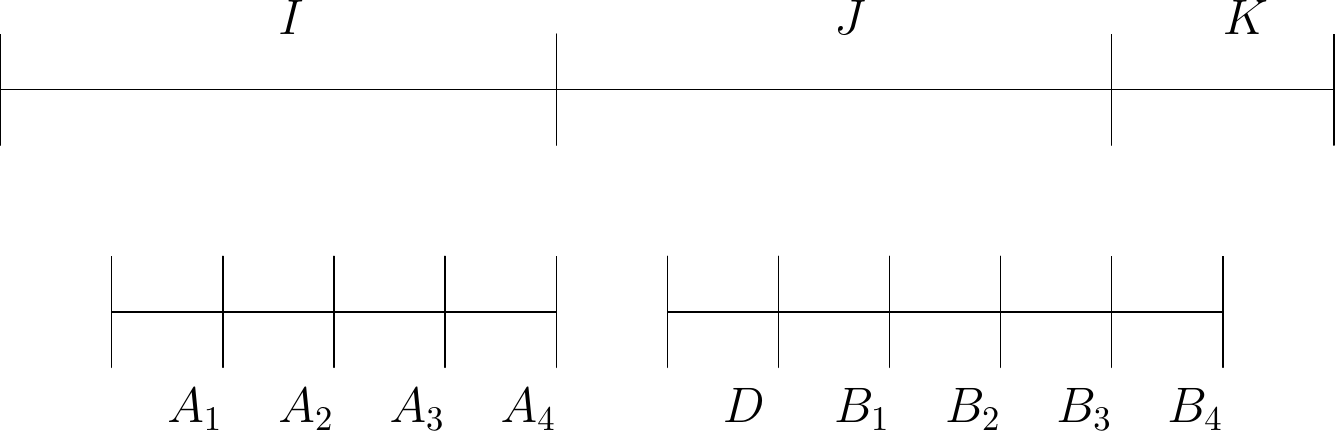}
\caption{\footnotesize{Illustration of Proposition \ref{Proposition example of product of two IET transposition can has arbitrarily large order.} with $n=5$.
\textbf{Left}: For the case ``$fg$ has order $2n$''.
\textbf{Right}: For the case ``$fg$ has order $2n-1$''.}}\label{figure_produit_transposition}
\end{figure}

\bigskip

\section{\texorpdfstring{A generating subset for $\Ker(\varphi_{\Gamma})$}{A generating subset for Ker(phi index Gamma)}}\label{Section a generating subset for Ker(restriction SAF)}

Let $\SGIET$ be a dense subgroup of $\mathbb{R/Z}$. We follow the idea of Y. Vorobets in \cite{Vorobets2011} and introduce the notion of balanced product of restricted rotations. The aim is to show Lemma \ref{Lemma Ker restriction SAF is generated by balanced product of IET Gamma restricted rotations}. We recall that $\varphi_{\SGIET}$ is the restriction of the $\SAF$-homomorphism to $\IET(\SGIET)$; and that $\tilde{\SGIET}_{+}$ is the positive cone of $\tilde{\SGIET}$.

\subsection{\texorpdfstring{Balanced product of $\Gamma$-restricted rotations}{Balanced product of Gamma-restricted rotations}}\label{subsection Ker restriction SAF and balanced product}~

\medskip

\begin{Def}
Let $n \in \NN$ and let $r_i$ be a restricted rotation for every $1 \leq i \leq n$. For every $a,b \in \tilde{\Gamma}_{+}$ let $n_{a,b}$ be the number of restricted rotation $r_i$ of type $(a,b)$. The tuple $(r_1,r_2,\ldots, r_n)$ is said to be a \textit{balanced tuple of restricted rotations} if $n_{a,b}=n_{b,a}$ for every $a,b \in \tilde{\SGIET}_{+}$. We say that a product of restricted rotations is a \textit{balanced product of restricted rotations} if it can be written as a product of a balanced tuple of restricted rotations.
\end{Def}

\begin{Exem}
Let $a$ be an element in $\tilde{\SGIET}_+$ with $a \leq \frac{1}{2}$.
Any $\SGIET$-restricted rotation of type $(a,a)$ is a balanced product of restricted rotations. These elements are also $\IET(\SGIET)$-transposition which swapped two consecutive intervals of same length $a$.
\end{Exem}

\begin{Exem}
Let $a,b \in \tilde{\SGIET}_+$ with $a+b \leq 1$. If $h$ is a $\SGIET$-restricted rotation of type $(a,b)$ then $h^{-1}$ is a $\SGIET$-restricted rotation of type $(b,a)$. Thus every element of $D(\IET(\SGIET))$ is a balanced product of $\SGIET$-restricted rotations.
\end{Exem}

In order to get the decomposition of the elements of $\Ker(\varphi_{\SGIET})$, we need to know the freeness of some family of $\Wedgealt V$.

\begin{Lem}\label{Lemma freeness of the image of a basis in the exterior algebra}
Suppose $V$ is a subgroup of $\RR$ (then it is a $\ZZ$-module). Let $k$ in $\NN_{\geq 1}$ and $v_1,v_2,\ldots v_k$ elements of $V$ which are $\ZZ$-linearly independent. Then the wedge products $v_i \wedge v_j$ for $1 \leq i < j \leq k$ are $\ZZ$-linearly independent in $\Wedgealt V$.
\end{Lem}

\begin{proof}
Let $v_1,v_2,\ldots v_k$ in $V$ which are $\ZZ$-linearly independent. It is sufficient to prove the lemma for $V=\RR$ because being $\ZZ$-linearly independent in $\Wedgealt \RR$ implies being $\ZZ$-linearly independent in $\Wedgealt V$. We know that $\Wedgealt \RR$ is isomorphic to $\bigwedge^2_{\ZZ} \RR$. Let us assume that $v_1,v_2, \ldots , v_k$ are in $\RR$. Then being $\ZZ$-linearly independent is the same that being $\QQ$-linearly independent. Indeed if there exist $p_1,p_2,\ldots ,p_k$ in $\ZZ$ and $q_1,q_2,\ldots , q_k$ in $\NN_{\geq 1}$ such that $\sum\limits_{i=1}^k \dfrac{p_i}{q_i}v_i=0$ then $\sum\limits_{i=1}^k (\prod\limits_{\substack{j=1 \\ j \neq i}} q_j p_i) v_i ~{=0}$ is an equality in $\ZZ$. Thus for each $i \in \lbrace 1,2,\ldots , k \rbrace$ we have $\prod\limits_{\substack{j=1 \\ j \neq i}} q_j p_i =0$. Or $q_j \neq 0$ for every $j$ then $p_i =0 $ for every $i$. \\
The $\QQ$-vector space generated by all the $v_i$ for $1 \leq i \leq k$ is isomorphic to $\QQ^k$. We can complete the $\QQ$-linearly independent set $\lbrace v_1,v_2,\ldots ,v_k \rbrace$ in a basis $S$ of $\RR$ seen as a $\QQ$-vector space. Thus in $\bigwedge^2_{\ZZ} \RR$ the elements $v_i \wedge v_j$ for $1 \leq i < j \leq k$ are $\QQ$-linearly independent so they are $\ZZ$-linearly independent in $\Wedgealt \RR$ and this gives the conclusion.
\end{proof}

\begin{Lem}\label{Lemma ker(phi) is generated by balanced product}
Any transformation $f$ in $\Ker(\varphi_{\SGIET})$ can be represented as a balanced product of $\SGIET$-restricted  rotations. 
\end{Lem}

\begin{proof}
Let $f \in \Ker(\varphi_{\SGIET})$. This is trivial if $f=\mathrm{id}$; assume otherwise. Let $(\mu,\sigma)$ be a $\SGIET$-combinatorial description of $f$, let $k \in \NN$ and $\lbrace I_1,I_2, \ldots,I_k \rbrace$ be the partition into intervals associated to $(\mu, \sigma)$ ( we have $k \geq 2$ as $f$ is not the identity). We recall that $\mu_i$ is the length of $I_i$ for every $1 \leq i \leq k$.

We treat the case where $\SGIET$ is finitely generated. Then $\tilde{\SGIET}$ is also finitely generated and we denote by $d$ its rank. By Corollary \ref{Corollary ultrasimplicially ordered group} there exist $\ell_1,\ell_2,\ldots,\ell_d$ in $\tilde{\SGIET}_{+}$ such that $\cL:=\lbrace \ell_1,\ell_2, \ldots,\ell_d \rbrace$ is a basis of $\tilde{\SGIET}$ and such that $L_i$ is a linear combination of $\ell_1,\ell_2, \ldots, \ell_d$ with non-negative integer coefficients for every $1 \leq i \leq k$. Then $I_i$ can be partitioned into smaller intervals with length in $\cL$ for every $1 \leq i \leq k$. We obtain a partition associated with $f$ whose intervals have length in $\cL$. By Lemma \ref{Lemma Decomposition in IET(Gamma)-restricted rotations} there exist $n\in \NN$ and a restricted rotation $f_i$ of type $(a_i,b_i)$ with $a_i,b_i \in \cL$ for $1 \leq i \leq n$ such that $f=f_1f_2\ldots f_n$. For any $i,j \in \lbrace 1,2,\ldots ,d \rbrace$ let $s_{ij}$ be the number of $\SGIET$-restricted rotation of type $(\ell_i,\ell_j)$ in the tuple $(f_1,f_2,\ldots,f_n)$. As $\varphi_{\SGIET}(f_i)=\ell_j \wedge \ell_i -\ell _i \wedge \ell_j=2 \ell_j \wedge \ell_i$, we obtain that: 
$$\varphi_{\SGIET}(f)=\sum\limits_{i=1}^d \sum\limits_{j=1}^d 2 s_{ij}(\ell_j \wedge \ell_i)=\sum\limits_{1\leq i < j \leq d}  2 (s_{ij}-s_{ji})(\ell_j \wedge \ell_i)$$
We know that $\lbrace \ell_1,\ell_2,\ldots , \ell_d \rbrace$ is a basis of $\tilde{\SGIET}$ thus by Lemma \ref{Lemma freeness of the image of a basis in the exterior algebra} we obtain that $\lbrace \ell_j \wedge \ell_i \rbrace_{1 \leq i<j \leq d}$ is a free family of $\Wedgealt \tilde{\SGIET}$. Then the assumption $\varphi_{\SGIET}(f)=0$ implies $s_{ij}=s_{ji}$ for every $1 \leq i <j \leq d$. This means that the product of $\SGIET$-restricted rotations $f_1f_2\ldots f_n$ is balanced.

We do not assume $\SGIET$ finitely generated any more. Hence we only know that $\varphi_{\SGIET}(f)=\sum\limits_{j=1}^{k} \big( \sum\limits_{\substack{i: \\ \sigma(i) < \sigma(j)}} \mu_i - \sum\limits_{i<j} \mu_i\big) \wedge \mu_j=0$ in $\Wedgealt \tilde{\SGIET}$ (see Proposition \ref{Proposition SAF invariant with combinatorial description}). We denote by $\overline{\varphi_{\SGIET}(f)}:=\sum\limits_{j=1}^{k} \big( \sum\limits_{\substack{i: \\ \sigma(i) < \sigma(j)}} \mu_i - \sum\limits_{i<j} \mu_i\big) \otimes \mu_j$. It is a representative of $\varphi_{\SGIET}(f)$ in $\bigotimes^2_{\ZZ} \tilde{\SGIET}$. Then there exist a finite set $J$ and $x_j,y_j \in \tilde{\SGIET}$ for every $j \in J$, such that $\overline{\varphi_{\SGIET}(f)}= \sum\limits_{j \in J} x_j \otimes y_j + y_j \otimes x_j$. We denote by $\tilde{A}$ the subgroup of $\tilde{\SGIET}$ generated by $\lbrace \mu_i \rbrace_{1 \leq i \leq k} \cup \lbrace x_j,y_j \rbrace_{j \in J}$. Then $\tilde{A}$ is a finitely generated subgroup of $\RR$ which contains $\ZZ$. Its image $A$ in $\mathbb{R/Z}$ is a finitely generated subgroup of $\mathbb{R/Z}$. Also we know that $f$ is in $\IET(A)$ and $(\mu,\sigma)$ is also a $A$-combinatorial description of $f$ and $\overline{\varphi_{\SGIET}(f)}$ is an element of $\bigotimes^2_{\ZZ} \tilde{A}$. Thus in $\Wedgealt \tilde{A}$ we have:
$$\varphi_{A}(f)=\sum\limits_{j=1}^{k} \big( \sum\limits_{\substack{i: \\ \sigma(i) < \sigma(j)}} \mu_i - \sum\limits_{i<j} \mu_i\big) \wedge \mu_j=[\overline{\varphi_{\SGIET}(f)}]_{\Wedgealt \tilde{A}}=[\sum\limits_{j \in J} x_j \otimes y_j + y_j \otimes x_j]_{\Wedgealt \tilde{A}}=0$$
Then we can applied the previous case and conclude that $f$ is a balanced product of $A$-restricted rotations, thus a balanced product of $\SGIET$-restricted rotations.
\end{proof}

\subsection{\texorpdfstring{$\Ker(\varphi_{\Gamma})$ is generated by $\IET(\Gamma)$-transpositions}{Ker(phi index Gamma) is generated by IET(Gamma)-transpositions}}~

\smallskip

The work of Y.Vorobets \cite{Vorobets2011} done for $\IET$ can be adapted to show the next two lemmas:

\begin{Lem}\label{Lemma product of same type restricted rotation}
Let $f$ and $g$ be two $\SGIET$-restricted rotations. If they have the same type then $f^{-1}g$ is a product of $\IET(\SGIET)$-transpositions.
\end{Lem}

\begin{Lem}\label{Lemma commutator with a restricted rotation is product of transposition}
Let $f$ be a $\SGIET$-restricted rotation and $g$ be any transformation in $\IET(\SGIET)$. Then the commutator $[f,g]$ is a product of $\IET(\SGIET)$-transpositions.
\end{Lem}

These lemmas with Lemma \ref{Lemma for every epsilon any element is a product of elements of support of measure smaller than epsilon} give us the next theorem. It is proved by Y.Vorobets in the case $\SGIET= \mathbb{R / Z}$ in \cite{Vorobets2011}, and as the proof is the same we only provide here a sketch.

\begin{Thm}\label{Theorem balanced product of IET Gamma restricted rotation is a product of IET Gamma transpositions}
Every balanced product of $\SGIET$-restricted rotations can be written as a product of $\IET(\SGIET)$-transpositions.
\end{Thm}

\begin{proof}[Proof of Theorem \ref{Theorem balanced product of IET Gamma restricted rotation is a product of IET Gamma transpositions} (sketched)]~\\
Let $(f_1,f_2,\ldots,f_n)$ be a balanced tuple of restricted rotations. The proof is by strong induction on the length $n$ of the tuple. If $n=1$ then $(f_1)$ is a balanced tuple of $\SGIET$-restricted rotations, thus $f_1$ is a $\SGIET$-restricted rotation of type $(a,a)$ with $a \in \tilde{\SGIET}$ so it is also an $\IET(\SGIET)$-transposition.\\
For the general case, let $(a,b)$ be the type of $f_1$. If $a=b$ then $f_1$ is an $\IET(\SGIET)$-transposition and $(f_2,f_3,\ldots,f_n)$ is a balanced tuple of restricted rotations. By the induction assumption we obtain the result. If $a \neq b$ then there exists $k \in \lbrace 2, \ldots ,n \rbrace$ such that $f_k$ is a $\SGIET$-restricted rotation of type $(b,a)$. Let $g_1=f_2 \ldots f_{k-1}$ or $g_1=\Id$ if $k=2$. Let $g_2=f_{k+1}\ldots f_n$ or $g_2= \Id$ if $k=n$. Then we can write 
$$
f_1f_2\ldots f_n=f_1g_1f_kg_2=(f_1f_k)(f_k^{-1} g_1 f_k g_1^{-1})(g_1 g_2)
$$
Hence, the induction assumption and Lemmas \ref{Lemma product of same type restricted rotation} and \ref{Lemma commutator with a restricted rotation is product of transposition} give the result.
\end{proof}

\begin{Coro}\label{Corollary Ker phi is generated by transpositions}
The kernel $\Ker(\varphi_{\SGIET})$ is generated by the set of all $\IET(\SGIET)$-transpositions.
\end{Coro}

\bigskip

\section{\texorpdfstring{Description of the abelianization of $\IET(\Gamma)$}{Description of the abelianization of IET(Gamma)}}\label{Section Description of the abelianization of IET(Gamma)}~
\smallskip

In this section we construct a group homomorphism $\eps_{\SGIET}: \IET(\SGIET) \rightarrow \Wedgealt \tilde{\SGIET}$ whose kernel is the derived subgroup $D(\IET(\SGIET))$.

\subsection{Boolean measures}~
\smallskip

In finite permutation groups there is a natural signature. One way to describe the signature is as follows: the signature of a finite permutation $f$, viewed in $\ZZ / 2\ZZ$ is the number modulo 2 of pairs $(x,y)$ such that $x<y$ and $f(x) >f(y)$. In our context where $f \in \IET(\SGIET)$, while this set is infinite, the idea is to measure it in a meaningful sense.

\begin{Def}
Let $A$ be a Boolean algebra and $G$ be an abelian group. Let $\mu: A \rightarrow G$ be a finitely additive map: $\forall~ U,V \in A$ disjoint, $\mu(U \sqcup V)=\mu(U)+\mu(V)$. Such a $\mu$ is called a \textit{Boolean algebra measure for $A$ in $G$}.
\end{Def}

\begin{Nota}
We recall that $\Itv(\SGIET)$ is the set of all intervals $[a,b[$ with $a$ and $b$ in $\tilde{\SGIET}$ and $0 \leq a <b \leq 1$. Let $A_{\SGIET}$ be the Boolean algebra of subsets of $\interfo{0}{1}$ generated by $\Itv(\SGIET)$. Then $A_{\SGIET}$ is a Boolean subalgebra of $\lbrace 0,1 \rbrace^{\mathopen{[}0,1 \mathclose{[}}$. By noting $\lambda$ the Lebesgue measure on $\mathopen{[}0,1\mathclose{[}$ we get that $\lambda$ is a Boolean measure for $A_{\SGIET}$ in $\tilde{\SGIET}$.
\end{Nota}

It might be useful to notice that for $k$ in $\NN_{\geq 1}$ and every $I_1,I_2, \ldots , I_k$ intervals in $\Itv(\SGIET)$, the intersection $\bigcap\limits_{i} I_i$ is still an element of $\Itv(\SGIET)$. Moreover for every $I$ in $\Itv(\SGIET)$, the complement of $I$ is the disjoint union of two elements of $\Itv(\SGIET)$. Thus any Boolean combination of elements of $\Itv(\SGIET)$ is a finite disjoint union of such elements.

\begin{Prop}
Let $X$ and $Y$ be two sets, let $A$ be a Boolean subalgebra of $\lbrace 0,1 \rbrace^{X}$ and let $B$ be a Boolean subalgebra of $\lbrace 0,1 \rbrace^Y$. Let $G$ and $H$ be two abelian groups, let $\mu: A \rightarrow G$ be a Boolean algebra measure for $A$ in $G$ and $\nu: B \rightarrow H$ be a Boolean algebra measure for $B$ in $H$. Let $C:= A \otimes B$ be the Boolean algebra product (generated by subsets of the form $a \times b$ with $a$ in $A$ and $b$ in $B$). Then there exists a unique Boolean algebra measure $\omega: C \rightarrow G \otimes H$ for $C$ in $G \otimes H$ such that for every $a$ in $A$ and $b$ in $B$ we have $\omega(a \times b)=\mu(a) \otimes \nu(b)$.
\end{Prop}

\begin{proof}
Let $\omega_1$ and $\omega_2$ be two such Boolean algebra measures, thus they are equal on every $a \times b$ for $a \in A$ and $b \in B$. Let $c$ be an element of $C$, then there exist $k$ in $\NN$ and $a_1,\ldots , a_k$ in $A$ and $b_1,\ldots , b_k$ in $B$ such that $c=\bigsqcup\limits_{i=1}^k a_i \times b_i$. So $\omega_1 (c)=\sum\limits_{i=1}^k \omega_1(a_i\times b_i)=\sum\limits_{i=1}^k \omega_2 (a_i\times b_i)=\omega_2(c)$. Thus $\omega_1=\omega_2$ and the unicity is proved. \\
It is sufficient to prove the existence for every finite Boolean subalgebra of $C$. Indeed if we assume that for every $D$ finite Boolean subalgebra of $C$ there exists a Boolean algebra measure $m_D$ for $D$ in $G \otimes H$ such that $m_D(a \times b)=\mu(a)\otimes \nu (b)$ for every $a$ in $A$ and $b$ in $B$ with $a \times b$ in $D$. Let $c$ be an element of $C$. Then $\lbrace 0_C,~ c,~ \neg c,~ 1_C \rbrace$ is a finite Boolean subalgebra of $C$ non-trivial. Moreover if $c$ is in $D_1 \cap D_2$ where $D_1$ and $D_2$ are two finite Boolean subalgebras of $C$ then by noting $D$ the Boolean subalgebra generated by $D_1$ and $D_2$ we get that $D$ is a finite Boolean subalgebra of $C$ containing $c$. Thus $m_D |_{D_1}$ is a Boolean measure for $D_1$ in $G \otimes H$  which satisfies $m_D |_{D_1}(a \times b)=\mu(a)\otimes \nu (b)$ for every $a$ in $A$ and $b$ in $B$ with $a \times b$ in $D_1$. By unicity we get $m_D |_{D_1}=m_{D_1}$ and the same argument gives $m_D |_{D_2}=m_{D_2}$ thus $m_{D_1}(c)=m_D(c)=m_{D_2}(c)$. So by putting $\omega(c)=m_D(c)$, the map $\omega$ is well-defined. also if we take two disjoint elements $c$ and $c'$ in $C$. Then by taking any finite Boolean subalgebra $D$ of $C$ which contains $c$ and $c'$ we get $m_D(c+c')=m_D(c)+m_D(c')=\omega(c)+\omega(c')$ and the value does not depend on $D$. Thus $\omega$ is the wanted Boolean algebra measure.\\
Let now $D$ be a finite Boolean subalgebra of $C$. Then there exist $k,\ell \in \NN$ and $a_1, \ldots , a_k \in A$ and $b_1, \ldots , b_{\ell} \in B$ such that every $d \in D$ is a Boolean combination of $a_i \times b_j$. Then let $D'$ be the finite Boolean algebra generated by all the $a_i \times b_j$ with $1 \leq i \leq k$ and $1 \leq j \leq \ell$. Let $U$ be the finite Boolean subalgebra of $A$ generated by all $a_i $ and let $V$ be the finite Boolean subalgebra of $B$ generated by all $b_j$. Then $U$ and $V$ are atomic. Let $u_1, \ldots u_n$ be the atoms of $U$ and $v_1, \ldots v_m$ be the atoms of $V$. Hence $D'$ is atomic with atoms given by $u_i \times v_j$ for every $1 \leq i \leq k$ and $1 \leq j \leq \ell$. Then for each element $d$ in $D'$ there exists a unique $J_d \subset \lbrace1,2,\ldots, n \rbrace \times \lbrace 1,2, \ldots, m \rbrace$ such that $d=\bigsqcup\limits_{(i,j) \in J_d} u_i \otimes v_j$. Hence the map $m_{D'}$ defined by $m_{D'}(d)=m_{D'}(\bigsqcup\limits_{(i,j) \in J_d} u_i \times v_j)=\sum\limits_{(i,j) \in J_d} \mu(u_i) \otimes \nu(v_j)$  is well-defined, finitely additive and satisfies $m_{D'}(a \times b)=\mu(a) \otimes \nu(b)$ for every $a \in A, ~b \in B$ such that $ a \times b \in D'$.
\end{proof}

\begin{Nota}
By applying the previous proposition with $X=Y=\mathopen{[} 0,1 \mathclose{[}$ and $A=B=A_{\SGIET}$ and $\mu=\nu=\lambda$, there exists a unique Boolean algebra measure $\omega_{\SGIET}: A_{\SGIET} \otimes A_{\SGIET} \rightarrow \bigotimes^2_{\ZZ} \tilde{\SGIET}$ such that for every $a,b,c$ and $d$ in $\tilde{\SGIET}_+$ with $a<b \leq 1$ and $c<d \leq 1$ we have $\omega_{\SGIET}(\mathopen{[}a,b\mathclose{[} \times \mathopen{[}c,d \mathclose{[} )=(b-a)\otimes (d-c)$.
\end{Nota}

We need to check some $\IET(\SGIET)$-invariance for the measure $\omega_{\SGIET}$. For this we define an action of $\IET(\SGIET)$ on $\mathopen{[}0,1 \mathclose{[} \times \mathopen{[}0,1 \mathclose{[} $ by $f.(x,y)=(f(x),f(y))$. Hence for every $P$ in $A_{\SGIET}$ we have $f.P$ in $A_{\SGIET}$, this gives us a new Boolean algebra measure $f_{*}\omega_{\SGIET}$.

\begin{Prop}\label{Proposition the measure is IET(Gamma) invariant}
For every $f$ in $\IET(\SGIET)$ and every $P$ in $A_{\SGIET} \otimes A_{\SGIET}$ we have $f.P:=\lbrace (f(x),f(y)) \mid (x,y) \in P \rbrace$ in $A_{\SGIET} \otimes A_{\SGIET}$. Furthermore we have  $f_{*}\omega_{\SGIET}=\omega_{\SGIET}$.
\end{Prop}

\begin{proof}
Let $f \in \IET(\SGIET)$ and $(\mu,\sigma)$ be a $\SGIET$-combinatorial description of $f$ and let $\lbrace I_1, \ldots, I_n \rbrace$ be the partition into intervals associated. Let $P$ be an element of $A_{\SGIET} \otimes A_{\SGIET}$. There exist $m$ in $\NN$ and pairwise disjoint elements $p_1, p_2, \ldots , p_m$ of $\Itv(\SGIET) \times \Itv(\SGIET)$ such that $P=\bigsqcup\limits_{i=1}^{k} p_i$. As $f$ is a permutation of $\interfo{0}{1}$ we get $f.\bigsqcup\limits_{i=1}^{m} p_i=\bigsqcup\limits_{i=1}^{m} f.p_i$, so it is enough to show that $f.p$ belongs to $A_{\SGIET} \otimes A_{\SGIET}$. For $i \in \lbrace 1, \ldots, k \rbrace$, let $a_i,b_i,c_i,d_i \in \tilde{\SGIET}$ such that $p_i=\mathopen{[}a_i,b_i \mathclose{[} \times \mathopen{[}c_i,d_i \mathclose{[}$. Then $f.p_i=\bigsqcup\limits_{(k,l)} f(\mathopen{[}a_i,b_i \mathclose{[} \cap I_k) \times f(\mathopen{[}c_i,d_i \mathclose{[} \cap I_l)$ which is a finite disjoint union of elements of $\Itv(\SGIET) \times \Itv(\SGIET)$ because $a_i,b_i,c_i,d_i$ and extremities of $I_l$ are in $\tilde{\SGIET}$ for each $1 \leq l \leq n$. In conclusion $f.p_i$ is in $A_{\SGIET} \otimes A_{\SGIET}$ thus $f.P$ is in $A_{\SGIET} \otimes A_{\SGIET}$.\\
Also $f$ is piecewise a translation and $\lambda$ is the Lebesgue measure, so for any $J$ in $\Itv(\SGIET)$ we have $\lambda(J)=\lambda(f(J))$. Thus: 
\begin{align*}
\omega_{\SGIET}(f.p_i) &=\sum\limits_{(k,l)}\omega_{\SGIET}(f(\mathopen{[}a_i,b_i \mathclose{[} \cap I_k) \times f(\mathopen{[}c_i,d_i \mathclose{[} \cap I_l))\\
&=\sum\limits_{(k,l)} \lambda(f(\mathopen{[}a_i,b_i \mathclose{[} \cap I_k)) \otimes \lambda(f(\mathopen{[}c_i,d_i \mathclose{[} \cap I_l))\\
&=(\sum\limits_{k} \lambda(\mathopen{[}a_i,b_i \mathclose{[} \cap I_k)) \otimes (\sum\limits_{l} \lambda(\mathopen{[}c_i,d_i \mathclose{[} \cap I_l))\\
&=\lambda(\mathopen{[}a_i,b_i \mathclose{[}) \otimes \lambda(\mathopen{[}c_i,d_i \mathclose{[})\\
&=\omega_{\SGIET}(p_i)
\end{align*}

This gives us $\omega_{\SGIET}(f.P)=\sum\limits_{i=1}^{m} \omega_{\SGIET}(f.p_i)=\sum\limits_{i=1}^{m} \omega_{\SGIET}(p_i)=\omega_{\SGIET}(P)$. Hence $\omega_{\SGIET}=f^{-1}_{*}\omega_{\SGIET}$. As $f^{-1}$ is also in $\IET(\SGIET)$ we can do the same to deduce $\omega_{\SGIET}=f_{*}\omega_{\SGIET}$.
\end{proof}

\subsection{Creation of a signature}\label{Subsection creation of a signature}~

\begin{Def}\label{Definition inversions}
Let $f$ be a transformation in $\IET(\SGIET)$. Every pair $(x,y)$ in $\mathopen{[} 0,1 \mathclose{[} \times \mathopen{[} 0,1 \mathclose{[}$ such that $x<y$ and $f(x)>f(y)$ is called an \textit{inversion of $f$}. We denote by $\cE_f$ the set of all inversions of $f$.
\end{Def}

\begin{Prop}\label{Proposition cE_f in A_{Gamma}}
Let $f$ be a transformation in $\IET(\SGIET)$ and $(\mu,\tau)$ be a combinatorial description of $f$. We have $\cE_f=\bigsqcup\limits_{j=1}^{n} \bigsqcup\limits_{\substack{i<j \\ \tau(i)>\tau(j)}} I_i \times I_j$.
\end{Prop}

\begin{proof}
Let $(\mu, \tau)$ be a $\SGIET$-combinatorial description of $f$ and let $\lbrace I_1, I_2, \ldots , I_k\rbrace $ be the partition into intervals associated. Let $(x_0,y_0)$ be an element of $\cE_f$. Then there exist $i,j \in \lbrace 1,2,\ldots , k \rbrace$ such that $x_0 \in I_i$ and $y_0 \in I_j$. As $x_0 < y_0$ we have $i\leq j$. Furthermore if $i=j$ then as $f$ is an isometry which preserves the order on $I_i$ we get $f(x_0) < f(y_0)$ which is a contradiction, we deduce that $i<j$. By definition of $f$ we have $f(I_i),~ f(I_j) \in \Itv(\SGIET)$ and they are disjoint. Thus as $f(x_0)>f(y_0)$ we deduce that for every $x \in I_i$ and $y \in I_j$ we have $x<y$ and $f(x)>f(y)$, so $I_i \times I_j \subset \cE_f$. Also, this implies $\sigma(i)>\sigma(j)$ and we deduce that $\bigsqcup\limits_{j=1}^{n} \bigsqcup\limits_{\substack{i<j \\ \tau(i)>\tau(j)}} I_i \times I_j = \cE_f$.
\end{proof}

\begin{Coro}
For every $f \in \IET(\SGIET)$ we have $\cE_f \in A_{\SGIET} \otimes A_{\SGIET}$.
\end{Coro}

We denote by $p$ the projection from $\bigotimes^2_{\ZZ} \tilde{\SGIET}$ into $\Wedgealt \tilde{\SGIET}$.

\begin{Def}\label{Definition signature for IET(SGIET)}
The \textit{signature for $\IET(\SGIET)$} is the following map:
\[
\begin{array}{cccc}
\eps_{\SGIET}: & \IET(\SGIET) & \longrightarrow &\Wedgealt \tilde{\SGIET} \\
 & f & \longmapsto & p \circ \omega_{\SGIET}(\cE_f)
\end{array}
\]
\end{Def}

\begin{Prop}\label{Proposition symmetry for the signature}
For every $A$ and $B$ in $A_{\SGIET}$ we have:
\[
p \circ \omega_{\SGIET} (A \times B)=- p \circ \omega_{\SGIET} (B \times A)
\]
\end{Prop}

\begin{proof}
Let $A,B \in A_{\SGIET}$ then:
\begin{align*}
p \circ \omega_{\SGIET}(A\times B) &= \lambda(A) \wedge \lambda(B)\\
&=
- \lambda(B) \wedge \lambda (A) \\
&=
-p \circ \omega_{\SGIET}(B \times A)
\end{align*}
\end{proof}

\begin{Thm}
The map $\eps_{\SGIET}$ is a group homomorphism.
\end{Thm}

\begin{proof}
Let $f,g \in \IET(\SGIET)$. We denote by $s$ the symmetry of axis $y=x$. We remark that every element $I$ of $A_{\SGIET} \otimes A_{\SGIET}$ satisfies $s(I) \in A_{\SGIET} \otimes A_{\SGIET}$. Then $\lbrace(x,y) \mid x<y,~ g(x)>g(y),~ fg(x)<fg(y) \rbrace= \cE_g \cap sg^{-1}(\cE_f)$ is an element of $A_{\SGIET} \otimes A_{\SGIET}$. We also notice that $\lbrace (x,y) \mid x<y,~ g(x)<g(y),~fg(x)>fg(y) \rbrace = \cE_{f\circ g} \cap \cE_g^{c}$ and $\lbrace (x,y) \mid x<y,~g(x)>g(y),~ fg(x) >fg(y) \rbrace = \cE_{f \circ g} \cap \cE_{g}$ are two elements of $A_{\SGIET} \otimes A_{\SGIET}$.

For more clarity we do some calculus first. By Proposition \ref{Proposition the measure is IET(Gamma) invariant} and Proposition \ref{Proposition symmetry for the signature} we get: 
\begin{align*}
&-p \circ \omega_{\SGIET} (\lbrace (x,y) \mid x<y,~ g(x)>g(y),~ fg(x) < fg(y) \rbrace)
&\\=~
&-p \circ \omega_{\SGIET} (\lbrace (g(x),g(y)) \mid x<y,~ g(x)>g(y),~ fg(x) < fg(y) \rbrace)
&\\=~
&p \circ \omega_{\SGIET} (\lbrace (g(y),g(x)) \mid x<y,~ g(x)>g(y),~ fg(x) < fg(y) \rbrace)
&\\=~
&p \circ \omega_{\SGIET} (\lbrace (u,v) \mid g^{-1}(u)>g^{-1}(v),~ u<v ,~ f(u)>f(v) \rbrace)
\end{align*}
and
\begin{align*}
&p \circ \omega_{\SGIET} (\lbrace (x,y) \mid x<y,~ g(x)<g(y) ,~ fg(x)>fg(y)  \rbrace)
&\\=~
&p \circ \omega_{\SGIET} (\lbrace (g(x),g(y)) \mid x<y,~ g(x)<g(y) ,~ fg(x)>fg(y) \rbrace)
&\\=~
&p \circ \omega_{\SGIET} (\lbrace (u,v) \mid g^{-1}(u)<g^{-1}(v),~ u<v,~ f(u)>f(v) \rbrace)
\end{align*}
Then:
\begin{align*}
&p\circ \omega_{\SGIET} (\lbrace (u,v) \mid g^{-1}(u)>g^{-1}(v),~ u<v ,~ f(u)>f(v) \rbrace)
&\\~
&~~~~~~+p \circ \omega_{\SGIET} (\lbrace (u,v) \mid g^{-1}(u)<g^{-1}(v),~ u<v,~ f(u)>f(v) \rbrace)
&\\=~
&p \circ \omega_{\SGIET}(\cE_f)
&\\=~
&\eps_{\SGIET}(f)
\end{align*}
Hence by adding and remove the same quantity at the fourth equality we obtain:
\begin{align*}
\eps_{\SGIET}(f \circ g)
=~
&p \circ \omega_{\SGIET}(\lbrace (x,y) \mid x<y,~ fg(x)>fg(y) \rbrace 
&\\=~
&p \circ \omega_{\SGIET}(\lbrace (x,y) \mid x<y,~ g(x)>g(y),~ fg(x)>fg(y) \rbrace
&\\~
&~~~~~~\sqcup \lbrace (x,y) \mid x<y,~ g(x)<g(y),~fg(x)>fg(y) \rbrace)
&\\=~
&p \circ \omega_{\SGIET}(\lbrace (x,y) \mid x<y,~ g(x)>g(y),~ fg(x)>fg(y) \rbrace)
&\\~
&~~~~~~+p \circ \omega_{\SGIET}(\lbrace (x,y) \mid x<y,~ g(x)<g(y),~fg(x)>fg(y) \rbrace)  
&\\=~
&p \circ \omega_{\SGIET}(\lbrace (x,y) \mid x<y,~ g(x)>g(y),~ fg(x)>fg(y) \rbrace)
&\\~
&~~~~~~+p \circ \omega_{\SGIET}(\lbrace (x,y) \mid x<y,~ g(x)>g(y),~ fg(x)<fg(y) \rbrace)
&\\~
&~~~~~~-p \circ \omega_{\SGIET}(\lbrace (x,y) \mid x<y,~ g(x)>g(y),~ fg(x)<fg(y) \rbrace)
&\\~
&~~~~~~+p \circ \omega_{\SGIET}(\lbrace (x,y) \mid x<y,~ g(x)<g(y),~fg(x)>fg(y) \rbrace)  
&\\=~
&p \circ \omega_{\SGIET}(\cE_g) + p \circ \omega_{\SGIET}(\cE_f) 
&\\=~
&\eps_{\SGIET}(g) + \eps_{\SGIET}(f)
\end{align*}

In conclusion, $\eps_{\SGIET}$ is additive thus it is a group homomorphism.
\end{proof}

\begin{Prop}\label{Proposition Value of eps_Gamma for a IET Gamma transposition}
Let $a \in \tilde{\SGIET}_+$ with $a \leq \frac{1}{2}$ and $f$ be an $\IET(\SGIET)$-transposition of type $a$. Then $\eps_{\SGIET}(f)= a \wedge a$.
\end{Prop}

\begin{proof}
Let $u,v,b \in \tilde{\SGIET}$ such that $((u,a,b,a,v),(2~4))$ is a $\SGIET$-combinatorial description of $f$ (see Figure \ref{Exemple signature graphe de f}). Let $\lbrace I_1, \ldots ,I_5 \rbrace $ be the partition into intervals associated. We already proved in \ref{Proposition cE_f in A_{Gamma}} that it is sufficient to check if a pair $(x,y) \in I_i \times I_j$ is in $\cE_f$ to know that $I_i \times I_j$ is in $\cE_f$. We also have $I_i \times I_j \notin \cE_f$ if $j\leq i$. Thus one can look at the graph of $f$ to find that $\cE_f$ is equal to the tiling space on Figure \ref{Exemple signature graphe de f}. We deduce that $\eps_{\SGIET}(f)=a \wedge b + a \wedge a + b \wedge a =a \wedge a$.
\end{proof}

\begin{figure}[ht]
\includegraphics[width=0.4\linewidth]{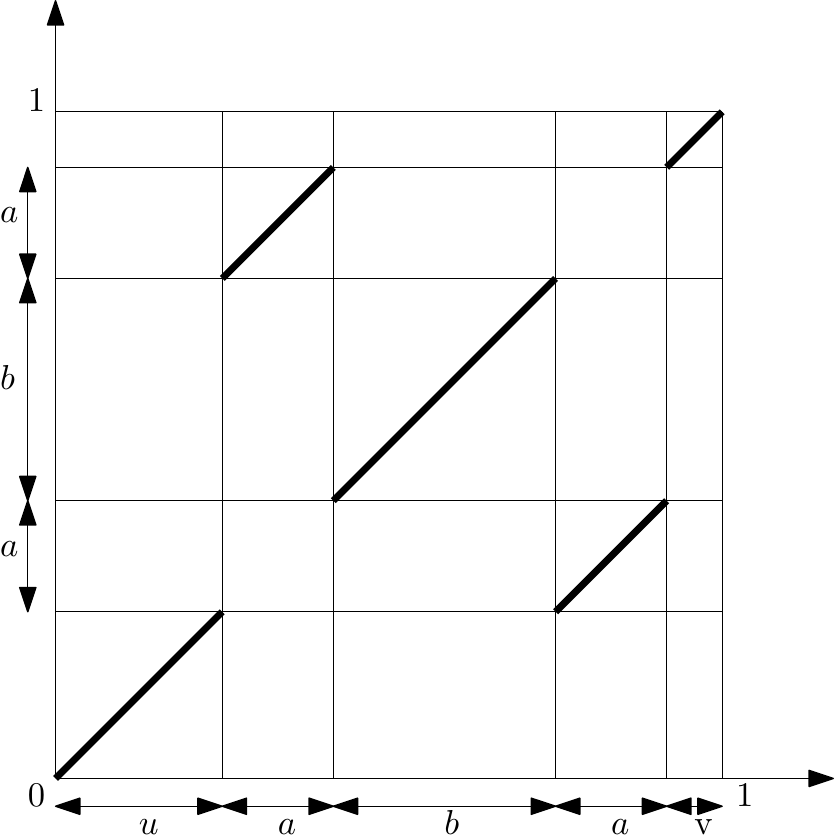} \hspace{20pt}
\includegraphics[width=0.4\linewidth]{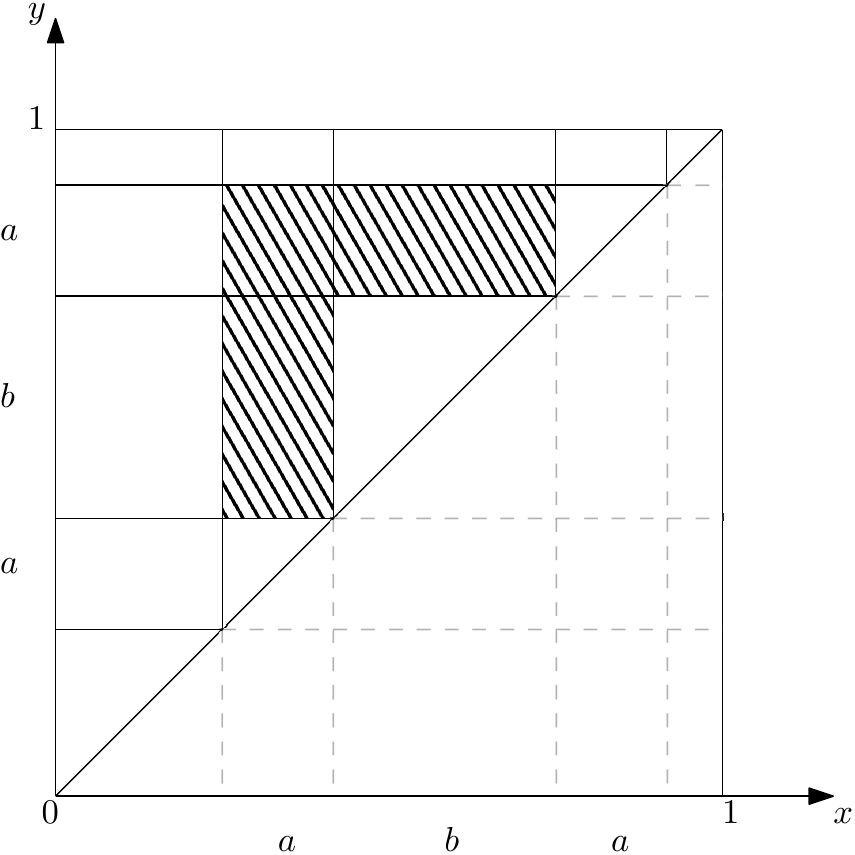}

\caption{\footnotesize{Illustration of how to calculate the value of $\eps_{\SGIET}$ on an $\IET(\SGIET)$-transposition $f$ in Proposition \ref{Proposition Value of eps_Gamma for a IET Gamma transposition}. \textbf{Left}: The graph of $f$. \textbf{Right}: The set $\cE_f$.}}\label{Exemple signature graphe de f}        
\end{figure}

\newpage

Thanks to Proposition \ref{Proposition cE_f in A_{Gamma}} we can calculate the value of $\eps_{\SGIET}$ for every $f \in \IET(\SGIET)$:

\begin{Prop}\label{Proposition signature morphism with combinatorial description}
Let $f \in \IET(\SGIET)$ and $(\alpha,\tau)$ be a $\SGIET$-combinatorial description of $f$. Let $n$ be the length of $\alpha$. Then
\[
\eps_{\SGIET}(f)=\sum\limits_{j=1}^n \sum\limits_{\substack{i<j \\ \tau(i) > \tau(j)}} \alpha_i \wedge \alpha_j
\]
\end{Prop}

\begin{Prop}\label{Proposition epsilon is surjective}
The group homomorphism $\eps_{\SGIET}$ is surjective.
\end{Prop}

\begin{proof}
Let $a,b \in \tilde{\SGIET}$. We assume that $0\leq a,b < 1$ and $0 \leq a+b \leq 1$. Let $r$ be the $\IET(\SGIET)$ restricted rotation of type $(a,b)$, whose intervals associated are $\interfo{0}{a}$ and $\interfo{a}{a+b}$. Then we obtain that $\eps_{\SGIET}(r)=a \wedge b$.

In the general case, let $w \in \tilde{\SGIET}$ with $0 \leq w \leq \frac{1}{2}$. Then there exist $k,\ell \in \ZZ$ and $a',b' \in \tilde{\SGIET}$ with $0 \leq a',b' \leq w$ such that $a=kw+a'$ and $b=\ell w+b'$. Then $a \wedge b=k \ell w \wedge w + k w \wedge b' + \ell a' \wedge w + a' \wedge b'$. By the previous case and as $\eps_{\SGIET}$ is a group homomorphism, we deduce that $a \wedge b$ is in $\textup{Im}(\eps_{\SGIET})$.
\end{proof}

\subsection{\texorpdfstring{Description of $\Ker(\eps_{\Gamma})$}{Description of Ker(epsilon index Gamma)}}~

\smallskip

The aim of this part is to conclude that $\Ker(\eps_{\SGIET})=D(\IET(\SGIET))$ and the induced morphism $\IET(\SGIET)_{\mathrm{ab}} \rightarrow \Wedgealt \tilde{\SGIET}$ is an isomorphism. We recall that $\varphi_{\SGIET}$ is the restriction of the $\SAF$-homomorphism $\varphi$ to $\IET(\SGIET)$.

\begin{Lem}\label{Lemma 2 epsilon = -phi}
We have $2\eps_{\SGIET}=-\varphi_{\SGIET}$.
\end{Lem}

\begin{proof}
Let $f \in \IET(\SGIET)$ and $(\alpha,\tau)$ be a $\SGIET$-combinatorial description of $f$. Let $n$ be the length of $\alpha$. Thanks to Propositions \ref{Proposition SAF invariant with combinatorial description} and \ref{Proposition signature morphism with combinatorial description} we have:
\begin{align*}
\varphi_{\SGIET}(f) 
=~
&\sum\limits_{j=1}^{n} \big( \sum\limits_{\substack{i \\ \tau(i)< \tau(j) }} \alpha_i - \sum\limits_{i<j}\alpha_i \big) \wedge \alpha_j
\\=~
&\sum\limits_{j=1}^{n} \sum\limits_{\substack{i>j \\ \tau(i)< \tau(j)}} \alpha_i \wedge \alpha_j ~+~ \sum\limits_{j=1}^{n} \big( \sum\limits_{\substack{i<j \\ \tau(i)< \tau(j)}} \alpha_i - \sum\limits_{i<j} \alpha_i \big) \wedge \alpha_j
\\=~
&\sum\limits_{i=1}^{n} \sum\limits_{\substack{j<i \\ \tau(j)>\tau(i)}} \alpha_i \wedge \alpha_j  ~+~ \sum\limits_{j=1}^{n} \sum\limits_{\substack{i<j \\ \tau(i)>\tau(j)}} \alpha_i \wedge \alpha_j
\\=~
&-2 \sum\limits_{i=1}^{n} \sum\limits_{\substack{j<i \\ \tau(j) > \tau(i)}} \alpha_j \wedge \alpha_i
\\=~
&-2\eps_{\SGIET}(f)
\end{align*}
\end{proof}

\begin{Coro}\label{Corollary kernel of epsilon is inside kernel of phi}
We have the inclusion $\Ker(\eps_{\SGIET}) \subset \Ker(\varphi_{\SGIET})$.
\end{Coro}

By Corollary \ref{Corollary Ker phi is generated by transpositions} we know that $\Ker(\varphi_{\SGIET})$ is generated by the set of all $\IET(\SGIET)$-transpositions. Thus it is natural to look at these elements who are also in $\Ker(\eps_{\SGIET})$. If $\sigma$ is an $\IET(\SGIET)$-transposition of type $a$ in $\Ker(\eps_{\SGIET})$ then we have the equality $a \wedge a =0$. We want to prove that $\sigma$ is in $D(\IET(\SGIET))$ if and only if $a \in 2 \tilde{\SGIET}$. 

We denote by $S^2 \tilde{\SGIET}$ the second symmetric power of $\tilde{\SGIET}$ and we denote by $a \odot a$ image of $a \otimes a$ in $S^2 \tilde{\SGIET}$. \\
For every group $G$ and every $w \in G$ we use the notation $w ~[\mathrm{mod}\ 2]$ for the image of $w$ in $G / 2G$.

\begin{Prop}\label{Proposition sufficient condition for dividing by 2 in Gamma}
Let $a \in \tilde{\SGIET}$, if $a \wedge a~[\mathrm{mod}\ 2]=0$ then $a$ belongs to $2\tilde{\SGIET}$.
\end{Prop}

\begin{proof}
For every group $G$, the group $\Wedgealt G / 2 (\Wedgealt G)$ is naturally isomorphic to the second symmetric power $S^2 G / 2 (S^2 G)$. This comes from the fact that these groups satisfy the following universal property: for every group $G$ and every abelian elementary $2$-group We denote $i$ the natural inclusion of $G \times G$ into $S^2 G / 2 (S^2 G)$. For every bilinear symmetric group homomorphism $b: G \times G \to A$ there exists a unique group homomorphism $f: S^2 G / 2 (S^2 G) \to A$ such that for every $g,h \in G$ we have $b(g,h)=f(i(g,h))$.

Let $a \in \tilde{\SGIET}$ with $a \neq 0$ (because  we already have $0=2 \times 0$). We denote by $a \odot a$ the image of $a \otimes a $ in $S^2 G / 2 (S^2 G)$, and we assume that $a \odot a ~[\textup{mod}~2]=0$. The projection $\tilde{\SGIET} \rightarrow \tilde{\SGIET} / 2\tilde{\SGIET}$ gives rise to a morphism $\zeta: S^2 \tilde{\SGIET} \rightarrow S^2 (\tilde{\SGIET} / 2\tilde{\SGIET})$. As $2 (S^2 \tilde{\SGIET}) \subset \Ker(\zeta)$ we obtain a morphism $\zeta ': S^2 \tilde{\SGIET} / 2 S^2 \tilde{\SGIET} \rightarrow S^2 (\tilde{\SGIET} / 2\tilde{\SGIET})$. Hence if $a \notin 2 \tilde{\SGIET}$ then $\zeta(a \odot a) \neq 0$ thus $\zeta ' (a \odot a~[\mathrm{mod}\ 2]) \neq 0$ which is a contradiction with the assumption. In conclusion $a \in 2 \tilde{\SGIET}$.
\end{proof}

\begin{Coro}\label{Corollary Ker(epsilon) case one transposition}
Every $\IET(\SGIET)$-transposition $f$ in $\Ker(\eps_{\SGIET})$ is in $D(\IET(\SGIET))$.
\end{Coro}

\begin{proof}
Let $a \in \tilde{\SGIET}_+$, with $a \leq \frac{1}{2}$, be the type of $f$ and let $u,v \in \tilde{\SGIET}_+$ such that $I_1=\mathopen{[}u,u+a \mathclose{[}$ and $I_2=\mathopen{[}v,v+a \mathclose{[}$ are the two intervals swapped by $f$. From $f \in \Ker(\eps)$ we deduce that $\eps_{\SGIET}(f)= a \wedge a=0$. Hence $a \wedge a ~[\mathrm{mod}\ 2]=0$ in $\Wedgealt \tilde{\SGIET} / 2 (\Wedgealt \tilde{\SGIET})$. Then by Proposition \ref{Proposition sufficient condition for dividing by 2 in Gamma} there exists $b \in \tilde{\SGIET}$ such that $a=2b$. Thus if we define $g$ as the unique $\IET(\SGIET)$-transposition of type $b$ that swaps $\mathopen{[}u,u+b \mathclose{[}$ and $\mathopen{[}v,v+b \mathclose{[}$ and $h$ as the unique $\IET(\SGIET)$ that swaps $\mathopen{[}u,u+b \mathclose{[}$ with $\mathopen{[}u+b,u+a \mathclose{[}$ and $\mathopen{[}v,v+b \mathclose{[}$ with $\mathopen{[}v+b,v+a \mathclose{[}$. Then $f=ghgh$ (see Figure \ref{Figure IET transposition is inside D(IET)}) and $g^2=h^2=\Id$, hence $f \in D(\IET(\SGIET))$.
\end{proof}

In order to show that $\Ker(\eps_{\SGIET})=D(\IET(\SGIET))$ we prove that any element $f \in \Ker(\eps_{\SGIET})$ can be written as $f=\sigma h$ where $h \in D(\IET(\SGIET))$ and $\sigma$ is an $\IET(\SGIET)$-transposition. This concludes because we just show that an $\IET(\SGIET)$-transposition which is also in $\Ker(\eps_{\SGIET})$ is in $D(\IET(\SGIET))$. We begin by a particular case of a product of $\IET(\SGIET)$-transpositions with pairwise disjoint support. The aim will be to reduce the general case to this one. We recall that the identity is considered as an $\IET(\SGIET)$-transposition.

\begin{Lem}\label{Lemma Ker(epsilon) case non overlapping}
Let $k \in \NN$ and $\tau_1, \tau_2, \ldots , \tau_k$ be $\IET(\SGIET)$-transpositions with pairwise disjoint support. Then $\tau_1\tau_2 \ldots \tau_k = \sigma h$ where $\sigma$ is an $\IET(\SGIET)$-transposition and $h$ is an element of $D(\IET(\SGIET))$. Moreover the support of $h$ and $\sigma$ do not overlap and are included in the union of the supports of the $\tau_i$.
\end{Lem}

\begin{proof}
By induction it is enough to show the result in the case $k=2$
Let respectively $a_1$ and $a_2$ be the type of $\tau_1$ and $\tau_2$. As their support do not overlap we know that $\tau_1$ and $\tau_2$ commute. Hence we can assume $a_1 \geq a_2$ without loss of generality. If $a_1=a_2$ then there exists $f \in \IET(\SGIET)$ such that $\tau_1=f \tau_2 f^{-1}$. Thus $\tau_1 \tau_2$ is in $D(\IET(\SGIET))$. If $a_1 >a_2$ let $u,v \in \tilde{\SGIET}$ such that $\mathopen{[}u,u+a_1\mathclose{[}$ and $\mathopen{[}v,v+a_1\mathclose{[}$ are the intervals swapped by $\tau_1$. Let $g$ and $h$ be $\IET(\SGIET)$-transpositions such that $g$ swaps the intervals $\mathopen{[}u,u+a_2\mathclose{[}$ and $\mathopen{[}v,v+a_2\mathclose{[}$ and $h$ swaps the intervals $\mathopen{[}u+a_2,u+a_1\mathclose{[}$ and $\mathopen{[}v+a_2,v+a_1\mathclose{[}$. Thus $\tau_1=hg$. Moreover $g$ and $\tau_2$ are two $\IET(\SGIET)$-transpositions with same type and non-overlapping support. Then by the previous case, the product $g\tau_2$ is in $D(\IET(\SGIET))$ and its support does not intersect the support of $h$. Then $f=h(g\tau_2)$ is the wanted decomposition.
\end{proof}

We can now treat the case of finite order elements:

\begin{Lem}\label{Lemma finite order implies transposition compose with an element of derived subgroup}
If $f \in \IET(\SGIET)$ has finite order then there exist an $\IET(\SGIET)$-transposition $\sigma$ and $h \in D(\IET(\SGIET))$ such that the support of $\sigma $ and $h$ are inside the support of $f$ and $f=\sigma h$.
\end{Lem}

\begin{proof}
Let $f \in \IET(\SGIET)$ as in the statement and let $n \in \NN$ and $\cP=\lbrace I_1, I_2, \ldots , I_n \rbrace \in \Pi_f$. By Lemma \ref{Lemma finite order implies being a product of IET transpositions}   there exists $\sigma \in \mfS_n$ such that $f(I_i)=I_{\sigma(i)}$. Let $k \in \NN$ and $\sigma=c_1 c_2 \ldots c_k$ be the disjoint cycle decomposition for $\sigma$. Let $f_i$ be the element of $\IET(\SGIET)$ that is equal to $f$ on $I_j$ for every $j \in \support(c_i)$ while fixing the rest of $\interfo{0}{1}$. Then $f=f_1 \ldots f_k$ and $f_i$ commutes with $f_j$ for every $1 \leq i \neq j \leq k$. Then if the statement is true for every $f_i$ we can write $f_i=\tau_ih_i$ with $\tau_i$ an $\IET(\SGIET)$-transposition and $h \in D(\IET(\SGIET))$, both of them with support inside the support of $f_i$. Then $f=f_1\ldots f_k=\tau_1h_1 \ldots \tau_k h_k=\tau_1 \ldots \tau_k h_1 \ldots h_k$ because the support of $\tau_i$ does not overlap with the support of $\tau_j$ of $h_j$ for every $1 \leq j \leq k$ and $j \neq i$. We conclude with Lemma \ref{Lemma Ker(epsilon) case non overlapping} applied to $\tau_1 \ldots \tau_k$.\\
Let $c$ by a cycle of length $n \geq 2$ and let $I_1, I_2, \ldots , I_n$ be non-overlapping intervals of $\Itv(\SGIET)$ of same length. Let $f \in \IET(\SGIET)$ be the element that permutes the set $\lbrace I_1, I_2, \ldots, I_n \rbrace $ by $c$. Then if $c \in D(\mfS_n)=\mfA_n$ we deduce that $f \in D(\IET(\SGIET))$. If $c \notin D(\mfS_n)$ then let $g$ be the unique $\IET(\SGIET)$-transposition that swaps $I_1$ with $I_{c(1)}$ (we notice that the support of $g$ is included in the support of $f$). By the previous case, $gf \in D(\IET(\SGIET))$ and we conclude that $f=g(gf)$ is a wanted decomposition.
\end{proof}

\begin{Coro}\label{Corollary Ker(epsilon) case 2 elements}
Let $\tau$ and $\tau '$ be two $\IET(\SGIET)$-transpositions. There exist an $\IET(\SGIET)$-transposition $\sigma$ and $h \in D(\IET(\SGIET))$ such that $\tau \tau '=\sigma h$ and the support of $\sigma $ and $h$ are included in the union of the support of $\tau$ and the one of $\tau'$.
\end{Coro}

\begin{proof}
By Proposition \ref{Lemma product of two transpositions has finite order} we deduce that $f:=\tau \tau '$ has finite order. Hence by Lemma \ref{Lemma finite order implies transposition compose with an element of derived subgroup} we obtain the result.
\end{proof}

\begin{Lem}\label{Lemma Ker(epsilon) decomposition k elements}
Let $k$ in $\NN$ and $\tau_1, \tau_2, \ldots , \tau_k$ be some $\IET(\SGIET)$-transpositions. Then there exist an $\IET(\SGIET)$-transposition $\sigma$ and $h \in D(\IET(\SGIET))$ such that $\tau_1\tau_2 \ldots \tau_k = \sigma h$.
\end{Lem}

\begin{proof}
The proof is by induction on $k$. The initialisation $k=1$ is immediate. The case $k=2$ is Corollary \ref{Corollary Ker(epsilon) case 2 elements}. Now if we assume the result for $k \geq 2$ let $\tau_1, \tau_2, \ldots , \tau_k, \tau_{k+1}$ be $\IET(\SGIET)$-transpositions. Then by assumption, applied to $\tau_2\tau_3 \ldots \tau_{k+1}$, there exist an $\IET(\SGIET)$-transposition $\sigma$ and $h \in D(\IET(\SGIET))$ such that $\tau_2\tau_3 \ldots \tau_{k+1} = \sigma h$. Hence $\tau_1 \tau_2 \ldots \tau_{k+1}=\tau_1 \sigma h$. By using the case $k=2$ we deduce that there exist an $\IET(\SGIET)$-transposition $\sigma '$ and $h' \in D(\IET(\SGIET))$ such that $\tau_1 \sigma = \sigma ' h'$. Thus $\tau_1 \tau_2 \ldots \tau_{k+1} = \sigma' g$, with $g=h'h \in D(\IET(\SGIET))$, which is a wanted decomposition.
\end{proof}

Finally we can prove the main theorem of the section:

\begin{Thm}\label{Theorem D(IET(Gamma)) is equal to the Kernel of Eps Gamma}
We have the equality $\Ker(\eps_{\SGIET}) = D(\IET(\SGIET))$, and the induced morphism $\IET(\SGIET)_{\mathrm{ab}} \rightarrow \Wedgealt \tilde{\SGIET}$ is an isomorphism.
\end{Thm}

\begin{proof}
The inclusion from right to left is immediate. For the other inclusion let $f \in \Ker(\eps_{\SGIET})$. By Corollary \ref{Corollary kernel of epsilon is inside kernel of phi} we know that $f \in \Ker(\varphi_{\SGIET})$, then by Corollary \ref{Lemma ker(phi) is generated by balanced product} there exists $k \in \NN$ such that $f$ is equal to the product $\tau_1 \tau_2 \ldots \tau_k$ where $\tau_i$ is an $\IET(\SGIET)$-transposition. By Lemma \ref{Lemma Ker(epsilon) decomposition k elements} there exist an $\IET(\SGIET)$-transposition $\sigma$ and $h \in D(\IET(\SGIET))$ such that $f = \sigma h$. Then $\eps(f)=\eps(\sigma) \eps(h)=\eps(\sigma)=0$. By Corollary \ref{Corollary Ker(epsilon) case one transposition} we deduce $\sigma \in D(\IET(\SGIET))$. Hence $f=\sigma h \in D(\IET(\SGIET))$.

We deduce that the induced group homomorphism $\eps'_{\SGIET}: \IET(\SGIET)_{\mathrm{ab}} \rightarrow \Wedgealt \tilde{\SGIET}$ is injective. Furthermore $\eps_{\SGIET}$ is surjective by Proposition \ref{Proposition epsilon is surjective} thus $\eps'_{\SGIET}$ is surjective and we conclude that $\eps'_{\SGIET}$ is an isomorphism.
\end{proof}

\bigskip

\section{\texorpdfstring{Abelianization of $\IET^{\bowtie}(\Gamma)$}{Abelianization of IET bowtie (Gamma)}}\label{Section Abelianization of IET ^ bowtie (Gamma)}

In order to identify $\IET^{\bowtie}(\SGIET)_{\mathrm{ab}}$ we construct two group homomorphisms. One is an analogue of the signature homomorphism $\eps_{\SGIET}$ constructed in Section \ref{Section Description of the abelianization of IET(Gamma)}. Its kernel will be slightly larger than $D(\IET^{\bowtie}(\SGIET))$; in fact, it will miss some reflections of a certain type. In order to tackle this issue we notice that, in some cases, a reflection is a conjugate of a restricted rotation and thus we will try to use the group homomorphism $\eps_{\SGIET}$. To this we need to affect for every element of $\IET^{\bowtie}(\SGIET)$ an element of $\IET(\SGIET)$.

We recall that the group $\IET$ can be seen as a subgroup of $\IET^{\bowtie}$. Thus restricted rotations in $\IET^{\bowtie}$ are well-defined.

The next proposition implies that a lot of $\SGIET$-reflections are in $D(\IET^{\bowtie}(\SGIET))$ with a sufficient condition on their type (see Definition \ref{Definition reflection map and reflection} for the definition of reflection and type of a reflection). It is a consequence of Proposition \ref{Proposition reflections of type 2a are conjugated to restricted rotation of type (a,a)} and Theorem \ref{Theorem D(IET(Gamma)) is equal to the Kernel of Eps Gamma}:

\begin{Prop}\label{Proposition a Gamma reflection with type in 4Gamma is in D(IET bowtie(Gamma))}
Every $\SGIET$-reflection of type $\ell \in 4 \tilde{\SGIET}$ belongs to $D(\IET^{\bowtie}(\SGIET))$
\end{Prop}

\subsection{Partitions associated and positive substitute}~
\smallskip

\subsubsection{Partitions}~
\smallskip

Partitions of $\interfo{0}{1}$ into finite intervals will be really useful there but the definition of $\IET^{\bowtie}(\SGIET)$ imposes to be more careful and to take them up to a finitely supported permutation.

\begin{Def}
Let $\cP$ be a finite partition of $\interfo{0}{1}$ into right-open and left-closed intervals and let $\widehat{f} \in \widehat{\IET^{\bowtie}}$. The partition $\cP$ is called a \textit{partition associated with $\widehat{f}$} if $\widehat{f}$ is continuous on every $I \in \cP$. Also $\cP$ is an \textit{essential partition associated with $\widehat{f}$} if there exists a finitely supported permutation $\sigma$ such that $\cP$ is a partition associated with $\sigma \widehat{f}$. We denote by $\widehat{f}(\cP)$ the partition into intervals of $\interfo{0}{1}$ composed of all right-open and left-closed intervals whose interior is the image by $\widehat{f}$ of the interior of an interval in $\cP$ up to a finite number of points; it is called the \textit{arrival partition of $\widehat{f}$ associated to $\cP$}.
\end{Def}

We can now define partitions associated with an element of $\IET^{\bowtie}$:

\begin{Def}
Let $\cP$ be a finite partition of $\interfo{0}{1}$ into right-open and left-closed intervals. Let $f \in \IET^{\bowtie}$ and let $\widehat{f}$ be a representative of $f$ in $\widehat{\IET^{\bowtie}}$. The partition $\cP$ is called a \textit{partition associated with $f$} if $\cP$ is an essential partition associated to $\widehat{f}$. We denote by $f(\cP)$ the partition $\widehat{f}(\cP)$;  it is called the \textit{arrival partition of $f$ associated to $\cP$}.
\end{Def}

\begin{Rem}
For every $f \in \IET^{\bowtie}$, the set $\Pi_f$ has a minimal element for the refinement. This element is also the unique partition that has a minimal number of interval.
\end{Rem}

We recall that a partition $\cP$ is a $\SGIET$-partition if for every $I \in \cP$ we have $I \in \Itv(\SGIET)$. Thus for every $f \in \IET^{\bowtie}$ we have $f \in \IET^{\bowtie}(\SGIET)$ if and only if there exists a $\SGIET$-partition associated with $f$.

Sometimes we will need to be more precise on where the length of intervals are. We recall that $\lambda$ denote the Lebesgue measure:

\begin{Def}
Let $S$ be a finite subset of $\RR$, we say that a subinterval $I$ of $\RR$ is a \textit{$S$-interval} if $\lambda(I) \in S$.
\end{Def}

In order to have more rigidity we also associate a partition to a tuple:

\begin{Def}
Let $n \in \NN$ and $f_1,f_2, \ldots, f_n \in \IET^{\bowtie}(\SGIET)$. Let $S$ be a finite subset of $\RR$ and let $\cP$ be a partition into intervals of $\mathopen{[} 0,1 \mathclose{[}$, we said that $\cP$ is a \textit{partition into $S$-intervals associated with $(f_1,f_2,\ldots,f_n)$} if:
\begin{enumerate}
\item $\cP$ is a partition into $S$-intervals associated with $f_1$,
\item For every $2 \leq i \leq n-1$, $f_if_{i-1} \ldots f_1 (\cP)$ is a partition into $S$-intervals associated with $f_{i+1}$.
\end{enumerate}
\end{Def}

\begin{Rem}\label{Remark refinement of a partition associated to a tuple is still a partition associated to the tuple}
Let $S$ and $T$ be two finite subsets of $\RR$ and let $\cP$ be a partition into $S$-intervals associated with $(f_1,f_2,\ldots , f_n)$. Then any refinement of $\cP$ into $T$-intervals is a partition into $T$-intervals associated with $(f_1,f_2,\ldots , f_n)$.
\end{Rem}

We also want to talk about order-preserving and order-reversing for elements in $\IET^{\bowtie}$. 

\begin{Def}
Let $f \in \IET^{\bowtie}$ and $\cP \in \Pi_f$. Let $I \in \cP$, we say that $f$ is \textit{order-preserving on $I$} (resp.\ \textit{order-reversing on $I$}) if there exists a representative of $f$ in $\widehat{\IET^{\bowtie}}$ that is order-preserving on $I$ (resp.\ order-reversing on $I$).
\end{Def}

Thanks to this when we have a partition $\cP$ associated with an element $f \in \IET^{\bowtie}$ we can always say that $f$ is either order-preserving or order-reversing on every interval of $\cP$.

\subsubsection{Positive substitute}~
\smallskip

We introduce the notion of positive substitute to affect an element of $\IET$ for every element of $\IET^{\bowtie}$. This substitute depends on a partition associated with $f$.

\begin{Def}
Let $\widehat{f} \in \widehat{\IET^{\bowtie}}$ and $\cP$ be an essential partition associated with $f$. Let $f$ be its image in $\IET^{\bowtie}$. Then there exists a unique $f_{\cP}^+ \in \IET$ such that $\cP$ is a partition associated with $f_{\cP}^+$ and such that for every $I \in \cP$ we have $f_{\cP}^+(I)=\widehat{f}(I)$ up to a finite number of points. This element is called the \textit{positive $\cP$-substitute of $\widehat{f}$ and $f$}.
\end{Def}

This element exists because there exists a finitely supported permutation $\sigma$ such that $\cP$ is a partition associated with $\sigma \widehat{f}$ and we define $f_{\cP}^+$ as the element of $\IET$ such that $f_{\cP}^+(I)=\sigma f(I)$ for every $I \in \cP$. The unicity come from the right continuity of $f_{\cP}^+$. The dependance on the partition is really important as we can see on Figures \ref{Figure positive substitute_1} and \ref{Figure positive substitute_2}.

\begin{figure}[ht]
\includegraphics[width=0.8\linewidth]{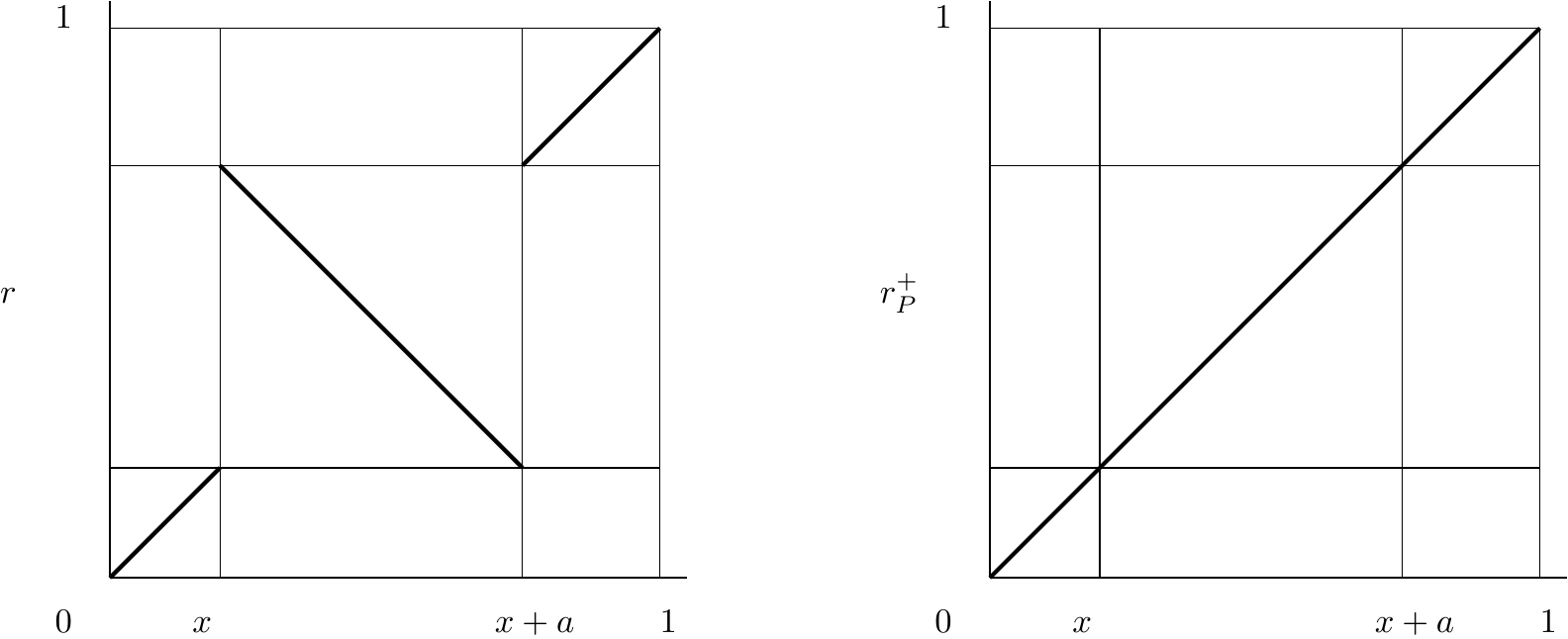}
\caption{\footnotesize{Positive substitute for a reflection in the case $S=\lbrace x,a,1-(x+a) \rbrace$.}}\label{Figure positive substitute_1}
\end{figure}

\begin{figure}[ht]
\includegraphics[width=0.8\linewidth]{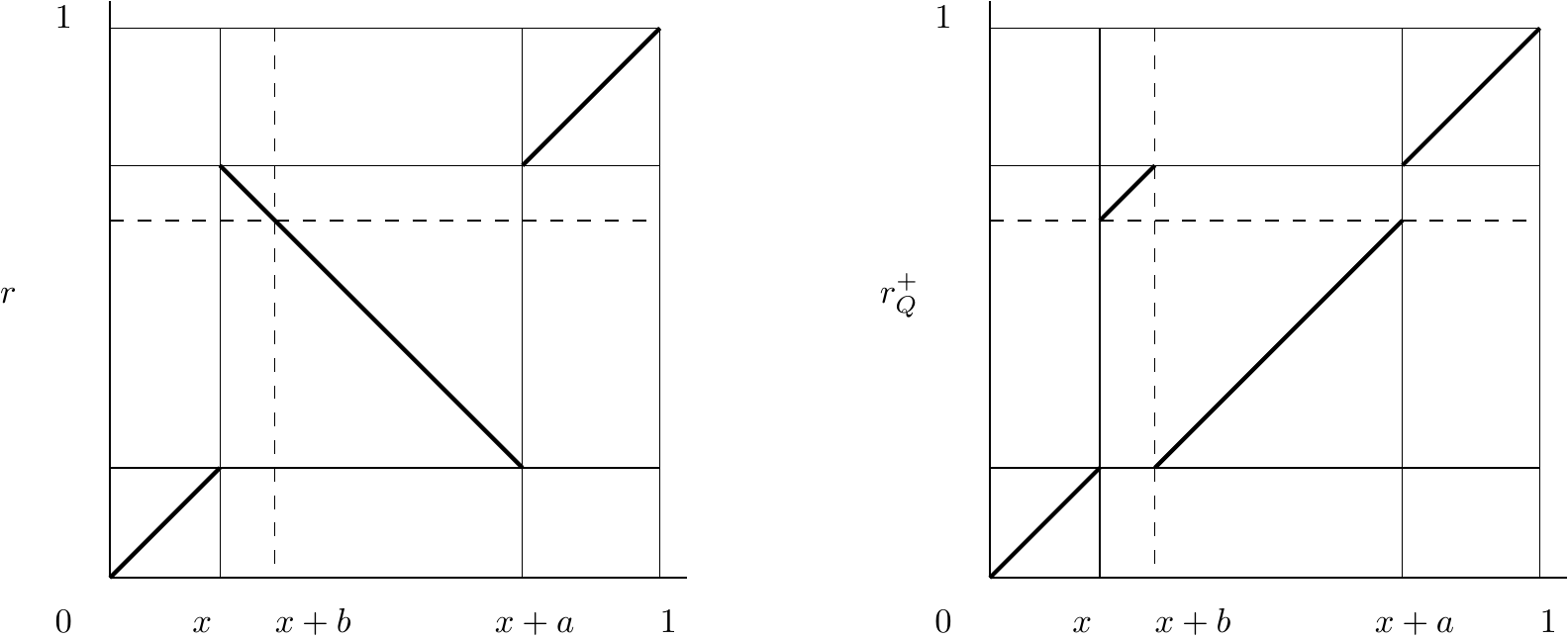}
\caption{\footnotesize{Positive substitute for a reflection in the case $S=\lbrace x,b,a-b,1-(x+a) \rbrace$.}}\label{Figure positive substitute_2}
\end{figure}

\newpage

This gives us a decomposition for every element of $\widehat{\IET^{\bowtie}}$.

\begin{Prop}\label{Proposition decomposition in widehat IET bowtie}
Let $\widehat{f} \in \widehat{\IET^{\bowtie}}$ and $\cP$ be an essential partition associated with $\widehat{f}$. Then there exist a unique finitely supported permutation $\sigma$ and a product $r$ of $I$-reflection maps with $I \in \widehat{f}(\cP)$ such that $\widehat{f}=\sigma r f_{\cP}^{+}$.
\end{Prop}

With some conditions we deduce how positive substitute behaves with composition.

\begin{Prop}\label{Proposition behaviour of positive substitute of a composition}
Let $n \in \NN$ and $f_1,f_2,\ldots,f_n \in \IET^{\bowtie}(\SGIET)$. Let $\cP$ be a partition into $S$-intervals associated with $(f_1,f_2,\ldots,f_n)$. Let $g_1 = (f_1)^+_{\cP}$ and $g_i=(f_if_{i-1}\ldots f_1)^+_{f_{i-1}\ldots f_1 (\cP)}$ for every $2 \leq i \leq n$. Then $(f_n f_{n-1} \ldots f_1)^+_{\cP}=g_n g_{n-1} \ldots g_1$.
\end{Prop}

\begin{proof}
For this proof we denote by $r_I$ the $I$-reflection for every subinterval $I$ of $\mathopen{[}0,1 \mathclose{[}$.\\
By iteration it is sufficient to show the result for $n=2$. Let $f,g \in \IET^{\bowtie}(\SGIET)$ and $\cP$ be a partition into $S$-intervals associated with $(f,g)$. Let $n \in \NN$ such that $\cP=\lbrace I_1,I_2, \ldots ,I_n \rbrace$ and $f(\cP)=\lbrace J_1,J_2, \ldots ,J_n \rbrace$ and $f_{\cP}^+ (I_i)=J_i$ for $1 \leq i \leq n$. Let $1 \leq i \leq n$, we notice that $r_{J_i} \circ f^+_{\cP}|_{I_i} \circ r_{I_i}=f^+_{\cP}|_{I_i}$. There are $4$ cases:
\begin{enumerate}
\item If $f$ is order-preserving on $I_i$ and $g$ is order-preserving on $J_i$ then $g \circ f$ is order-preserving on $I_i$ so :
\[
(g \circ f)^+_{\cP}|_{I_i}=(g \circ f)|_{I_i}=g|_{J_i} \circ f|_{I_i}=g^+_{f(\cP)}|_{J_i} \circ f^+_{\cP}|_{I_i}
\]
\item If $f$ is order-preserving on $I_i$ and $g$ is order-reversing on $J_i$ then $g \circ f$ is order-reversing on $I_i$ so :
\[
(g \circ f )^+_{\cP}|_{I_i} = (g \circ f)_{I_i} \circ r_{I_i}= g|_{J_i} \circ f|_{I_i} \circ r_{I_i} =g^+_{f(\cP)}|_{J_i} \circ r_{J_i} \circ f^+_{\cP}|_{I_i} \circ r_{I_i}= g^+_{f(\cP)}|_{J_i} \circ f^+_{\cP}|_{I_i}
\]
\item If $f$ is order-reversing on $I_i$ and $g$ is order-preserving on $J_i$ then $g \circ f$ is order-reversing on $I_i$ so :
\[
(g \circ f )^+_{\cP}|_{I_i} = (g \circ f)_{I_i} \circ r_{I_i}= g|_{J_i} \circ f|_{I_i} \circ r_{I_i} =g^+_{f(\cP)}|_{J_i} \circ f^+_{\cP}|_{I_i} \circ r_{I_i} \circ r_{I_i}= g^+_{f(\cP)}|_{J_i} \circ f^+_{\cP}|_{I_i}
\]
\item If $f$ is order-reversing on $I_i$ and $g$ is order-reversing on $J_i$ then $g \circ f$ is order-preserving on $I_i$ so :
\[
(g \circ f )^+_{\cP}|_{I_i} = (g \circ f)|_{I_i}= g|_{J_i} \circ f|_{I_i}=g^+_{f(\cP)}|_{J_i} \circ r_{J_i} \circ f^+_{\cP}|_{I_i} \circ r_{I_i}= g^+_{f(\cP)}|_{J_i} \circ f^+_{\cP}|_{I_i}
\]
\end{enumerate}
\end{proof}

\subsection{Analogue of the signature}

\subsubsection{Balanced product of reflections}~
\smallskip

We give here a first description of $D(\IET^{\bowtie})$. It is inspired of the work done for $\IET$ with balanced product restricted rotations.

\begin{Def}
Let $n \in \NN$ and $r_1,r_2, \ldots, r_n$ be some $\SGIET$-reflections. For every $\ell \in \tilde{\SGIET}_+$ let $n_{\ell}$ be the number of $\SGIET$-reflections of type $\ell$ among these elements The tuple $(r_1,r_2,\ldots, r_n)$ is a \textit{balanced tuple of $\SGIET$-reflections} if $2$ divides $n_{\ell}$ for every $\ell \in \tilde{\SGIET}_+$. We say that a product of $\SGIET$-reflections is a \textit{a balanced product of $\SGIET$-reflections} if it can be written as a product of a balanced tuple of $\SGIET$-reflections. 
\end{Def}

\begin{Lem}\label{Lemma D(IET bowtie Gamma) is generated by balanced product of reflections}
The set of all balanced products of $\SGIET$-reflections is a generating subset of $D(\IET^{\bowtie}(\SGIET))$.
\end{Lem}

\begin{proof}
As any element of $\IET^{\bowtie}(\SGIET)$ is a finite product of $\SGIET$-reflections (see Proposition \ref{Proposition Generating set for IET bowtie SGIET}) and as a reflection has order $2$ we deduce that every element of $D(\IET^{\bowtie}(\SGIET))$ is a balanced product of reflections.\\
Let $r$ and $s$ be two reflections with the same type. By Proposition \ref{Proposition reflections of same type are conjugated} we deduce that they are conjugate. As a reflection has order $2$, the product $rs$ is a commutator.\\
Let $n \in \NN$ and $r_1,r_2,\ldots r_n$ be $\SGIET$-reflections such that $r_1r_2 \ldots r_n$ is a balanced product of reflections. Then $n$ is even and up to compose with an element of $D(\IET^{\bowtie})$ we can assume that $r_{2i-1}$ and $r_{2i}$ have the same type for every $1 \leq i \leq \frac{n}{2}$. Thus by the previous case we deduce that $r_1r_2 \ldots r_n$ is in $D(\IET^{\bowtie})$.
\end{proof}

By Proposition \ref{Proposition Generating set for IET bowtie SGIET} and Lemma \ref{Lemma D(IET bowtie Gamma) is generated by balanced product of reflections} we deduce the following:
\begin{Coro}\label{Corollary squares of an element are in D(IET^bowtie(Gamma))}
Let $f \in \IET^{\bowtie}(\SGIET)$ then $f^2 \in D(\IET^{\bowtie}(\SGIET))$.
\end{Coro}

\begin{Prop}\label{Proposition balanced product of restricted rotations is a balanced product of reflections}
Any balanced product of $\SGIET$-restricted rotations with type in $\lbrace (a,b) \mid a \neq b \in \tilde{\SGIET}_+ \rbrace$ is a balanced product of $\SGIET$-reflections.
\end{Prop}

\begin{proof}
Let $a,b \in \tilde{\SGIET}_+$. Let $r$ be a restricted rotation of type $(a,b)$ an $s$ be a restricted rotation of type $(b,a)$. Let $I$ and $J$ be the two consecutive intervals associated with $r$. Then $r$ is the composition of the $I$-reflection, the $J$-reflection and the $I\cup J$-reflection thus a product of a reflections of type $a$ with one of type $b$ and one of type $a+b$. The same is true for $s$ thus we obtain that $rs$ is the product of two reflections of type $a$, two of type $b$ and two of type $a+b$; so $rs$ is a balanced product of $\SGIET$-reflections.
\end{proof}

\subsubsection{The group homomorphism}~
\smallskip

Here we start with the work done in Section \ref{Section Description of the abelianization of IET(Gamma)}. We remark that if we denote $A'_{\RR}$ the Boolean algebra of subsets of $\interfo{0}{1}$ generated by the set of all intervals $\mathopen{[}a,b \mathclose{]}$, $A_{\RR}$ the one generated by the set of all intervals $\interfo{a}{b}$ and $A_{\mathrm{fin}}$ the one generated by all the singletons $\lbrace x \rbrace$ then $A_{\RR}$ is isomorphic to $A_{\RR}^*:=A'_{\RR} / A_{\mathrm{fin}}$. This is why we do not make a difference between $A_{\SGIET}$ and its image in $A_{\RR}^*$.

 The notion of inversions as defined in \ref{Definition inversions} is no longer relevant because for every reflection $r$ we have $\cE_{r} \notin A_{\SGIET} \otimes A_{\SGIET}$; indeed $\cE_{r}$ is a triangle. We need to be more precise:

\begin{Def}\label{Definition inversions in IET bowtie}
For every element $f \in \widehat{\IET^{\bowtie}(\SGIET)}$ we define:
\begin{enumerate}
\item $\cE_{f,1}:= \lbrace (x,y) \mid x<y,~f(x)>f(y) \rbrace$, the set of all inversions of type 1 of $f$,
\item $\cE_{f,2}:= \lbrace (x,y) \mid y<x,~f(y)>f(x) \rbrace$, the set of all inversions of type 2 of $f$,
\item $\cE_{f}^{\bowtie}:= \cE_{f,1} \cup \cE_{f,2}$, the set of all inversions of $f$.
\end{enumerate}
\end{Def}

From now on we will write $\cE_{f}^{\bowtie}=\cE_{f}$. 

\begin{Exem}
\begin{enumerate}
\item If we consider a $\SGIET$-restricted rotation $r$ of type $(a,b)$ (with $a,b \in \tilde{\SGIET}_+$), we obtain $w_{\SGIET}(\cE_{r})=a \otimes b + b \otimes a$ (see Figure \ref{Figure Set of pair inverse for a restricted rotation and a reflection}).

\item Let $I \in \Itv(\SGIET)$ be a subinterval of $\mathopen{[} 0,1 \mathclose{[}$ of length $a$. Let $r$ be the $I$-reflection. We have $w_{\SGIET}(\cE_{r})= a \otimes a$ (see Figure \ref{Figure Set of pair inverse for a restricted rotation and a reflection}).

\item Let $\tau$ a finitely supported permutation. Then $\cE_{\tau}$ is a union of singletons thus $\cE_{\tau} = \emptyset \in A_{\SGIET} \otimes A_{\SGIET}$ and $w_{\SGIET}(\cE_{\tau})=w_{\SGIET}(\emptyset)=0$.

\end{enumerate}
\begin{figure}[ht]
\includegraphics[width=0.45\linewidth]{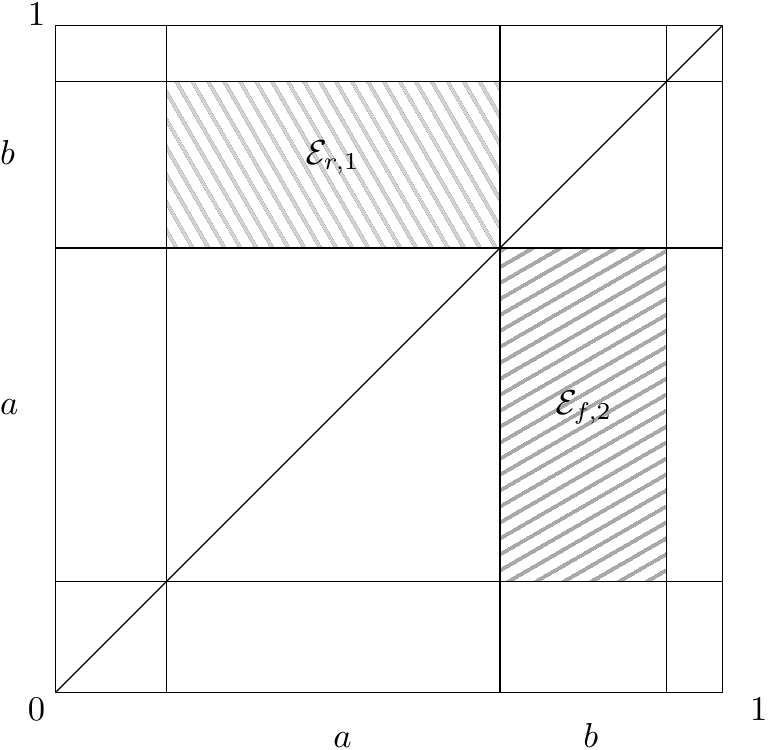}
\hspace{20pt}
\includegraphics[width=0.45\linewidth]{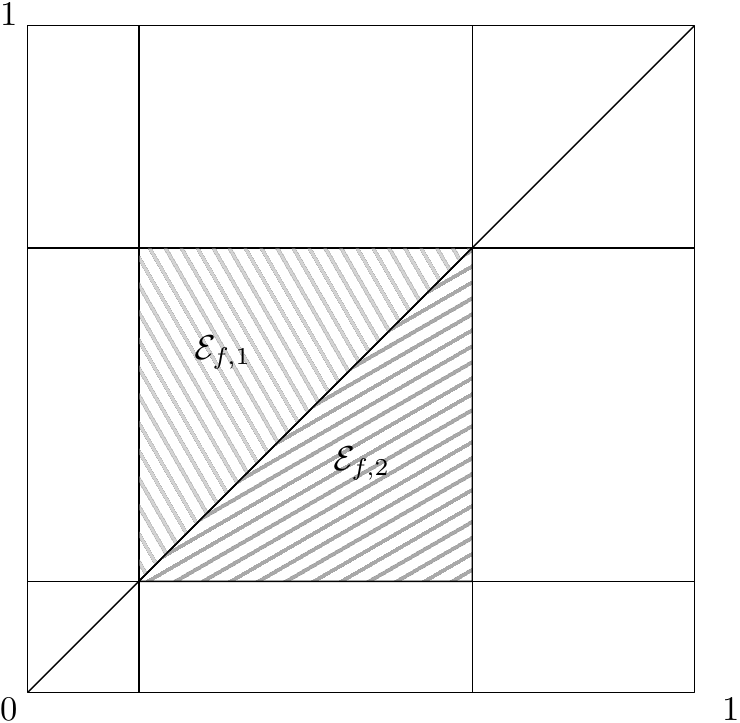}
\caption{\footnotesize{\textbf{Left}: Set of inversions for a restricted rotation. \textbf{Right}: Set of inversions for a $I$-reflection}}\label{Figure Set of pair inverse for a restricted rotation and a reflection}
\end{figure}

\end{Exem}

\begin{Prop}
For every $f \in \widehat{\IET^{\bowtie}(\SGIET)}$ we have $\cE_{f} \in A_{\SGIET} \otimes A_{\SGIET}$.
\end{Prop}

\begin{proof}
Let $f \in \widehat{\IET^{\bowtie}}$ and let $\cP$ be a partition associated with $f$. Let $f_{\cP}^+$ be the positive $\cP$-substitute of $f$. Let $\cJ \subset \cP$ be the subset of intervals where $f$ is order-reversing. By seeing $A_{\SGIET}$ as its image in $A_{\RR}^*$ then we deduce that $\cE_f=\cE_{f_{\cP}^{+}} \cup \bigcup\limits_{I \in \cJ} I \times I$ is an element of $A_{\SGIET} \otimes A_{\SGIET}$.
\end{proof}

Then this set can be measured with the same measure $\omega_{\SGIET}$ used in Section \ref{Section Description of the abelianization of IET(Gamma)}. Here we expect a $2$-group for the image of our group homomorphism. We denote by $\pi$ the projection of $\bigotimes^2_{\ZZ} \tilde{\SGIET}$ onto $\bigotimes^2_{\ZZ} \tilde{\SGIET} / 2(\bigotimes^2_{\ZZ} \tilde{\SGIET})$ and for every $a \in \tilde{\SGIET}$ we denote $\pi(a \otimes a)$ by $a \otimes a ~[\textup{mod}~2]$.

\begin{Def}
We define \textit{the signature for $\widehat{\IET^{\bowtie}(\SGIET)}$} as the map:
\[
\begin{array}{cccc}
\widehat{\eps_{\SGIET}^{\bowtie}}: & \widehat{\IET^{\bowtie}(\SGIET)} & \longrightarrow &  \bigotimes^2_{\ZZ} \tilde{\SGIET} / 2(\bigotimes^2_{\ZZ} \tilde{\SGIET})\\
 & f & \longrightarrow & w_{\SGIET}(\cE_f) \modulo{2}
\end{array}
\]
\end{Def}

For more clarity we explicit some equalities used to show that $\widehat{\eps_{\SGIET}^{\bowtie}}$ is a group homomorphism.

\begin{Lem}\label{Interactions des ensemble E(fg), E(f) et E(g)}
Let $f,g \in \widehat{\IET^{\bowtie}(\SGIET)}$. We have the following equalities:
\begin{enumerate}
\item $\cE_{f \circ g} \cup (\cE_{g} \cap g^{-1}(\cE_f))=\cE_g \cup g^{-1}(\cE_f)$,
\item $\cE_{f \circ g} \cap \cE_{g} \cap g^{-1}(\cE_f)= \emptyset$.
\end{enumerate}
\end{Lem}

\begin{Rem}
We notice that $\cE_{g} \cap g^{-1}(\cE_f)$ is an element of $A_{\SGIET} \otimes A_{\SGIET}$.
\end{Rem}

\begin{proof}
\begin{enumerate}

\item We proceed by double inclusions.\\
From left to right, we know that $\cE_{g} \cap g^{-1}(\cE_f) \subset \cE_{g} \cup g^{-1}(\cE_f)$ hence it is sufficient to show the inclusion $\cE_{f\circ g} \subset \cE_{g} \cup g^{-1}(\cE_f)$. Let $(x,y)$ be in $\cE_{f \circ g}$. We can assume that $x<y$, the case $x>y$ is similar. Then we deduce that $f(g(x)) >f(g(y))$. We have two cases, if $g(x)>g(y)$ then $(x,y) \in \cE_g$ else we have $g(x)<g(y)$ and $f(g(x))>f(g(y))$ thus $(x,y)=g^{-1}(g(x),g(y)) \in g^{-1} \cE_f$.\\
From right to left, let $(x,y) \in \cE_g \cup g^{-1}(\cE_f)$. We can assume that $x<y$ the case $x>y$ is similar. If $(x,y) \in \cE_g \cap g^{-1}(\cE_f)$ then it is done. We have two cases, if $(x,y) \in \cE_g$ and $(x,y) \notin g^{-1}(\cE_f)$ then as $x< y$ we have $g(x)>g(y)$ and $f(g(x))>f(g(y))$ thus $(x,y) \in \cE_{f \circ g}$. If $(x,y) \in g^{-1}(\cE_f)$ and $(x,y) \notin \cE_g$ then as $x<y$ we have $ g(x)< g(y)$ and $f(g(x))>f(g(y))$ thus $(x,y) \in \cE_{f \circ g}$.

\item By contradiction let us assume that there exists $(x,y) \in \cE_{f \circ g} \cap \cE_{g} \cap g^{-1}(\cE_f)$. We can assume that $x<y$, the case $x>y$ is similar. As $x<y$ and $(x,y) \in \cE_{f\circ g} \cap \cE_{g}$ we know that $g(x) > g(y)$ and $f(g(x))> f(g(y))$. However $g.(x,y)=(g(x),g(y)) \in \cE_f$ thus as $g(x)>g(y)$ we have $f(g(x))<f(g(y))$ which is a contradiction.
\end{enumerate}
\end{proof}

\begin{Thm}
The signature $\widehat{\eps_{\SGIET}^{\bowtie}}$ is a group homomorphism.
\end{Thm}

\begin{proof}
Let $f,g \in \widehat{\IET^{\bowtie}(\SGIET)}$. Thanks to Lemma \ref{Interactions des ensemble E(fg), E(f) et E(g)} we have:
\begin{align*}
\widehat{\eps_{\SGIET}^{\bowtie}}(f \circ g)= &\omega_{\SGIET}(\cE_{f\circ g}) \modulo{2}\\
=& \omega_{\SGIET}(\cE_{f\circ g}) + \omega_{\SGIET}(\cE_g \cap g^{-1}(\cE_f)) - \omega_{\SGIET}(\cE_g \cap g^{-1}(\cE_f)) \modulo{2}\\
=& \omega_{\SGIET}(\cE_{f\circ g} \sqcup (\cE_g \cap g^{-1}(\cE_f))) - \omega_{\SGIET}(\cE_g \cap g^{-1}(\cE_f)) \modulo{2}\\
=& \omega_{\SGIET}(\cE_g \cup g^{-1}(\cE_f)) - \omega_{\SGIET}(\cE_g \cap g^{-1}(\cE_f)) \modulo{2}\\
=& \omega_{\SGIET}(\cE_g)+\omega_{\SGIET}(g^{-1}(\cE_f)) -2 \omega_{\SGIET}(\cE_g \cap g^{-1}(\cE_f))  \modulo{2}\\
=& \omega_{\SGIET}(\cE_g)+\omega_{\SGIET}(g^{-1}(\cE_f)) \modulo{2}\\
=& \widehat{\eps_{\SGIET}^{\bowtie}}(g) + \widehat{\eps_{\SGIET}^{\bowtie}}(f)
\end{align*}

\end{proof}

We notice that every finitely supported permutation is in $\Ker(\widehat{\eps_{\SGIET}^{\bowtie}})$.

\begin{Coro}
There exists a surjective group homomorphism $\eps_{\SGIET}^{\bowtie}: \IET^{\bowtie}(\SGIET) \rightarrow \bigotimes^2_{\ZZ} \tilde{\SGIET} / 2(\bigotimes^2_{\ZZ} \tilde{\SGIET})$.
\end{Coro}

\begin{Exem}
We give the value of $\eps_{\SGIET}^{\bowtie}$ for two kinds of elements:
\begin{enumerate}
\item Let $r$ be a $\SGIET$-reflection of type $a$ then $\eps_{\SGIET}^{\bowtie}(r)=a \otimes a \modulo{2}$,
\item Let $s$ be a $\SGIET$-restricted rotation of type $(p,q)$ then:
\[\eps_{\SGIET}^{\bowtie}(s)=p \otimes q + q \otimes p \modulo{2}
\]
\end{enumerate}
\end{Exem}

As $\IET^{\bowtie}$ is generated by reflections, we deduce the image of $\eps_{\SGIET}^{\bowtie}$:

\begin{Coro}\label{Corollary Image de epsilon_Gamma}
We have the following isomorphism:
\[
\IET^{\bowtie}(\SGIET) / \Ker(\eps_{\SGIET}^{\bowtie}) \simeq \textup{Im}(\eps_{\SGIET}^{\bowtie})=\langle {\lbrace a \otimes a \modulo{2} \mid a \in \tilde{\SGIET} \rbrace} \rangle,
\]
where $\langle {\lbrace a \otimes a \modulo{2} \mid a \in \tilde{\SGIET} \rbrace} \rangle$ is a subgroup of $\bigotimes^2_{\ZZ} \tilde{\SGIET} / 2(\bigotimes^2_{\ZZ} \tilde{\SGIET})$.
\end{Coro}

\subsubsection{Description of $\Ker(\eps_{\SGIET}^{\bowtie}$)}~
\smallskip

As $\bigotimes^2_{\ZZ} \tilde{\SGIET} / 2(\bigotimes^2_{\ZZ} \tilde{\SGIET})$ is an abelian group we know that $D(\IET^{\bowtie}(\SGIET))$ is included in $\Ker(\eps_{\SGIET}^{\bowtie})$. We will see later that the other inclusion is false in general.

\begin{Nota}
 We denote by $\Omega_{\SGIET}$ the conjugate closure of the group generated by the set of all $\SGIET$-reflections of type $2 \ell$ with $\ell \in \tilde{\SGIET} \setm 2 \tilde{\SGIET}$ (the closure inside $\IET^{\bowtie}(\SGIET)$).
\end{Nota}
The inclusion $\Omega_{\SGIET} \subset \Ker(\eps_{\SGIET}^{\bowtie})$ is immediate. By \ref{Proposition reflections of type 2a are conjugated to restricted rotation of type (a,a)} we know that every restricted rotation of type $(\ell,\ell)$, with $\ell \in \tilde{\SGIET}$, is in $\Omega_{\SGIET}$.

The aim here is to show the equality $\Ker(\eps_{\SGIET}^{\bowtie})=D(\IET^{\bowtie}(\SGIET))\Omega_{\SGIET}$. 

\medskip

We begin by proving the result in the specific case where $\SGIET$ has finite rank. We will reduce the general case to this one.

\begin{Lem}\label{Ker=Derive Omega cas finiment engendré}
Let $\SGIET$ be a finitely generated subgroup of $\mathbb{R/Z}$. Then for every $f \in \Ker(\eps_{\SGIET}^{\bowtie})$ there exist $\delta \in D(\IET^{\bowtie}(\SGIET))$ and $h \in \Omega_{\SGIET}$ such that $f=\delta h$.
\end{Lem}

\begin{proof}
As $\SGIET$ is finitely generated, we know that $\tilde{\SGIET}$ is finitely generated. Let $d$ be the rank of $\tilde{\SGIET}$.

Let $f \in \Ker(\eps_{\SGIET}^{\bowtie})$. Let $n \in \NN$ and $\cP:=\lbrace I_1,I_2,\ldots, I_n \rbrace$ be a partition into $\SGIET$-intervals associated with $f$. We denote by $L_i$ the length of $I_i$ for every $1 \leq i \leq n$. By Corollary \ref{Corollary ultrasimplicially ordered group} there exists $B:=\lbrace \ell_1, \ell_2, \ldots ,\ell_d \rbrace$ a basis of $\tilde{\SGIET}$ with elements in $\tilde{\SGIET}_{+}$ such that $L_i \in \Vect_{\NN}(B)$ for every $1 \leq i \leq n$. Hence we can cut each $I_i$ into smaller intervals with length in $B$. This operation gives us a new partition $\cQ:= \lbrace J_1,J_2, \ldots ,J_k \rbrace$, with $k \in \NN$, into $\SGIET$-intervals associated with $f$.

For every $1 \leq i \leq k$ we define $r_i$ as the $J_i$-reflection if $f$ is order-reversing on $J_i$ else we put $r_i= \Id$. Let $g$ be the product $fr_1r_2\ldots r_k$; it is an element of $\IET(\SGIET)$ and $\cQ$ is a partition into $\SGIET$-intervals associated with $g$. By Lemma \ref{Lemma Decomposition in IET(Gamma)-restricted rotations} $g$ can be written as a finite product of $\SGIET$-restricted rotations with type inside $\lbrace (l_p,l_q) \mid p,q \in \lbrace 1 ,2, \ldots, l_d \rbrace \rbrace$. Thanks to an element of $D(\IET^{\bowtie}(\SGIET))$ we can organize this product to put all $\SGIET$-restricted rotations of type $(l_p,l_p)$ together (with $1 \leq p \leq d$; they are elements in $\Omega_{\SGIET}$): there exist $w_1 \in D(\IET^{\bowtie}(\SGIET))$ and $h \in \Omega_{\SGIET}$ and $m \in \NN$ and $s_1,s_2,\ldots, s_m$ some $\SGIET$-restricted rotations with type inside $\lbrace (l_p,l_q) \mid p,q \in \lbrace 1,2,\ldots ,d \rbrace ,~ p \neq q \rbrace$ such that $g=w_1hs_1s_2 \ldots s_m$. Then $f=w_1 h s_1 s_2 \ldots s_m r_k r_{k-1} \ldots r_1$.

We define $u_p:= \Card \lbrace i \in \lbrace 1,2, \ldots ,k \rbrace \mid r_i \neq \Id ,~ \text{type}(r_i)=l_p \rbrace$. Let $v_{p,p}=0$ for every $1 \leq p \leq d$ and let $v_{p,q}:= \Card \lbrace j \in \lbrace 1,2, \ldots ,m \rbrace \mid \text{type}(s_j)=(l_p,l_q) \rbrace$ for every $1 \leq p \neq q \leq d$. Then we have:
$$\eps_{\SGIET}^{\bowtie}(f)=\sum\limits_{i=1}^k \eps_{\SGIET}^{\bowtie}(r_i) + \sum\limits_{j=1}^m \eps_{\SGIET}^{\bowtie}(s_j)=\sum\limits_{p=1}^d u_p l_p \otimes l_p + \sum\limits_{p=1}^d \sum\limits_{q=1}^d v_{p,q} (l_p \otimes l_q + l_q \otimes l_p) \modulo{2}=0$$
We notice that $\sum\limits_{p=1}^d \sum\limits_{q=1}^d v_{p,q} (l_p \otimes l_q + l_q \otimes l_p)=\sum\limits_{p=1}^d \sum\limits_{q=1}^d (v_{p,q}+v_{q,p}) l_p \otimes l_q$.
Furthermore $B$ is a basis of $\tilde{\SGIET}$ so $\lbrace l_p \otimes l_q \rbrace_{1 \leq p,q \leq d}$ is a basis of $\bigotimes^2_{\ZZ} \tilde{\SGIET}$, thus we deduce that $2$ divides $u_p$ for every $1 \leq p \leq d$ and $2$ divides $v_{p,q} + v_{q,p}$ for every $1 \leq p,q \leq d$.

We obtain that $r_1r_2\ldots r_k$ is a balanced product of $\SGIET$-reflections hence by Lemma \ref{Lemma D(IET bowtie Gamma) is generated by balanced product of reflections} it is an element of $D(\IET^{\bowtie}(\SGIET))$ denoted $w_2$. We also deduce that the product $s_1s_2\ldots s_m$ is a balanced product of $\SGIET$-restricted rotations with type inside $\lbrace (a,b) \mid a \neq b \in \tilde{\SGIET}_+ \rbrace$. Hence by Proposition \ref{Proposition balanced product of restricted rotations is a balanced product of reflections} we obtain that it is also an element of $D(\IET^{\bowtie}(\SGIET))$, denoted $w_3$.

Finally we have $f=w_1hw_3w_2=\delta h$ with $\delta \in D(\IET^{\bowtie}(\SGIET))$ and $h \in \Omega_{\SGIET}$.
\end{proof}

The next lemma gives an inclusion used to conclude in the general case:

\begin{Lem}\label{Lemme inclusion D()Omega pour Gamma subset Gamma}
For all $\SGIET, A$ subgroups of $\mathbb{R/Z}$ such that $A \subset \SGIET$ we have:
$$D(\IET^{\bowtie}(A))\Omega_{A} \subset D(\IET^{\bowtie}(\SGIET))\Omega_{\SGIET} $$
\end{Lem}

\begin{proof}
The inclusion $D(\IET^{\bowtie}(A)) \subset D(\IET^{\bowtie}(\SGIET))$ is immediate. It is sufficient to show that $\Omega_{A} \subset D(\IET^{\bowtie}(\SGIET))\Omega_{\SGIET}$. Let $\tilde{A}$ be the preimage of $A$ in $\RR$.

Let $f$ be an element of $\Omega_{A}$. Then there exist $n \in \NN$ and $a_1,a_2,\ldots,a_n \in \tilde{A} \setm 2 \tilde{A}$ and $w_1,w_2,\ldots, w_n$ some $\SGIET$-reflections such that the type of $w_i$ is $2a_i$ and there exist $g_1,g_2,\ldots g_n \in \IET^{\bowtie}(A)$ such that $f = \prod\limits_{i=1}^n g_iw_ig_i^{-1}$.

Let $U:= \lbrace i \in \lbrace 1,2, \ldots, n \rbrace \mid a_i \in \tilde{\SGIET} \setm 2\tilde{\SGIET} \rbrace$. By definition we have $\lbrace g_iw_ig_i^{-1} \mid i \in U \rbrace \subset \Omega_{\SGIET}$. Take $V:= \lbrace 1,2,\ldots, n \rbrace \setm U$. As $A$ is a subgroup of $\SGIET$ we deduce that $V= \lbrace j \in \lbrace 1,2,\ldots ,n \rbrace \mid a_j \in 2\tilde{\SGIET} \rbrace$. Thus for every $j \in V$ the type of $w_j$ is in $4 \tilde{\SGIET}$, then by Proposition \ref{Proposition a Gamma reflection with type in 4Gamma is in D(IET bowtie(Gamma))} we deduce that $w_j$ and $g_jw_jg_j^{-1}$ belong to $D(\IET^{\bowtie}(\SGIET))$. We know that there exists $h \in D(\IET^{\bowtie}(\SGIET))$ such that: 
\[
f= h \prod\limits_{j \in V} g_j w_jg_j^{-1} \prod \limits_{i \in U} g_iw_ig_i^{-1}
\]
Then $f \in D(\IET^{\bowtie}(\SGIET))\Omega_{\SGIET}$.
\end{proof}

We can prove the theorem for the general case:
\begin{Thm}\label{Theorem description kernel of epsilon_Gamma}
For any dense subgroup $\SGIET$ of $\mathbb{R/Z}$ we have:
\[
\Ker(\eps_{\SGIET}^{\bowtie}) = D(\IET^{\bowtie}(\SGIET)) \Omega_{\SGIET}
\]
\end{Thm}

\begin{proof}
The inclusion from right to left is already proved.

Let $f \in \Ker(\eps_{\SGIET}^{\bowtie})$, let $n \in \NN$ and $\cP:=\lbrace I_1,I_2,\ldots, I_n \rbrace$ be a partition into $\SGIET$-intervals associated with $f$. We denote by $L_i$ the length of $I_i$ for every $1 \leq i \leq n$. As $\eps_{\SGIET}^{\bowtie}(f)=0$ we know that there exist $k \in \NN$ and $a_1,a_2\ldots,a_k,b_1,b_2,\ldots ,b_k \in \tilde{\SGIET}$ such that $\omega_{\SGIET}(\cE_f)=2\sum\limits_{i=1}^k a_i \otimes b_i$ inside $\bigotimes^2_{\ZZ} \tilde{\SGIET}$.

Let $\tilde{A}$ be the subgroup of $\RR$ generated by $\lbrace L_i \rbrace_{i=1\ldots n} \cup \lbrace 1 \rbrace \cup \lbrace a_i,b_i \rbrace_{i=1 \ldots k }$. Then $\tilde{A}$ contains $\ZZ$, is finitely generated and is a subgroup of $\tilde{\SGIET}$. Let $A$ be the image of $\tilde{A}$ in $\mathbb{R/Z}$. The partition $\cP$ is also a partition into $A$-intervals associated with $f$ thus $f$ belongs to $\IET^{\bowtie}(A)$. Furthermore we have $\omega_{A}(\cE_f)=\omega_{\SGIET}(\cE_f)$. Hence:

\[
\eps^{\bowtie}_{A}(f)=[\omega_{A}(\cE_f)]_{\bigotimes^2_{\ZZ} A / 2(\bigotimes^2_{\ZZ} A)}=[\omega_{\SGIET}(\cE_f)]_{\bigotimes^2_{\ZZ} A / 2(\bigotimes^2_{\ZZ} A)}=[2\sum\limits_{i=1}^k a_i \otimes b_i]_{\bigotimes^2_{\ZZ} A / 2(\bigotimes^2_{\ZZ} A)}=0
\]

By Lemma \ref{Ker=Derive Omega cas finiment engendré} we deduce that $f \in D(\IET^{\bowtie}(A))\Omega_{A}$ and by Lemma \ref{Lemme inclusion D()Omega pour Gamma subset Gamma} we deduce that $f \in D(\IET^{\bowtie}(\SGIET))\Omega_{\SGIET}$.
\end{proof}

\subsection{The positive contribution}~
\smallskip

We know that every $\SGIET$-reflection of type $2\ell$ with $\ell \in \tilde{\SGIET} \setm 2 \tilde{\SGIET}$ is conjugated to a $\SGIET$-restricted rotation of type $(\ell,\ell)$. Also this is an element of $\IET^+(\SGIET)$ which is not send on the trivial element by the morphism $\eps_{\SGIET}$.
We use the notion of positive substitute in order to use the group homomorphism $\eps_{\SGIET}$ to send such a reflection on a nontrivial element.
For this we need to use the sequence $(S_n)_{ n \in \NN}$ of $\ZZ$-independent subset of $\tilde{\SGIET}_{+}$ given by Corollary \ref{Corollary the sequence of finite Z indepndent subset for tilde SGIET}. We recall that for each $n \in \NN$ we have $\Vect_{\NN}(S_n) \subset \Vect_{\NN}(S_{n+1})$ and $\tilde{\SGIET}_{+}$ is equal to the direct limit $\lim\limits_{\longrightarrow} \Vect_{\NN}(S_n)$. Also if $\SGIET$ is finitely generated we can choose $S_n$ to be a basis of $\tilde{\SGIET}$.

For every $n \in \NN$ we construct a map that depends on $S_n$. They are not group homomorphisms but satisfies the group homomorphism property on some products thanks to Proposition \ref{Proposition behaviour of positive substitute of a composition}.

\subsubsection{Some subsets of $\IET^{\bowtie}(\SGIET)$}~

Let $S$ be a finite set of $\tilde{\SGIET}_{+}$. We denote by $G_S$ the set of all $f $ in $\IET^{\bowtie}(\SGIET)$ such that there exists a partition $\cP$ into $S$-intervals associated with $f$. We remark that $G_S$ is not a group in general. We want to know how these sets and $\IET^{\bowtie}(\SGIET)$ are linked.

\begin{Prop}\label{Proposition inclusion of G_S in G_T}
Let $S$ and $T$ be two finite subsets of $\tilde{\SGIET}_{+}$. If $S \subset \Vect_{\NN}(T)$ then $G_S \subset G_T$. More precisely for every partition $\cP$ into $S$-intervals of $\interfo{0}{1}$ there exists a refinement $\cQ$ of $\cP$ which is a partition into $T$-intervals of $\interfo{0}{1}$.
\end{Prop}

\begin{proof}
Let $f \in G_S$ and let $\cP$ be a partition into $S$-intervals associated with $f$. As $S \subset \Vect_{\NN}(T)$ each interval $I \in P$ can be subdivided with intervals of length in $T$. After subdividing this way the intervals of $\cP$, we obtain a refinement $\cQ$ of $\cP$. We notice that $\cQ$ is a partition into $T$-intervals associated with $f$ so $f \in G_T$.
\end{proof}

As $\tilde{\SGIET}_+$ is the filtered union of the $\Vect_{\NN}(S_n)$ we deduce the next proposition:

\begin{Prop}\label{Proposition G_S is always inside a (G_S)_n}
For every finite subset $S$ of $\tilde{\SGIET}_+$ there exists $N \in \NN$ such that for every $n \geq N$ we have $S \subset \Vect_{\NN}(S_n)$.
\end{Prop}

From Propositions \ref{Proposition inclusion of G_S in G_T} and \ref{Proposition G_S is always inside a (G_S)_n} we obtain:

\begin{Coro}\label{Corollary Every element of IET bowtie Gamma is in one  (G_S)_n}
Let $f \in \IET^{\bowtie}(\SGIET)$. There exists $N \in \NN$ such that for every $n \geq N$ we have $f \in G_{S_n}$. Furthermore $\IET^{\bowtie}(\SGIET)=\bigcup\limits_{n \in \NN} G_{S_n}$.
\end{Coro}

We also want to check that if we take any product of elements in $\IET^{\bowtie}(\SGIET)$ there will be a moment where we have a partition associated to the tuple of these elements.

\begin{Prop}\label{Proposition There exists a common S_n for a tuple of IET bowtie Gamma}
Let $k \in \NN$ and $f_1,f_2, \ldots ,f_k \in \IET^{\bowtie}(\SGIET)$. There exists $N \in \NN$ such that for every $n \geq N$ there exists $\cP_n$ a partition into $S_n$-intervals associated with $(f_1,f_2, \ldots, f_k)$.
\end{Prop}

\begin{proof}
Thanks to Proposition \ref{Proposition inclusion of G_S in G_T} and Remark \ref{Remark refinement of a partition associated to a tuple is still a partition associated to the tuple} we deduce that it is sufficient to find only one $n \in \NN$ where the statement is satisfied.

By Proposition \ref{Corollary Every element of IET bowtie Gamma is in one  (G_S)_n} there exist $n_i \in \NN$ and $\cP_i$ a partition into $S_{n_i}$ intervals associated with $f_i$ for every $1 \leq i \leq k$. We denote by $V_i$ the set of all the endpoints of the intervals in $\cP_i$. Let $V= V_1 \cup f_1^{-1}(V_2) \cup \ldots \cup (f_{k-1} \ldots f_1)^{-1}(V_k)$. We know that $V$ is finite thus there exist $m \in \NN$ and $v_0,v_1,\ldots ,v_m \in \tilde{\SGIET}_{+}$ such that $V=\lbrace v_i \rbrace_{i=0 \ldots m}$. Up to change the order we can assume that $v_0=0<v_1<v_2< \ldots < v_{m-1}<v_m=1$. Let $I_i$ be the interval $\mathopen{[}v_{i-1},v_i\mathclose{[}$ and $\ell_i$ be the length of $I_i$ for every $1 \leq i \leq m$. Let $T=\lbrace \ell_i \rbrace_{i=1 \ldots m}$ and $\cP= \lbrace I_i \rbrace_{i=1 \ldots m}$.\\
We have $V_1 \subset V$ then $\cP$ is a refinement of $\cP_1$ so $\cP$ is a partition into $T$-intervals associated with $f_1$. Similarly for every $2 \leq i \leq k$ we know that $V_i \subset f_{i-1}\ldots f_1(V)$ thus $f_{i-1}\ldots f_1(\cP)$ is a refinement of $\cP_i$ so $f_{i-1}\ldots f_1(\cP)$ is a partition into $T$-intervals associated with $f_i$. Hence $\cP$ is a partition into $T$-intervals associated with $(f_1,f_2,\ldots ,f_k)$. \\
Thanks to Proposition \ref{Proposition G_S is always inside a (G_S)_n} there exists $n \in \NN$ such that $T \subset \Vect_{\NN}(S_n)$ and by Proposition \ref{Proposition inclusion of G_S in G_T} and Remark \ref{Remark refinement of a partition associated to a tuple is still a partition associated to the tuple} there exists a refinement $\cQ$ of $\cP$ which is a partition into $S_n$-intervals associated with $(f_1,f_2,\ldots ,f_k)$.

\end{proof}

\subsubsection{The $S$-map}~
\smallskip

Let $S$ be a finite subset of $\tilde{\SGIET}_{+}$ for all this subsection. We assume that $S$ is free inside $\Vect_{\ZZ}(S)$. For every $a,b \in \tilde{\SGIET}$, we use the notation $[a \wedge b]_{2 \Wedgealt \tilde{\SGIET}}= a \wedge b \modulo{2}$.

Let $f \in G_S$ and $\cP$ be a $S$-partition associated with $f$. We show that the value $\eps_{\SGIET}(f^+_{\cP}) \modulo{2}$ does not depend on $\cP$, where $\eps_{\SGIET}$ is the group homomorphism define in \ref{Definition signature for IET(SGIET)}.

\begin{Prop}\label{Proposition eps_Gamma mod 2 is independent of the chosen partition}
Let $f \in \IET^{\bowtie}$ and let $\cP, \cQ$ be two partitions into $S$-intervals associated with $f$. Then $\eps_{\SGIET}(f^+_{\cP}) \modulo{2}=\eps_{\SGIET}(f^+_{\cQ}) \modulo{2}$. 
\end{Prop}

\begin{proof}
For this proof we denote by $r_I$ the $I$-reflection for every subinterval $I$ of $\mathopen{[}0,1 \mathclose{[}$ and we recall that $\lambda$ is the Lebesgue measure. First we reduce the case to a simpler one.

Let $f \in \IET^{\bowtie}(\SGIET)$ and $\cP,\cQ$ be two partitions into $S$-intervals associated with $f$. Let $n,k \in \NN$ such that $\cP=\lbrace I_1,I_2, \ldots, I_n \rbrace $ and $\cQ=\lbrace J_1,J_2,\ldots J_k \rbrace$. Up to change the index we can assume that the $I_i$ are consecutive intervals and so are the $J_i$. We denote by $\cM$ the unique partition into $\SGIET$-intervals associated with $f$ that has the minimal number of intervals. Let $m$ be this number and $\cM:= \lbrace M_1,M_2,\ldots ,M_m \rbrace$ where the $M_i$ are consecutive intervals. As $\cP,\cQ$ are also partitions into $\SGIET$-intervals we know that they are refinements of $M$.

Let $1 \leq i \leq n$ and let $n_0=k_0=0$. There exist $n_1 <n_2< \ldots <n_m, k_1<k_2<\ldots<k_n \in \NN$ such that $M_i=\bigsqcup\limits_{j=n_{i-1} +1}^{n_{i}} I_j=\bigsqcup\limits_{j=k_{i-1} +1}^{k_{i}} J_j$. Hence $\lambda(M_i)=\sum\limits_{j=n_{i-1} +1}^{n_{i}} \lambda(I_i) =\sum\limits_{j=k_{i-1} +1}^{k_{i}} \lambda(J_j)$. For every $s \in S$ we define $a_s= \Card( \lbrace j \in \lbrace n_{i-1} +1, \ldots,n_{i} \rbrace \mid \lambda(I_j)=s \rbrace$ and $b_s= \Card( \lbrace j \in \lbrace k_{i-1} +1, \ldots,k_{i} \rbrace \mid \lambda(J_j)=s \rbrace$. Then $\sum\limits_{s \in S} a_s s= \sum\limits_{s \in S} b_s s$. As $S$ is free in $\Vect_{\ZZ}(S)$ we deduce that $a_s=b_s$ for every $s \in S$. This implies that $n_i=k_i$ for every $1 \leq i \leq n$ and $n=k$. We also get the existence of a permutation $\sigma_i$ of the set $\lbrace n_{i-1}+1, \ldots, n_{i} \rbrace$ such that $\lambda(I_j)=\lambda(J_{\sigma_i(j)})$ for every $j \in \lbrace n_{i-1}+1, \ldots, n_{i} \rbrace$.

We deduce that it is sufficient to show the result when only one permutation $\sigma_i$ is a transposition $(a ~ a+1)$ with $n_{i-1}+1 \leq a < n_i$.\\
Let $i_0$ in $\lbrace1,2, \ldots m \rbrace$ and $a \in \lbrace n_{i_0-1}+1, \ldots, n_{i_0} \rbrace$. We assume now that $\sigma_{i_0}$ is the transposition $(a ~ a+1)$ and $\sigma_i=\Id$ for every $i \neq i_0$. Then $I_j=J_j$ for every $j \in \lbrace 1,2, \ldots, n \rbrace \setm \lbrace a,a+1 \rbrace$. We also have $I_a \cup I_{a+1}=J_a \cup J_{a+1}$ and $\lambda(I_{a+1})=\lambda(J_{a})$. In the case where $f$ is order-preserving on $M_{i_0}$ then $f_{\cP}^+ = f_{\cQ}^+ $. In the case where $f$ is order-reversing on $M_{i_0}$ we deduce that $(f_{\cQ}^+)^{-1} \circ f_{\cP}^+$ is equal to the square of the restricted rotation whose intervals associated are $I_a$ and $I_{a+1}$ (see Figure \ref{Figure Independence of the partition for eps_Gamma of the positive substitute}).

In both cases we deduce $\eps_{\SGIET}(f^+_P) \modulo{2}=\eps_{\SGIET}(f^+_Q) \modulo{2}$.
\end{proof}
\begin{figure}[ht]
\includegraphics[width=\linewidth]{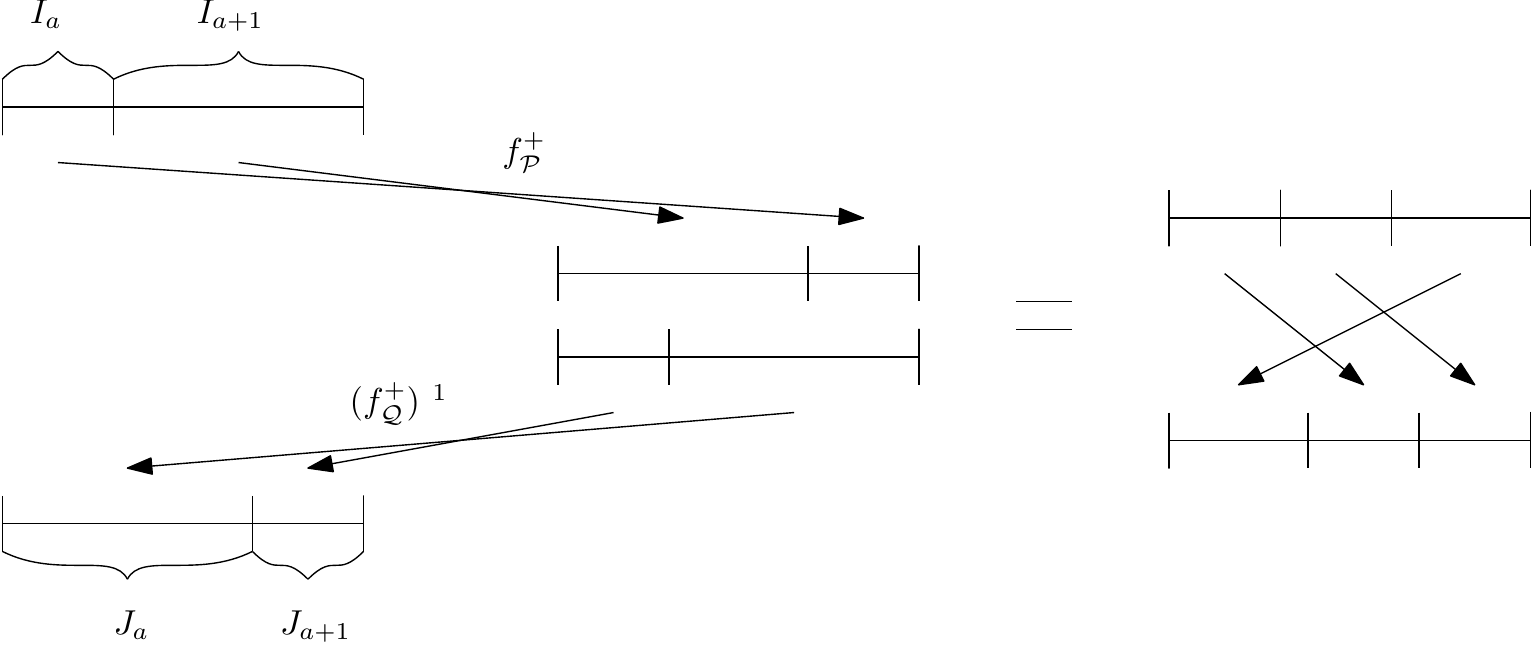}
\caption{\footnotesize{Illustration of the fact that $(f_{\cQ}^+)^{-1} \circ f_{\cP}^+$ is equal to the square of a restricted rotation in Proposition \ref{Proposition eps_Gamma mod 2 is independent of the chosen partition}}}\label{Figure Independence of the partition for eps_Gamma of the positive substitute}
\end{figure}

\begin{Def}
We define the $S$-map as:
\[
\begin{array}{rccc}
\psi_S: &\IET^{\bowtie}(\SGIET) & \longrightarrow & \Wedgealt \tilde{\SGIET} / 2\Wedgealt \tilde{\SGIET} \\
& f & \longmapsto &
\left\{
\begin{array}{cl}
\eps_{\SGIET}( f^+_{\cP}) \modulo{2} & f \in G_S \\
0 & f \notin G_S
\end{array} 
\right.
\end{array}
\]
Where $\cP$ is a partition into $S$-intervals associated with $f$.
\end{Def}

Thanks to Proposition \ref{Proposition eps_Gamma mod 2 is independent of the chosen partition}, the map $\psi_{\SGIET}$ is well-defined. We check that $\psi_S$ satisfies the morphism condition when we have the existence of a partition in $S$-intervals associated with a tuple.

\begin{Prop}\label{Proposition psi_S is nearly group homomorphism }
Let $n \in \NN$ and $f_1,f_2,\ldots ,f_n \in \IET^{\bowtie}(\SGIET)$. If there exists $\cP$ a partition into $S$-intervals associated with $(f_1,f_2,\ldots, f_n )$ then:
$$
\psi_S(f_nf_{n-1} \ldots f_1)=\sum\limits_{i=1}^n \psi_S(f_i)
$$
\end{Prop}

\begin{proof}

Let $n \in \NN$ and $f_1,f_2,\ldots ,f_n \in \IET^{\bowtie}(\SGIET)$. The case $n=1$ is trivial thus we assume $ n \geq 2$. Let $\cP$ be a partition into $S$-intervals associated with $(f_1,f_2, \ldots , f_n)$. Let $g_1=(f_1)^+_{\cP}$ and for every $2 \leq i \leq n$ let $g_i=(f_i)^+_{f_{i-1}\ldots f_1(\cP)}$. By Proposition \ref{Proposition behaviour of positive substitute of a composition} we know that $(f_nf_{n-1} \ldots f_1)^+_{\cP}=g_n g_{n-1} \ldots g_1$, then:
\[
\psi_S(f_nf_{n-1} \ldots f_1) = \eps_{\SGIET} ( (f_n f_{n-1} \ldots f_1)^+_{\cP}) \modulo{2}= \eps_{\SGIET}(g_ng_{n-1} \ldots g_1) \modulo{2}
\]

\noindent Also $\eps_{\SGIET}$ is a group homomorphism thus:
\[
\psi_S(f_nf_{n-1} \ldots f_1)=\sum\limits_{i=1}^n \eps_{\SGIET}(g_i) \modulo{2}=\sum\limits_{i=1}^n \psi_S(f_i)
\]
\end{proof}

\subsubsection{The group homomorphism}~
\smallskip

We are now able to define a new group homomorphism from $\IET^{\bowtie}(\SGIET)$. 

\begin{Def}
We define the \textit{positive contribution of $\IET^{\bowtie}(\SGIET)$} as the next map:
\[
\begin{array}{rccl}
\psi_{\SGIET}: &\IET^{\bowtie}(\SGIET) & \longrightarrow & (\Wedgealt \tilde{\SGIET} / 2\Wedgealt \tilde{\SGIET})^{\NN} / \lbrace (w_n)_{n\in \NN} \mid \exists N \in \NN, \forall n \geq N, w_n=0 \rbrace \\
& f & \longmapsto & [(\psi_{S_n}(f))_{n \in \NN}]
\end{array}
\]
\end{Def}

\begin{Prop}
The map $\psi_{\SGIET}$ is a group homomorphism.
\end{Prop}

\begin{proof}
Let $f,g \in \IET^{\bowtie}(\SGIET)$. By Proposition \ref{Proposition There exists a common S_n for a tuple of IET bowtie Gamma} there exists $N \in \NN$ such that for every $n \geq N$ there exists a partition $\cP_n$ into $S_n$-intervals associated with $(f,g)$. Then  by Proposition \ref{Proposition psi_S is nearly group homomorphism } we deduce that for every $n \geq N$ we have $\psi_{S_n}(g \circ f)=\psi_{S_n}(g)+\psi_{S_n}(f)$. We denote by $(w_n)_{n\in \NN}$ the element of $(\Wedgealt \tilde{\SGIET} / 2\Wedgealt \tilde{\SGIET})^{\NN}$ defined by $w_n=\psi_{S_n}(g \circ f) -\psi_{S_n}(g)- \psi_{S_n}(f)$ for every $n <N$ and $0$ elsewhere.\\
We notice that $(\psi_{S_n}(g \circ f))_n=(\psi_{S_n}(g))_n + (\psi_{S_n}(f))_n -w$. So $\psi_{\SGIET}(g \circ f) = \psi_{\SGIET}(g) + \psi_{\SGIET}(f)$.
\end{proof}

The following lemma gives the value of an element in $\IET(\SGIET)$. The proof is immediate from the definition of $\psi_{S_n}$.

\begin{Lem}\label{Proposition value of psi_Gamma for elements in IET(Gamma)}
For every $f \in \IET(\SGIET)$ there exists $N \in \NN$ such that for $ n < N$ we have $\psi_{S_n}(f)=0$ and for $n \geq N$ we have $\psi_{S_n}(f)= \eps_{\SGIET}(f) ~[\textup{mod}~2]$.
\end{Lem}

The next proposition show that the set we used to define $\Omega_{\SGIET}$ is not sent to the trivial element by $\psi_{\SGIET}$; thus this is not a subset of $D(\IET^{\bowtie}(\SGIET))$.

\begin{Prop}\label{Proposition Value of psi_Gamma on a reflection of type 2l with l notin 2Gamma}
Let $\ell \in \tilde{\SGIET} \setm 2\tilde{\SGIET}$ and $r$ be a $\SGIET$-reflection of type $2\ell$. Then $\psi_{\SGIET}(r)=[(\ell \wedge \ell \modulo{2})_n] \neq 0$.
\end{Prop}

\begin{proof}
We know that such a $\SGIET$-reflection is conjugate to a $\SGIET$-restricted rotation of type $(\ell,\ell)$. We denote by $s$ this $\SGIET$-restricted rotation. We have $\eps_{\SGIET}(s)= \ell \wedge \ell ~[\textup{mod}~2]$. Then by Lemma \ref{Proposition value of psi_Gamma for elements in IET(Gamma)} we have $\psi_{\SGIET}(r)=\psi_{\SGIET}(s)=[(\ell \wedge \ell \modulo{2})_n]$.
Thus if $\psi_{\SGIET}(r)=0$ then there exists $(w_n)_n \in (\Wedgealt \tilde{\SGIET} / 2\Wedgealt \tilde{\SGIET})^{\NN}$ which is $0$ above a certain rank such that $(\ell \wedge \ell \modulo{2})_n+(w_n)_n=0$. As $w_n$ is $0$ above a certain rank, we deduce that $\ell \wedge \ell \modulo{2}=0$. Or this implies $\ell \in 2\tilde{\SGIET}$ and this is a contradiction. Hence $\psi_{\SGIET}(r) \neq 0$.
\end{proof}

\subsection{\texorpdfstring{Description of $\IET^{\bowtie}(\Gamma)_{\text{ab}}$}{Description of the abelianization of IET bowtie (Gamma)}}~
\smallskip

With both morphisms $\eps_{\SGIET}^{\bowtie}$ and $\psi_{\SGIET}$ we are now able to describe $D(\IET^{\bowtie}(\SGIET))$. We recall that $\Omega_{\SGIET}$ is the conjugate closure of the group generated by the set of all $\SGIET$-reflections of type $2 \ell$ with $\ell \in \tilde{\SGIET} \setm 2\tilde{\SGIET}$.

\begin{Lem}\label{Lemma Omega_Gamma cap kernel of psi_Gamma is in the derived subgroup}
We have the inclusion:
\[
\Omega_{\SGIET}\cap \Ker(\psi_{\SGIET}) \subset D(\IET^{\bowtie}(\SGIET))
\]
\end{Lem}

\begin{proof}
Let $w \in \Omega_{\SGIET}\cap \Ker(\psi_{\SGIET})$. There exist $n \in \NN$ and $g_1,g_2,\ldots g_n \in \IET^{\bowtie}(\SGIET)$ and $r_1,r_2,\ldots ,r_n$ some $\SGIET$-reflections with type inside $\tilde{\SGIET} \setm 2\tilde{\SGIET}$ such that $w = \prod\limits_{i=1}^n g_ir_ig_i^{-1}$. We know that every $\SGIET$-reflection is conjugate to a $\SGIET$-restricted rotation so there exist $h_1,h_2, \ldots h_n \in \IET^{\bowtie}(\SGIET)$ and $s_1,s_2,\ldots ,s_n$ some $\SGIET$-restricted rotations such that $s_i=h_ir_ih_i^{-1}$ for every $1 \leq i \leq n$. Then $w=\prod\limits_{i=1}^n g_ih_i^{-1}s_ih_ig_i^{-1}$. By Lemma \ref{Lemma D(IET bowtie Gamma) is generated by balanced product of reflections} it is sufficient to show that $w$ is a balanced product of $\SGIET$-reflections.

As for every $1 \leq i \leq n$ we have $g_i$ and $g_i^{-1}$ which appear the same number of time in $w$ (the same is true for $h$ and $h^{-1}$) we deduce that it is sufficient to show that $w':=\prod\limits_{i=1}^n s_i$ is a balanced product of $\SGIET$-reflections. By Lemma \ref{Lemma D(IET bowtie Gamma) is generated by balanced product of reflections} it is enough to show that $w' \in D(\IET^{\bowtie}(\SGIET))$.

 We notice that $\psi_{\SGIET}(w)=\psi_{\SGIET}(w')$ thus by the assumption we deduce that $\psi_{\SGIET}(w')=0$. Hence by Lemma \ref{Proposition value of psi_Gamma for elements in IET(Gamma)} we deduce that $\eps_{\SGIET}(w') ~[\textup{mod}~2] =0$. This equality stands in $\Wedgealt \tilde{\SGIET} / 2\Wedgealt \tilde{\SGIET}$. There exists $k \in \NN$ and for every $1 \leq j \leq k$ there exist $a_j,b_j \in \tilde{\SGIET}_{+}$ with $a_j+b_j<1$ and $n_j \in \NN$ and $\nu_j \in \lbrace -1,1 \rbrace$ such that $\eps_{\SGIET}(w')+ \sum\limits_{j=1}^k 2n_j\nu_j a_j \wedge b_j=0$. For every $1 \leq j \leq k$ let $\gamma_j$ be a $\SGIET$-restricted rotation of type $(a_j,b_j)$. Then the element $w'\prod\limits_{j=1}^k (\gamma_j^{n_j\nu_j})^2$ is in $\IET(\SGIET)$ and satisfies:
\[
\eps_{\SGIET}(w'\prod\limits_{j=1}^k (\gamma_j^{n_j\nu_j})^2)=\eps_{\SGIET}(w')+ \sum\limits_{j=1}^k 2n_j\nu_j a_j \wedge b_j=0
\]

By Theorem \ref{Theorem D(IET(Gamma)) is equal to the Kernel of Eps Gamma} the element $w'\prod\limits_{j=1}^k (\gamma_j^{n_j\nu_j})^2$ is in $D(\IET(\SGIET)) \subset D(\IET^{\bowtie}(\SGIET))$. By Corollary \ref{Corollary squares of an element are in D(IET^bowtie(Gamma))} we know that $\prod\limits_{j=1}^k (\gamma_j^{n_j\nu_j})^2$ is in $D(\IET^{\bowtie}(\SGIET))$. Then we deduce that $w' \in D(\IET^{\bowtie}(\SGIET))$. Hence $w$ is in $D(\IET^{\bowtie}(\SGIET))$.
\end{proof}

\begin{Thm}
We have $D(\IET^{\bowtie}(\SGIET))=\Ker(\eps_{\SGIET}^{\bowtie}) \cap \Ker(\psi_{\SGIET})$.
\end{Thm}

\begin{proof}
The inclusion from left to right is trivial.

Let $f \in \Ker(\eps_{\SGIET}^{\bowtie}) \cap \Ker(\psi_{\SGIET})$. By Theorem \ref{Theorem description kernel of epsilon_Gamma} there exist $g \in D(\IET^{\bowtie}(\SGIET)$ and $h \in \Omega_{\SGIET}$ such that $f=gh$. We deduce that $\psi_{\SGIET}(h)=\psi_{\SGIET}(f)=0$ so $h \in \Ker(\psi_{\SGIET}) \cap \Omega_{\SGIET}$. By Lemma \ref{Lemma Omega_Gamma cap kernel of psi_Gamma is in the derived subgroup} we obtain that $h \in D(\IET^{\bowtie}(\SGIET))$, thus $f \in D(\IET^{\bowtie}(\SGIET))$.
\end{proof}

\begin{Coro}\label{Corollary derived subgroup is equal to the kernel of psi_Gamma restricted to kernel of eps_Gamma}
We have $D(\IET^{\bowtie}(\SGIET))=\Ker(\psi_{\SGIET}|_{ \Ker(\eps_{\SGIET}^{\bowtie})})$.
\end{Coro}

\begin{Lem}\label{Lemma Ker/D simeq Gamma wedge Gamma [2]}
The quotient $\Ker(\eps_{\SGIET}^{\bowtie}) / D(\IET^{\bowtie}(\SGIET))$ is isomorphic to the subgroup $\langle \lbrace \ell \wedge \ell \modulo{2} \mid \ell \in \tilde{\SGIET} \rbrace \rangle$ of $\Wedgealt \tilde{\SGIET} / 2\Wedgealt \tilde{\SGIET}$.
\end{Lem}

\begin{proof}
By Corollary \ref{Corollary derived subgroup is equal to the kernel of psi_Gamma restricted to kernel of eps_Gamma} we have:
\[\Ker(\eps_{\SGIET}^{\bowtie}) / D(\IET^{\bowtie}(\SGIET)) \simeq \textup{Im}(\psi_{\SGIET}|_{\Ker(\eps_{\SGIET}^{\bowtie})})=\psi_{\SGIET}(\Ker(\eps_{\SGIET}^{\bowtie}))
\]
By Theorem \ref{Theorem description kernel of epsilon_Gamma} and as $\psi_{\SGIET}$ is a group homomorphism we have the equality $\psi_{\SGIET}(\Ker(\eps_{\SGIET}^{\bowtie}))=\psi_{\SGIET}(\Omega_{\SGIET})$. Furthermore $\Omega_{\SGIET}$ is the normal closure of the group generated by all $\SGIET$-reflections of type $2\ell$ with $\ell \in \tilde{\SGIET} \setm 2\tilde{\SGIET}$. Hence we deduce that:
\[
\psi_{\SGIET}(\Omega_{\SGIET})= \langle \lbrace [(\ell \wedge \ell \modulo{2})_{n \in \NN}] \mid \ell \in \tilde{\SGIET} \setm 2\tilde{\SGIET} \rangle= \langle \lbrace [(\ell \wedge \ell \modulo{2})_{n \in \NN}] \mid \ell \in \tilde{\SGIET} \rangle
\]
Thus $\psi_{\SGIET}(\Ker(\eps_{\SGIET}^{\bowtie})) \simeq \langle \lbrace \ell \wedge \ell \modulo{2} \mid \ell \in \tilde{\SGIET} \rbrace \rangle$.
\end{proof}

\begin{Thm}
We have the following group ismomorphisms:
\begin{align*}
\IET^{\bowtie}(\SGIET)_{\mathrm{ab}} &\simeq \textup{Im}(\eps_{\SGIET}^{\bowtie}) \times \Ker(\eps_{\SGIET}^{\bowtie}) / D(\IET^{\bowtie}(\SGIET)) \\
&\simeq \langle \lbrace a \otimes a \modulo{2} \mid a \in \tilde{\SGIET} \rbrace  \rangle \times \langle \lbrace \ell \wedge \ell \modulo{2} \mid \ell \in \tilde{\SGIET} \rbrace \rangle,
\end{align*}

where the left term of the product is in $\bigotimes^2_{\ZZ} \tilde{\SGIET} / (2\bigotimes^2_{\ZZ} \tilde{\SGIET})$ and the right one is in 
$\Wedgealt \tilde{\SGIET} / (2\Wedgealt \tilde{\SGIET})$.
\end{Thm}

\begin{proof}
The second isomorphism is given by Corollary \ref{Corollary Image de epsilon_Gamma} and Lemma \ref{Lemma Ker/D simeq Gamma wedge Gamma [2]}.

For the first isomorphism we recall that we have the following exact sequence:
$$1 \rightarrow \Ker(\eps_{\SGIET}^{\bowtie}) / D(\IET^{\bowtie}(\SGIET)) \rightarrow \IET^{\bowtie}(\SGIET)_{\mathrm{ab}} \rightarrow \IET^{\bowtie}(\SGIET) / \Ker(\eps_{\SGIET}^{\bowtie}) \rightarrow 1 $$
Each group in this exact sequence has exponent $2$. Then they are also $\mathbb{F}_2$-vectorial spaces. We deduce that this exact sequence is an exact sequence of $\mathbb{F}_2$-vectorial spaces, thus it splits and gives the result.
\end{proof}

\begin{Rem}
If $\tilde{\SGIET}$ has dimension $d$ then $\langle \lbrace a \otimes a \modulo{2} \mid a \in \tilde{\SGIET} \rbrace  \rangle$ has dimension $\frac{d(d+1)}{2}$, as $\mathbb{F}_2$-vector space, and $\langle \lbrace \ell \wedge \ell \modulo{2} \mid \ell \in \tilde{\SGIET} \rbrace \rangle$ has dimension $d$, as $\mathbb{F}_2$-vector space, so $\IET^{\bowtie}(\SGIET)_{\mathrm{ab}}$ has dimension $\frac{d(d+3)}{2}$ over $\mathbb{F}_2$.
\end{Rem}

\begin{Rem}
The inclusion of $\IET(\SGIET)$ in $\IET^{\bowtie}(\SGIET)$ induces a group morphism $\iota$ from $\IET(\SGIET)_{\mathrm{ab}} / 2(\IET(\SGIET)_{\mathrm{ab}})$ to $\IET^{\bowtie}(\SGIET)_{\mathrm{ab}}$. By Lemma \ref{Lemma Decomposition in IET(Gamma)-restricted rotations} we know that $\IET(\SGIET)$ is generated by $\SGIET$-restricted rotations thus we deduce that the image of $\iota$ is the subgroup $\langle \lbrace p \otimes q + q \otimes p ~[\textup{mod}~2] \mid p,q \in \tilde{\SGIET} \rbrace \rangle \times \langle \lbrace a \wedge a ~[\textup{mod}~2] \rbrace \rangle$ of $\bigotimes^2_{\ZZ} \tilde{\SGIET} /(2\bigotimes^2_{\ZZ} \tilde{\SGIET}) \times \Wedgealt \tilde{\SGIET} / (2\Wedgealt \tilde{\SGIET})$. This is isomorphic to $\Wedgealt \tilde{\SGIET}/ (2\Wedgealt \tilde{\SGIET})$ and if $\tilde{\SGIET}$ has dimension $d$ then its dimension is $\frac{d(d+1)}{2}$ as $\mathbb{F}_2$-vector space. In this case $\iota$ is not surjective and its cokernel has dimension $d$ over $\mathbb{F}_2$.
In the case where $\tilde{\Gamma}$ has infinite dimension over $\ZZ$ we deduce that $\textup{Im}(\iota)$ also has infinite dimension over $\mathbb{F}_2$. By Proposition \ref{Proposition sufficient condition for dividing by 2 in Gamma} we deduce that the group $\langle \lbrace p \otimes q + q \otimes p ~[\textup{mod}~2] \mid p,q \in \tilde{\SGIET} \rbrace \rangle$ is equal to the group $\langle \lbrace a \otimes a \modulo{2} \mid a \in \tilde{\SGIET} \rbrace  \rangle$ if and only if $\tilde{\SGIET}=2\tilde{\SGIET}$.
Then $\iota$ is surjective if and only if $\tilde{\SGIET}=2\tilde{\SGIET}$.
\end{Rem}

\begin{Prop}
The group homomorphism $\iota$ is injective.
\end{Prop}

\begin{proof}
For every $f \in \IET(\SGIET)$ we denote by $[f]$ its image in $\IET(\SGIET)_{\mathrm{ab}}$.
Thanks to Theorem \ref{Theorem D(IET(Gamma)) is equal to the Kernel of Eps Gamma} we know that $[f] \in 2\IET(\SGIET)_{\mathrm{ab}}$ if and only if $\eps_{\SGIET}(f) \in 2 \Wedgealt \SGIET$. Hence to prove the statement it is enough to prove that for every $f \in \IET(\SGIET)$ such that $\eps_{\SGIET}^{\bowtie}(f)=0$ and proj$_{\Ker(\eps_{\SGIET}^{\bowtie})}(f)=0$ we have $\eps_{\SGIET}(f) \in 2 \Wedgealt \SGIET$. We use notations of inversions defined in Definition \ref{Definition inversions in IET bowtie}. By Corollary \ref{Corollary ultrasimplicially ordered group} there exists $n \in \NN$ and a $\ZZ$-linearly independent family $\lbrace l_1,l_2,\ldots , l_n \rbrace$ of $\SGIET_{+}$ and $n_{i,j} \in \ZZ$ such that $\cE_{f,1}=\sum\limits_{i,j} n_{i,j} l_i \otimes l_j$. The equality proj$_{\Ker(\eps_{\SGIET}^{\bowtie})}(f)=0$ gives us that $n_{i,i}=0$. We have $\eps_{\SGIET}^{\bowtie}(f)=\sum\limits_{i \neq j} (n_{i,j}+n_{j,i}) l_i \otimes l_j ~[\textup{mod}~2]=0$. We deduce that $2$ divides $(n_{i,j}+n_{j,i})$ for every $1 \leq i \neq j \leq n$. We obtain that :
\begin{align*}
\eps_{\SGIET}(f)&=\sum\limits_{i \neq j} n_{i,j} l_i \wedge l_j \\
&=\sum\limits_{i <j} (n_{i,j}-n_{j,i})l_i \wedge l_j \\
&=\sum\limits_{i <j} (n_{i,j}+n_{j,i})l_i \wedge l_j - 2 \sum\limits_{i <j} n_{j,i}l_i \wedge l_j
\end{align*}

We deduce that $2$ divides $\eps_{\SGIET}(f)$ and this gives the result.

\end{proof}

\bibliographystyle{abbrv}
\bibliography{biblio}
\end{document}